\documentclass[11pt,  reqno]{amsart}

\usepackage[T1]{fontenc}

 \usepackage[margin=1.6in]{geometry}

\usepackage{graphicx}
\usepackage{relsize}
\usepackage{t1enc}
\usepackage{latexsym}
\usepackage{amssymb}
\usepackage{graphicx}
\usepackage{amsmath}
\usepackage{amsthm}
\usepackage{amsfonts}
\usepackage{mathrsfs}
\usepackage[all]{xy}
\usepackage[british]{babel}
\usepackage[usenames]{xcolor}
\usepackage[hypertexnames=false,
    pdftex,
	pdfpagemode=UseNone,
	breaklinks=true,
	extension=pdf,
	colorlinks=true,
	linkcolor=blue,
	citecolor=blue,
	urlcolor=blue,
]{hyperref}

\newcommand{\cyclic}{\mathop{\kern0.9ex{{+}\kern-2.10ex\raise-0.20
      ex\hbox{\Large\hbox{$\circlearrowright$}}}}\limits}

\newtheoremstyle{daniel}{3.0mm}{2mm}{\itshape}{}{\bfseries}{.}{1.5mm}{}
\theoremstyle{daniel}
\newtheorem{thm}{Theorem}[section]
\newtheorem{prop}[thm]{Proposition}
\newtheorem{Defi}[thm]{Definition}
\newtheorem{lemma}[thm]{Lemma}
\newtheorem{cor}[thm]{Corollary}
\newtheorem{Exs}[thm]{Examples}
\newtheorem{Rems}[thm]{Remarks}

\newtheorem*{thm*}{Theorem}
\newtheorem*{cor*}{Corollary}
\newtheorem*{thm3.6}{Theorem 3.6}
\newtheorem*{thm4.3}{Theorem 4.3}

\newtheorem*{prop*}{Proposition}
\newtheorem*{Notation}{Notation}

\newtheorem{Cor}[thm]{Corollary}
\newtheorem{Prop}[thm]{Proposition}

\newtheorem{claim}[thm]{Claim}

\newtheorem{Rem}[thm]{Remark}

\newtheorem{Ex}[thm]{Example}

\newtheorem*{Setup}{Setup}

\newenvironment{rem}   {\begin{Rem}\em}{\end{Rem}}

\newenvironment{defi}  {\begin{Defi}\em}{\end{Defi}}

\newenvironment{setup}   {\begin{Setup}\em}{\end{Setup}}

\newcommand{\cA}{\Acal }

\def\cC{\mathcal{C}}

\def\cE{\mathcal E}
\def\cF{\mathcal F}

\def\cO{\mathcal O}
\def\cS{\mathcal S }
\def\cT{\mathcal T}

\newcommand{\C}{\mathbb{C}}

\newcommand{\N}{\mathbb{N}}
\renewcommand{\P}{\mathbb{P}}

\newcommand{\Z}{\mathbb{Z}}

\DeclareMathOperator\im{\mbox{Im}}

\DeclareMathOperator\NS{\mbox{NS}}

\DeclareMathOperator\Tors{\mbox{Tors}}

\newcommand{\Ibold}{{\bf I}}

\newcommand{\ebold}{{\bf e}}

\newcommand{\CBbb}{\mathbb C}

\newcommand{\PBbb}{\mathbb P}
\newcommand{\QBbb}{\mathbb Q}
\newcommand{\RBbb}{\mathbb R}

\newcommand{\ZBbb}{\mathbb Z}

\newcommand{\Acal}{\mathcal A}

\newcommand{\Ccal}{\mathcal C}

\newcommand{\Ecal}{\mathcal E}
\newcommand{\Fcal}{\mathcal F}
\newcommand{\Gcal}{\mathcal G}
\newcommand{\Hcal}{\mathcal H}
\newcommand{\Ical}{\mathcal I}
\newcommand{\Jcal}{\mathcal J}
\newcommand{\Kcal}{\mathcal K}

\newcommand{\Ncal}{\mathcal N}
\newcommand{\Ocal}{\mathcal O}

\newcommand{\Qcal}{\mathcal Q}
\newcommand{\Rcal}{\mathcal R}
\newcommand{\Scal}{\mathcal S}
\newcommand{\Tcal}{\mathcal T}
\newcommand{\Ucal}{\mathcal U}

\newcommand{\Zcal}{\mathcal Z}

\newcommand{\gfrak}{\mathfrak g}

\newcommand{\Cscr}{\mathscr C}

\newcommand{\Escr}{\mathscr E}

\newcommand{\Hscr}{\mathscr H}

\newcommand{\Uscr}{\mathscr U}

\newcommand{\SL}{\mathsf{SL}}

\newcommand{\GL}{\mathsf{GL}}

\newcommand{\SU}{\mathsf{SU}}

\newcommand{\SO}{\mathsf{SO}}

\DeclareMathOperator{\End}{End}

\DeclareMathOperator{\imag}{im}

\DeclareMathOperator{\rank}{rank}

\DeclareMathOperator{\vol}{vol}

\DeclareMathOperator{\coker}{coker}
\DeclareMathOperator{\codim}{codim}

\DeclareMathOperator{\tr}{tr}

\DeclareMathOperator{\sing}{sing}
\DeclareMathOperator{\gr}{gr}
\DeclareMathOperator{\Gr}{Gr}
\DeclareMathOperator{\CS}{CS}
\DeclareMathOperator{\supp}{supp}

\DeclareMathOperator{\Quot}{Quot}

\newcommand{\dbar}{\bar\partial}

\newcommand{\lra}{\longrightarrow}

\newcommand{\Pic}{{\rm Pic}}

\newcommand{\ch}{{\rm ch}}

\newcommand{\HYM}{\text{\rm\tiny HYM}}

\newcommand{\GMC}{\text{\rm\tiny GM}}
\newcommand{\Gss}{\text{\rm\tiny Gss}}
\newcommand{\muss}{\text{\rm\tiny $\mu$ss}}

\newcommand{\doublequotient}{\bigr/ \negthinspace\negthinspace \bigr/}

\def\ddual#1{#1^{\vee\vee}}

\newcommand{\isorightarrow}{\xrightarrow{
   \,\smash{\raisebox{-0.5ex}{\ensuremath{\sim}}}\,}}

\numberwithin{equation}{section}

\allowdisplaybreaks

\begin{document}
\title[Compactifications of the moduli space of HYM connections]{Complex algebraic compactifications of the moduli space of Hermitian-Yang-Mills connections on a projective manifold}

\author[Greb]{Daniel Greb}
\address{Daniel Greb\\Essener Seminar f\"ur Algebraische Geometrie und Arithmetik\\Fakult\"at f\"ur Ma\-the\-matik\\Universit\"at Duisburg--Essen\\
45112 Essen\\ Germany}
\email{daniel.greb@uni-due.de}
\urladdr{\href{https://www.esaga.uni-due.de/daniel.greb/}
{https://www.esaga.uni-due.de/daniel.greb/}}

\author[Sibley]{Benjamin Sibley}
\address{Benjamin Sibley, Universit\'e Libre de Bruxelles \\
Campus de la Plaine \\
Boulevard du Triomphe, ACC.2\\
1050 Bruxelles\\
Belgique}
\email{bsibley1060@gmail.com or 
benjamin.sibley@ulb.ac.be}

\author[Toma]{Matei Toma}
\address{Matei Toma, 
Universit\'e de Lorraine, CNRS, IECL, F-54000 Nancy, France}

\email{Matei.Toma@univ-lorraine.fr}
\urladdr{\href{http://www.iecl.univ-lorraine.fr/~Matei.Toma/}{http://www.iecl.univ-lorraine.fr/~Matei.Toma/}}

\author[Wentworth]{Richard Wentworth}
\address{Richard Wentworth, Department of Mathematics \\ University of Maryland \\ College Park, MD
20742, USA}
\email{raw@umd.edu}
\urladdr{\href{http://www.math.umd.edu/~raw/}{http://www.math.umd.edu/~raw/}}

\date{\today}

\keywords{Hermitian-Yang--Mills connections, stability, moduli of coherent sheaves, Kobayashi--Hitchin correspondence, Donaldson--Uhlenbeck compactification}
\subjclass[2010]{14D20, 14J60, 32G13, 53C07}

\begin{abstract}
In this paper we study the relationship between three compactifications of the moduli space of gauge equivalence classes of Hermitian-Yang-Mills connections on a fixed Hermitian vector bundle over a projective algebraic manifold of arbitrary dimension. Via the Donaldson-Uhlenbeck-Yau theorem, this space is analytically isomorphic to  the moduli space of stable holomorphic vector bundles, and as such it admits an algebraic compactification by 
Gieseker-Maruyama semistable torsion free sheaves. 
A  recent construction due to the first and third authors gives another  compactification as a moduli space of slope semistable sheaves. In the present article, following fundamental work of Tian generalising the analysis of Uhlenbeck and Donaldson in complex dimension two, we define a gauge theoretic compactification by adding certain gauge equivalence classes of ideal connections at the boundary. Extending work of Jun Li in the case of bundles on algebraic surfaces, we exhibit comparison maps from the sheaf theoretic compactifications and prove their continuity. The continuity, together with a delicate analysis of the fibres of the map from the moduli space of slope semistable sheaves allows us to endow the gauge theoretic compactification with the structure of a complex analytic space. 
\end{abstract}
\maketitle

\noindent

\section{Introduction}
The study of the Hermitian-Yang-Mills (HYM) equations on a K\"ahler manifold 
joins two rather different areas of research in differential geometry. On the one hand, the Donaldson-Uhlenbeck-Yau theorem \cite{Donaldson:85,Donaldson:87,UhlenbeckYau:86}, which relates   irreducible solutions of these equations to Mumford-Takemoto stable holomorphic vector bundles, 
 is an important example of a more general relationship between nonlinear geometric PDEs in complex geometry and algebraic geometric notions of stability coming from Geometric Invariant Theory (GIT). On the other hand, the HYM equations are an archetypal instance of equations appearing in mathematical gauge theory. Indeed, in the case of K\"ahler surfaces they correspond to the anti-self-duality equations that were so profitably exploited by Donaldson and others in the 1980s and 1990s. The purpose of this paper is to develop further one aspect of the link between these two points of view in the case of higher dimensional projective algebraic manifolds: namely, the precise relationship between
 degenerations of locally free sheaves and singular limits of HYM connections. 

Concretely, we are concerned with the moduli space $M_\HYM^\ast(E,h,a_J )$ of gauge equivalence classes of  smooth irreducible $h$-unitary connections that are
solutions to the HYM equations on a fixed hermitian vector bundle $(E,h)$ of rank $r$ on a  projective algebraic manifold $X\subset \PBbb^N$ of dimension $n$ that induce a fixed connection $a_J$ on $J=\det E$. 
If $M_\HYM^\ast(E,h,a_J)$ is nonempty it will generally fail to be compact. Both in differential and algebraic geometry, when one studies such noncompact moduli spaces it is often useful to construct a compactification by adding singular objects of one kind or another at the boundary. 

\subsection{Algebro-geometric modular compactifications} The first point of view mentioned at the beginning of the introduction suggests adjoining "singular" holomorphic vector bundles, or torsion free coherent sheaves, in order to compactify the moduli space $M^{s}(E, \Jcal)$ of slope stable holomorphic structures on $E$ with fixed determinant $\Jcal$, corresponding to $a_J$. In fact, it is  possible to implement this idea in  two  different ways, corresponding to two different natural restrictions imposed on the sheaves allowed. The most well-studied solution, going back to Gieseker and Maruyama \cite{Gieseker:77,Maruyama:76},  uses Gieseker-Maruyama (GM) semistable sheaves $\Ecal$ with the same topological invariants as $E$. This is a natural choice because all such sheaves are quotients  $\mathcal{H}\rightarrow \Ecal\rightarrow 0$ of a certain fixed vector bundle $\Hcal\to X$ (see \eqref{eqn:H}), and GM-semistability turns out to be the same as GIT stability for a particular linearised line bundle on the Quot scheme $\Quot(\Hcal, c(E))$ parametrising such quotients. A projective scheme $M^{\Gss}(E, \Jcal),$ containing  $M^{s}(E, \Jcal)$ as a Zariski open subset may then be constructed  as a GIT quotient, and taking the closure of $M^{s}(E, \Jcal)$ provides a compactification $\overline M^\GMC(E, \Jcal)$. 

It is also reasonable to ask if there is a moduli space which allows degeneration to torsion free sheaves which are merely slope semistable. This has been  analysed in the case of surfaces in  \cite{LePotier:92, Li:93} (see also \cite{HuybrechtsLehn:10}), and very recently in higher dimensions by the first and third authors in \cite{GrebToma:17}. All relevant slope semistable  sheaves are again quotients of  $\Hcal$, and there is furthermore a certain natural linearised line bundle  over the Quot scheme whose restriction to the open subset $R^\muss$ of slope semistable sheaves is equivariantly semiample for the case where $X$ is a surface. In dimensions greater than two, semiampleness of the analogous line bundle $\mathscr L_{n-1}$
 can be established on the weak normalisation
$(R^\muss)^{wn}$.\footnote{Recall that the \emph{weak
normalisation} of a  complex space $Z$ is a reduced complex
space $Z^{wn}$ on which the first Riemann Extension Theorem
holds, together with a homeomorphic, holomorphic map $Z^{wn}
\to Z$. Weak normalisation is a functor on the category of complex spaces, which by definition factors through the reduction functor.
}
We refer the reader to the introduction of \cite{GrebToma:17} for an explanation of why the passage to the weak normalisation is necessary for technical, and very likely also philosophical, reasons. Given this, one then defines the moduli space $M^{\muss}(E,\Jcal)$ of slope semistable sheaves with fixed determinant $\Jcal$, in imitation of the GIT construction, to be the projective spectrum of the ring of invariant sections of powers of $\mathscr L_{n-1}$. An application of Langton's Theorem implies that this defines a complete, hence projective, variety. Considering the separation properties of sections, one then shows that $M^{s}(E,\Jcal)^{wn}$ embeds as a Zariski open subset of $M^{\muss}(E, \Jcal)$. By taking its closure inside $M^\muss(E,\Jcal)$ we thus obtain a second compactification $\overline M^\mu(E,\Jcal)$. 

Points of $\overline M^{\GMC}(E,\Jcal)$ and $\overline M^{\mu}(E,\Jcal)$ may be represented by torsion free GM- or slope semistable sheaves, respectively, but it is important to understand their isomorphism classes; that is, when do two sheaves represent the same point in the moduli space? 
For $\overline M^{\GMC}(E, \Jcal)$, this is given by the notion of $s$-equivalence. For $\overline M^{\mu}(E, \Jcal)$, in the case of algebraic surfaces the correct identifications were found by Jun Li in \cite{Li:93}.
If $\dim{X}\geq 3$, the issue is more subtle and was analysed in \cite{GrebToma:17}. The characterisation of  points in these spaces, together with the fact that GM-semistability implies slope semistability, leads to the definition 
of  a birational morphism  $\Xi: \overline{M}^\GMC(E, \Jcal)^{wn}\to \overline M^{\mu}(E, \Jcal)$ that extends the identity on $M^s (E, \Jcal)^{wn}$, but usually collapses points at the boundary. A more detailed summary of the theory described in the preceding paragraphs is given in Section \ref{sec:moduli} below.

\subsection{Gauge theoretic compactification in higher dimensions}
From the gauge theoretic point view, a natural candidate for a compactification would be to add  "singular" HYM connections. These may be understood from the foundational results of Uhlenbeck \cite{Uhlenbeck:82b,Uhlenbeck:82a,UhlenbeckPreprint} which provide weak convergence along a subsequence assuming  appropriate bounds on curvature (see also \cite{Nakajima:88,Sedlacek:82}).  Essential to the present paper is the work of Tian \cite{Tian:00}, 
who proves
that a sequence of smooth HYM connections $\{A_i\}$ converges (subsequentially)  to a limiting HYM connection $A_\infty$ smoothly  away from a  holomorphic subvariety $S\subset X$ of codimension at least two. Moreover, using the weak limit of the Yang-Mills energy density, 
one can assign a positive integer multiplicity to the irreducible components of $S$ of pure codimension two. 
In this way the limit  produces an effective holomorphic cycle $\Ccal$. The complement in $S$ of the support of $\Ccal$ consists of nonremovable singularities of the limiting connection $A_\infty$.
Hence,  there is a decomposition $S=S_{b}\cup S(A_{\infty })$, where $S_{b}$ is a union of pure codimension two subvarieties, and $A_\infty$ extends as a smooth connection locally on the complement of $S(A_\infty)$ (\cite{Tian:00, TaoTian:04}; Theorems \ref{thm:tian-singset} and \ref{thm:tao-tian} below). Furthermore, by results of Bando and Siu (\cite{BandoSiu:94}; Theorem \ref{thm:bando-siu} below), the holomorphic bundle defined by $A_\infty$ away from $S$ actually extends to all of $X$ as a slope polystable reflexive sheaf $\Ecal_\infty$, and $S(A_{\infty })=\sing(\mathcal{E}_{\infty })$. Note that we may have $S_{b}\cap S(A_{\infty })\neq \emptyset$. 

This suggests that ideal boundary points should be represented by triples of the form $(A,\Ccal,S(A))$, or equivalently pairs $(\Ecal,\Ccal)$, where $A$ is a finite energy HYM connection defined on the complement of $|\Ccal|\cup S(A)$, $\Ccal$ is an effective codimension two holomorphic cycle, $S(A)$ is a holomorphic subvariety of codimension at least 3 which is exactly the set of nonremovable singularities of $A$, and $\Ecal$ is a slope polystable reflexive sheaf with $S(A)=\sing(\Ecal)$, subject to the cohomological 
condition $\ch_2(E)=\ch_2(\Ecal)+[\Ccal]$. 

 This description of the "ideal connections" appearing at infinity should be compared to the definition of the  Donaldson-Uhlenbeck compactification (see \cite{Donaldson:86}, \cite{DonaldsonKronheimer:90}, and the summary provided in \cite[Sect.\ 4]{Li:93}) used in the theory of instantons on smooth four dimensional real manifolds. In this latter situation, $S$ is a finite set of points, and the connection $A$ extends smoothly to the entire manifold ($S(A)=\emptyset$), albeit on a bundle with modified second Chern class.
  Originally, the theory is developed for either $\SU(2)$ or $\SO(3)$ bundles, and the second Chern number is then essentially the Yang-Mills energy; in particular, it is positive. The cohomological relation stated above then gives  a uniform bound on the total length of the cycles (the sum of the number of points with multiplicity) and this is enough to obtain convergence of the singular sets to a limiting cycle. One may then apply Uhlenbeck compactness on the complement of (the support of) this limiting cycle to obtain a limiting ideal connection, and the set of all ideal connections is thereby (sequentially) compact. Taking the closure inside this compact space yields a gauge theoretic compactification.

In order to obtain an analogue of the Donaldson-Uhlenbeck compactification in higher dimensions, we mimic the procedure described in the previous paragraph using the approach of \cite{Tian:00}. The appearance of the nonremovable singularity sets $S(A)$ is a key new feature in higher dimensions. For a sequence $(A_i,\Ccal_i,S(A_i))$, the Bogomolov inequality and the relations $\ch_2(E)=\ch_2(A_i)+[\Ccal_i]$, imply a uniform bound on the degrees of the $\Ccal_i$, and therefore also subsequential convergence of the $\Ccal_i$ as currents to some codimension two cycle $\Ccal^{\prime}_{\infty}$. Given a general sequence of ideal connections, however, it would  seem there is no a priori control on the singular sets $S(A_i)$ beyond the fact that they converge in the Hausdorff topology to some compact set, the Hausdorff dimension of which is in principle only bounded in real codimension four (cf.\ \cite[Prop.\ 3.3]{TianYang:02}).
Although Uhlenbeck compactness applies on the complement of this closed set  to obtain a limiting connection $A_\infty$, this will not result in an ideal connection in our sense unless we can pin down a more precise structure for the singular set. One problem is the extension of the codimension two subvarieties of the bubbling locus to all of $X$.
A second crucial issue  is the behaviour of the sets 
 $S(A_i)$. In particular, since the lower dimensional strata $S(A_\infty)$ associated to an ideal connection should be precisely the codimension at least three subvariety given by the singularities of a reflexive sheaf, we wish first of all to rule out the possibility that the $S(A_i)$ accumulate along a set of strictly larger Hausdorff dimension. 

Subsequential convergence 
of the $S(A_i)$ in the cycle space is 
 implied by a uniform bound on the degrees (or volumes) of the
holomorphic subvarieties given by the irreducible components
of the $S(A_i)$ in each dimension. The problem that appears at
this point is that it is by no means obvious that the set
$\widehat M_{\HYM}(E,h,a_J)$ of ideal connections, or indeed
even the set of limits of smooth, irreducible, HYM connections
satisfies this property. This is where we are forced to use
the projectivity assumption on $X$. Thinking of points of
$\widehat M_{\HYM}(E,h, a_J)$ as a set of pairs
$(\Ecal,\Ccal)$ the above property for  $S(A)=\sing(\Ecal)$
would follow from the statement that the family of all
polystable reflexive sheaves $\Ecal$ arising as the sheaf
component of an ideal connection is  bounded. Here we may rely
on fundamental results of \cite{Grothendieck:61} controlling
the Hilbert polynomials of such a family of sheaves; namely,
only finitely many Hilbert polynomials appear as we range over
a bounded family. This fact, together with the sheaf theoretic
description of the singular sets, is enough to give a uniform
degree bound for sequences in $\widehat M_{\HYM}(E,h, a_J)$.
Therefore the question is reduced to whether or not the
sheaves $\Ecal$ form a bounded family. This is answered affirmatively, using \cite[Main Theorem]{Maruyama:81} (see Lemma \ref{lem:bounded-admissible}).

Applying the compactness results of Uhlenbeck and Tian on the
complement of $|\Ccal^{\prime}_{\infty}|\cup
S^{\prime}_{\infty}$, we obtain sequential compactness for the
space $\widehat M_{\HYM}(E,h, a_J)$ (see Theorem
\ref{thm:ideal-convergence}). A crucial aspect of this result,
also essential for the algebraic arguments of Section
\ref{sect:equivalence_relation},
 is the fact that the limiting set $S^{\prime}_{\infty}$ is actually contained in  $|\Ccal_{\infty}|\cup S(A_\infty)$, where $\Ccal_{\infty}$ is a cycle obtained from $\Ccal^{\prime}_{\infty}$ and the cycle produced by bubbling of the connections. 
  Interestingly, the proof of this fact relies on $\dbar$-methods and relatively recent 
  estimates of  Chakrabarti-Shaw \cite{ChakrabartiShaw:11}.

With these results we may proceed to define the compactification \break $\overline M_{\HYM}(E,h, a_J)$ to be the set of Uhlenbeck limits in  $\widehat M_{\HYM}(E,h, a_J)$
of smooth, irreducible HYM connections.
 A diagonalisation argument shows that this space is in fact closed, or in other words any limit of a sequence of ideal connections in $\overline M_{\HYM}(E,h, a_J)$ again arises as an Uhlenbeck limit of connections in $M_{\HYM}^\ast(E,h, a_J)$. From the definition of convergence emerges  a natural topology on $\widehat M_{\HYM}(E,h, a_J)$, and hence also for $\overline M_{\HYM}(E,h, a_J)$, resulting in the following theorem.

\begin{thm}[{\sc Gauge theoretic compactification}] \label{thm:mainI}
$\overline M_{\HYM}(E,h, a_J)$ is a sequentially  compact Hausdorff space containing $ M_{\HYM}^\ast(E,h, a_J)$ as a dense open subspace. 
\end{thm}

We should point out here that after the proof of Theorem \ref{thm:mainII} below we will   conclude that $\overline M_{\HYM}(E,h, a_J)$ is actually compact, and not just sequentially so.

\subsection{A complex structure on the HYM compactification}
It is natural to ask how much structure the topological space $\overline M_{\HYM}(E,h, a_J)$  inherits from the projective  manifold $X$. As mentioned in the first paragraph, the Donaldson-Uhlenbeck-Yau Theorem 
 gives a bijection $$\Phi :M^{s}(E, \Jcal)\longrightarrow M_{\HYM}^\ast(E,h, a_J)\ .$$
  In fact, much more  is true: the space $M^s(E, \Jcal)$ has the structure of a possibly nonreduced complex analytic space, obtained for example by analytification from the scheme structure of the Gieseker-Maruyama moduli space,\footnote{This complex structure coincides with the one induced by the moduli space of simple holomorpic bundles and also with the one given by the moduli space of simple holomorphic structures (\cite[Chapter VII]{Kobayashi:87} and \cite{Miyajima:89}).} while $M_{\HYM}^\ast(E,h, a_J)$ has a natural  real-analytic space structure. With respect to these, the map $\Phi$ is real-analytic, and may hence be used to transfer the complex structure from $M^{s}(E,\Jcal)$ to $M_{\HYM}^\ast(E,h, a_J)$ (see \cite{FujikiSchumacher:87}, \cite{LubkeOkonek:87}, \cite{Miyajima:89}, and especially \cite{LubkeTeleman:95} for this point of view).
The question then arises as to what relationship the algebraic compactifications $\overline M^\GMC(E,\Jcal)$ and $\overline M^\mu(E,\Jcal)$ might have to  $\overline M_{\HYM}(E,h, a_J)$, and whether or not the gauge theoretic compactification $\overline M_{\HYM}(E,h, a_J)$ also has the structure of a complex or even (projective) algebraic space.

In the case of projective surfaces, these two questions have been answered by Jun Li in \cite{Li:93} (see also \cite{Donaldson:90}). In this case, a complex space closely connected with $\overline M^\mu(E, \Jcal)$ turns out to be  homeomorphic to the Uhlenbeck-Donaldson compactification (see Remark \ref{rem:identify_in_surface_case} for the details), and 
Li realizes this homeomorphism  via an algebraic map
$\overline M^\GMC(E, \Jcal)\to \overline M^\mu(E,\Jcal)$. We also note  that in this context Morgan  \cite{Morgan:93}  constructed a continuous map from $M^{\Gss}(E, \Jcal)$ to $ \overline M_{\HYM}(E,h, a_J)$ (see also \cite{BuchdahlTelemanToma:17} for a generalisation to non-K\"ahlerian surfaces). 

This brings us to the second main result of the paper, which provides a generalisation of Jun Li's theorem to higher dimensions.
\begin{thm}[{\sc Complex structure}]\label{thm:mainII} 
The space $\overline M_{\HYM}(E,h, a_J)$ can be endowed with the the structure of a reduced complex analytic space such that the natural map $M_{\HYM}^\ast(E,h, a_J)^{wn} \hookrightarrow \overline M_{\HYM}(E,h, a_J)$ is holomorphic and embeds $M_{\HYM}^\ast(E,h, a_J)^{wn}$ as a Zariski open and dense subset. There are moreover natural surjective holomorphic maps 
$$\overline M^{\GMC}(E,\Jcal)^{wn}\lra\overline M^\mu(E, \Jcal)
\stackrel{\overline{\Phi}}{\xrightarrow{\hspace*{.75cm}}} 
\overline M_{\HYM}(E,h, a_J)\ ,$$ where the first map is birational, and the second map is finite and extends the  map 
$$\Phi^{wn}: M^s(E, \Jcal)^{wn} \stackrel{\cong}{\xrightarrow{\hspace*{.75cm}}} M_{\HYM}^\ast(E,h, a_J)^{wn}$$
 to the respective compactifications. 
\end{thm}

Once the complex structure on $\overline M_{\HYM}(E,h, a_J)$ has been constructed, the fact that $\overline{M}^\mu(E, \Jcal)$ is projective, and hence algebraic, implies by standard arguments that $\overline M_{\HYM}(E,h, a_J)$ is actually likewise algebraic (see the end of Section\ \ref{sect:equivalence_relation}).

\begin{cor}[{\sc Algebraicity of complex structure}]\label{cor:algebraic}
 The complex space $\overline M_{\HYM}(E,h, a_J)$ constructed in Theorem\ \ref{thm:mainII} is the analytification of an algebraic space, which we also denote by $\overline M_{\HYM}(E,h, a_J)$, such that
  $$\overline{M}^\mu(E,\Jcal) \stackrel{\overline{\Phi}}{\xrightarrow{\hspace*{.75cm}}} \overline M_{\HYM}(E,h, a_J)$$ is a morphism of algebraic spaces which embeds $M^s(E, \Jcal)^{wn} \subset \overline{M}^\mu(E, \Jcal)$ as a Zariski open subset into $\overline M_\HYM(E,h,a_J)$.
\end{cor}

\subsection{The sheaf-theory-to-gauge-theory comparison map}
Let us give a rough idea of the
 proof of Theorem\ \ref{thm:mainII}, which can be found in Section\ \ref{sect:equivalence_relation}.
 Broadly speaking, 
 we follow the strategy of Li in the case of surfaces.
However,  as we shall see,  the relationship of $\overline M_\HYM(E,h,a_J)$ to $\overline M^{\mu}(E, \Jcal)$ is less straightforward if $\dim X\geq 3$. We begin with a comparison map 
\begin{equation*}
\overline{\Phi }:\overline{M}^\mu(E,\Jcal)\lra \widehat M_{\HYM}(E,h, a_J)\ ,
\end{equation*}
defined by sending $[\Ecal]\longmapsto (\ddual{(\Gr\Ecal)},\mathcal{C}_{\Ecal})$, where $\Ecal$ is  a slope semistable sheaf $\Ecal$, $\Gr\Ecal$ is
a torsion free sheaf associated to a Jordan-H\"older filtration of $\Ecal$, $\ddual{(\Gr\Ecal)}$ is its sheaf theoretic double dual,  and $\Ccal_{\Ecal}$ is the support cycle of the torsion sheaf  $\ddual{(\Gr\Ecal)}/\Gr\Ecal$. That this map is well defined is a consequence of the geometry of $M^{\muss}(E, \Jcal)$ (see Section \ref{sec:moduli}). In fact, as a direct result of Propositions  \ref{prop:cont1} and  \ref{prop:singset-limit} in Section \ref{sec:comparison}, we will see that the image of this map actually lies in $\overline M_{\HYM}(E,h, a_J)$ (see Lemma \ref{lem:images}).

The first point is that $\overline\Phi$ is continuous. We prove this by showing that it respects limits in the respective spaces.
The existence of subsequential limits in $\overline M_{\HYM}(E,h, a_J)$ is guaranteed by Theorem \ref{thm:mainI}.
 We first prove the continuity for degenerations of the form $\Ecal_{i}\rightarrow\Ecal_\infty$, where the $\Ecal_i$ are slope stable and locally free and $\Ecal_\infty$ is slope semistable. The result follows from continuity on sequences of this type, together with a diagonalisation argument. Since a point in $\overline M_{\HYM}(E,h, a_J)$ is a pair $(\Ecal,\Ccal)$, the proof reduces to two natural continuity-type statements; one for the limiting reflexive sheaf, and one for the holomorphic cycle.

In the sheaf component, the key point is that in the projective setting any Uhlenbeck limit $\Ecal_\infty$ of a sequence of smooth, irreducible, HYM connections can be identified with the double dual of a polystable quotient $\mathcal{H}\rightarrow \widehat{\mathcal{E}}_{\infty }\rightarrow 0$ in $\Quot(\mathcal{H},c(E))$
  (see Section~\ref{sec:limits1}). The continuity essentially follows by applying this result to the HYM connections in the complex gauge orbit of the $\Ecal_i$. Here we follow  the argument of \cite{Li:93}, with some important modifications. 
The idea, which is reminiscent of  methods used in \cite{Donaldson:85, Morgan:93, DaskalWentworth:04, Sibley:15, BuchdahlTelemanToma:17}, is to find a subsequential limit of an $L^2$-orthonormal basis of sections of $\Ecal_i$ twisted by a high power of the polarisation.  
The associated maps $\ebold _{i}^{(k)}:\mathcal{H}(k)\rightarrow \mathcal{E}_{i}(k)$
give rise to a subsheaf $\widetilde \Ecal_\infty\subset \Ecal_\infty$ of full rank. On the other hand, one can take a limit of the sequence in $\Quot(\Hcal, c(E))$, and we show that $\widetilde \Ecal_\infty$ coincides with this limit.

For continuity in the cycle component, our proof is rather different to that of \cite{Li:93}, and  we adapt instead an argument of \cite{SibleyWentworth:15}. The main part of the proof is to show the equality of the multiplicity of an irreducible component of the support cycle of the quotient $\ddual{(\Gr\Ecal_\infty)}/\Gr\Ecal_\infty$ with the multiplicity defined by the limit of the Yang-Mills energy densities for $A_{i}$ 
(see Proposition \ref{prop:multiplicity}). We accomplish this by means of a slicing argument from \cite{SibleyWentworth:15} together with a slight modification of the Bott-Chern formula proven in that reference. The delicate point here is that in order to isolate the irreducible component in question, we must work on a modification of $X$ by a sequence of blowups that achieves certain useful properties. The argument in \cite{SibleyWentworth:15} is applicable to the support cycle of the quotient of the associated graded of a holomorphic vector bundle and its double dual (in that case this double dual arises as an Uhlenbeck limit). The main trick is then to notice that by virtue of the preceding paragraph, the Uhlenbeck limit $\Ecal_\infty$ is the double dual of a quotient of $\Hcal$, and so the cycle in question is the one associated to a kind of graded object  to which the singular Bott-Chern argument  applies. With this in hand, the proof proceeds in a manner similar to that of \cite{SibleyWentworth:15} (see Section \ref{sec:agree-cycle-component}).

\subsection{The induced equivalence relation} The geometric description of the points of the space $M^{\muss}(E)$ given in \cite{GrebToma:17} (see also Section \ref{sec:moduli}) implies that the map $\overline{\Phi}$ is in fact finite-to-one. A major difference to the case of surfaces is that in general the fibres may potentially consist of more than one point. We define a finite equivalence relation on $\overline M^\mu(E,\Jcal)$ by considering two points to be equivalent precisely when they lie in the same fibre. Once the continuity of  $\overline{\Phi}$ has been established, the quotient of $\overline M^\mu(E, \Jcal)$ by this equivalence relation is then homeomorphic to $\overline M_{\HYM}(E,h, a_J)$. In order to endow the quotient with a complex structure, we need to know that our equivalence relation has sufficiently nice properties so that the complex structure on $\overline M^\mu(E, \Jcal)$ descends to the quotient.

Fortunately, there is a  criterion due to H. Cartan for when the quotient of a complex space is itself complex; namely, that points can be separated by local holomorphic functions that collapse the relation. We show that this indeed holds for the relation on $\overline M^\mu(E,\Jcal)$ defined above by a more detailed study of sections of the line bundle $\mathscr L_{n-1}$ used to define $M^{\muss}(E)$. The sheaf theoretic moduli space was defined using all equivariant sections of $\mathscr L_{n-1}$, which give a globally defined map to projective space, and whose image is by definition $M^{\muss}(E, \Jcal)$ . It is possible that we thereby separate too many sheaves and so obtain a space that is larger than $\overline M_{\HYM}(E,h, a_J)$. We rectify this by considering instead only those sections which are "lifted" (in a certain precise sense) from certain well chosen complete intersection curves inside of a fixed complete intersection surface. The associated linear systems $W$ give rise to locally defined morphisms $\nu :U_{S}\longrightarrow \mathbb{P}(W^{\ast })$ with $U_{S}\subset \overline{M}^{\mu }(E, \Jcal)$ open. Taking products 
of these maps for finitely many choices of surfaces, we obtain maps defined on certain open neighbourhoods of a given point of $\overline M^\mu(E, \Jcal)$ that identify precisely the fibres of $\overline{\Phi}$. 

The chief difficulty in implementing this strategy is verifying first of all that the morphisms so obtained are actually constant on the fibres of $\overline{\Phi}$ and secondly that by restricting to sections arising from these fixed linear subsystems, we continue to separate sheaves not lying in the same fibre of $\overline{\Phi}$. The first point is established by showing that the image of a sheaf under the map $\nu$ is determined by the image under the corresponding (global) map on the corresponding moduli space on the surface $S$ of the restriction of its associated ideal connection to $S$ (see Proposition \ref{prop:nonseparation}). Here we appeal to Jun Li's result mentioned above, that on a surface the algebraic moduli space is the same as the space of ideal connections. The collapsing then follows, since if two sheaves give the same ideal connection in $\overline M_{\HYM}(E,h, a_J)$, their restrictions to the surface are still equal in the moduli space on $S$. The second point is verified by showing that any ideal connection is completely determined by its restriction to finitely many such complete intersection surfaces $S$ (see Lemma \ref{lemma: cycle separation}).  

\subsection{Open questions and further directions}
The proof of Theorem \ref{thm:mainII} outlined above leaves unanswered several important questions. 
First, one would like to know whether or not the strong form of Li's theorem holds in higher dimensions; that is, whether or not $\overline M_{\HYM}(E,h, a_J)$ is a scheme or even projective. In contrast to the surface case this does not follow immediately from the method of proof, since  $\overline{\Phi}$ is not bijective. There is no good general theory for deciding whether the quotient of a scheme by a finite equivalence relation is again a scheme, even when the quotient is an algebraic space (see \cite{KollarFiniteEquiv} for a thorough discussion).

Second, it would be interesting to find examples where the
finite-to-one property described here actually occurs. A
related question  is whether the gauge theoretic
compactification can be realized as the  completion of \break
$ M_{\HYM}^\ast(E,h, a_J)$ for an appropriate metric. In
\cite{Donaldson:89}, Donaldson proves this for the moduli
space of anti-self-dual equations on a compact four manifold,
but  generalising the result to higher dimensions  requires
more information about the structure of a neighbourhood of an
ideal HYM connection. Note that there has been recent progress
in this direction (cf.\
\cite{Yang:03,Naber:16,ChenSun:18a,ChenSun:18b, ChenSun:19}).

Finally, we remark that, via the use of boundedness results, projectivity of $X$ is essential to the construction of the compactification  $\overline M_{\HYM}(E,h, a_J)$ and its complex structure that is presented here. In principle, one would like to define a gauge theoretic compactification in various other scenarios along the lines of the programme in \cite{DonaldsonSegal:11}, including HYM connections on nonprojective K\"ahler manifolds, where the tools used in this paper are
  not available. This seems to require at the very least 
  some kind of analytic control on the higher codimension
   singular sets $S(A_i)$. 

\subsection{Organisation of the paper and conventions}
In Section \ref{sec:background}
 we give basic background about the moduli spaces, coherent
sheaves, the Quot scheme, and the compactifications using
algebraic geometry. The whole of Section
\ref{sec:gauge-theory}
 is devoted to the proof of Theorem \ref{thm:mainI}; namely,
the construction of \break $\overline M_{\HYM}(E,h, a_J)$ and
the proof of compactness. In Section \ref{sec:comparison} we
prove continuity of $\overline{\Phi}$. In Section
\ref{sect:equivalence_relation}
 we analyse the fibres of $\overline{\Phi}$ in detail and prove Theorem \ref{thm:mainII}. For more details, see the table of contents below.

 If not explicitly stated otherwise, $X$ will always denote a smooth projective variety over the complex numbers $\mathbb{C}$. For the purposes of a clear notation, $C^\infty$ bundles will be denoted with 
straight characters $E$, etc., whereas
coherent sheaves will be denoted in script, e.g.\ $\Ecal\to X$. We shall confuse the notation of a holomorphic vector bundle with its sheaf of germs of holomorphic sections.
  When the distinction becomes necessary, we will generally
use $\Fcal$ for elements of Quot schemes, and $\Ecal$ for a
holomorphic bundle associated to a $\dbar$-operator on $E$
giving an HYM connection, or to an Uhlenbeck limit of these,
or more generally to the reflexive sheaf arising from an ideal
connection (see Sections \ref{sec:background} and
\ref{sec:gauge-theory}). 
We use mathscript letters, e.g.\ $\mathscr F$, to denote families of sheaves over a parameter space. If $\mathcal{L} \to X$ is a holomorphic line bundle, we will denote by $|\mathcal{L}|$ the associated linear system. If $\mathcal{O}_X(1)$ is the very ample line bundle on $X\subset\PBbb^N$, and $S$ is a subvariety of $X$, we will denote the restriction of $\mathcal{O}_X(1)$ to $S$ by $\mathcal{O}_S(1)$. Analytic subsets are also called holomorphic subvarieties; in particular, the latter are not assumed to be irreducible. Dimensions and codimensions $\dim$ and $\codim$ are always  over the complex numbers $\CBbb$ unless otherwise indicated.

\bigskip

\par\noindent{\bf Acknowledgments.} 
D.G.'s research has been supported by the  DFG-Colla\-bo\-ra\-tive Research Center SFB/TRR 45 "Periods, moduli spaces and arithmetic of algebraic varieties" (subproject M08-10 "Moduli of vector bundles on higher-dimensional varieties"). 
B.S. and M.T. are grateful for the hospitality of the Max Planck Institut f\"ur Mathematik in Bonn and the Department of Mathematics at the University of Maryland, where a portion of this work was completed. 
B.S. also received support from the Simons Center for Geometry
and Physics when he was a postdoc there, and
 from the PAI grant "DyGest" when he was a postdoc at the Universit\'e Libre de Bruxelles. R.W.'s research was supported by grants from the National Science Foundation. 

Finally, the authors warmly thank Jun Li, Mei-Chi Shaw, Carlos Simpson, and Katrin Wehrheim for helpful comments and suggestions concerning this research. 
We are grateful to Xuemiao Chen and Song Sun for pointing out an error in the first version of this paper. We also thank the referee for detailed comments, which have helped improve the exposition.

\vfill
\newpage 

\setcounter{tocdepth}{2}

{\small\tableofcontents}

\newpage


\section{HYM connections, reflexive sheaves, and their moduli}
\label{sec:background}
In this preliminary section we fix basic notation and provide background material.
We give precise definitions of the moduli spaces discussed in the Introduction.
We collect the necessary lemmas on compact analytic cycles on (projective) K\"ahler manifolds. Finally, we formulate a general construction used to produce line bundles on base spaces of flat families, and we go on to define the two algebraic compactifications referred to previously.

\subsection{Moduli spaces} \label{sec:prelim}
Let $X\subset \PBbb^N$ be as in the Introduction. 
Throughout this section and for the rest of the paper, 
$(E,h)\rightarrow X$ will be a $C^\infty$ complex vector bundle of rank $r$ with hermitian metric $h$.
Recall that a $\dbar$-operator $\dbar_E: \Omega^0(X,E)\to \Omega^{0,1}(X, E)$
is \emph{integrable} if   $\dbar_E\wedge \dbar_E=0$.
In this case, the Newlander-Nirenberg theorem guarantees the existence of local frames annihilated by $\dbar_E$, and $\dbar_E$ therefore defines a holomorphic structure on $E$.  Conversely, every holomorphic structure on $E$ defines a unique $\dbar$-operator via the Leibniz formula:
$
\dbar_E(f\cdot s)=\dbar f\otimes s
$,
for holomorphic sections $s$ and functions $f$.
We sometimes specify this relationship by $\Ecal=(E,\dbar_E)$. 

In the presence of the metric $h$, a $\dbar$-operator defines a unique \emph{Chern connection} $A=(\dbar_E, h)$ that is compatible with $\dbar_E$ in the sense that $\dbar_A := (d_A)^{0,1}=\dbar_E$, and $A$ is \emph{unitary}, that is,
$$
d\langle s_1,s_2\rangle_h= \langle d_As_1,s_2\rangle_h+\langle s_1,d_A s_2\rangle_h\ ,
$$ 
for any local smooth sections $s_1, s_2$ of $E$.  
We shall say that a unitary connection $A$ is integrable if $\dbar_A$ is an integrable $\dbar$-operator.  

Let $\det E=J$ and $\det h=h_J$ for a hermitian line bundle
 $J\to X$. We suppose $J$ carries an integrable $\dbar$-operator $\dbar_J$
with Chern connection $a_J=(\dbar_J, h_J)$, and  we denote the associated holomorphic bundle by $\Jcal:=(J,\dbar_J)$. 
 As stated in the Introduction, we will always work with the space
 $\Acal_{hol}(E, \dbar_J)$  of integrable $\dbar$-operators on $E$ which induce the fixed operator $\dbar_J$ on $J$.

For the purposes of this paper, it will be convenient to renormalise the restriction of the Fubini-Study K\"ahler form on $\PBbb^N$ to give a K\"ahler metric $\omega$ on $X$ with $\vol(X,\omega)=2\pi$. 
Let $\Ocal_X(1)$ be the restriction of the hyperplane line on $\PBbb^N$ to $X$.
We will denote the underlying $C^\infty$ line bundle of $\Ocal_X(1)$ by $L$, it's $\dbar$-operator by $\dbar_L$, and by $h_L$ the hermitian metric on $L$ whose
 Chern connection $a_L=(\dbar_L, h_L)$ has curvature equal to $-2\pi i \lambda\cdot\omega$.
Here, $c_1(\Ocal_X(1))=\lambda \cdot\omega$, for some constant $\lambda>0$. We will mostly  omit $\omega$ from the notation, but we note here that all of the constructions depend on the K\"ahler class in an interesting and important manner.

 Let $\Acal(E,h, a_J)$ denote the infinite dimensional space of $h$-unitary  connections on $E$ inducing the connection $a_J$ on $J$.  This is an affine space modelled on 
 $\Omega^1(X, \gfrak_E)$, where $\gfrak_E$ denotes the bundle of traceless skew-hermitian endomorphisms of $E$. 
For $A\in \Acal(E,h, a_J)$, the curvature $F_A=d_A\wedge d_A\in \Omega^2(X,\gfrak_E)$.  
The space of integrable connections will be denoted $\Acal^{1,1}(E,h,a_J)\subset \Acal(E,h,a_J)$. Then $A\in \Acal^{1,1}(E,h,a_J)$ if and only if 
$F_A$ is of type $(1,1)$, and the Chern connection construction yields the identification  $\Acal^{1,1}(E,h, a_J)\simeq\Acal_{hol}(E, \dbar_J)$.

Choose local holomorphic coordinates $\{z^i\}$ on $X$ and write $$\omega=\sqrt{-1}\sum g_{i\bar j}dz^i\wedge d\bar z^j\ .$$
Given $A\in \Acal^{1,1}(E,h)$, we have an expression
 $F_A=\sum_{i,j} F_{i\bar{j}}dz^{i}\wedge d\bar{z}^{j}$.  If $\Lambda$ is the adjoint of
the Lefschetz operator defined by wedging with $\omega$, then the hermitian endomorphism
$$
\sqrt{-1}\Lambda F_A =\sum_{i,j} g^{i\bar j}F_{i\bar j}\in \Omega^0(X,\sqrt{-1}\gfrak_E)\ ,
$$
is
called the \emph{Hermitian-Einstein tensor} (abbreviated \emph{HE tensor}).
 Note that $\Lambda$  is characterized
by the property that
\begin{equation} \label{eqn:lefschetz}
\left( \Lambda \Omega \right) \omega ^{n}=n(\Omega\wedge \omega ^{n-1})\ .
\end{equation}
for any $(1,1)$ form $\Omega$.

The integrable connections important for this paper are those for which the HE tensor is a constant multiple of the identity.

\begin{defi}
A \emph{Hermitian-Yang-Mills $($or HYM$)$ connection}  is an integrable, unitary connection  $A\in \Acal^{1,1}(E,h, a_J)$ satisfying
\begin{equation} \label{eqn:hym}
\sqrt{-1}\Lambda F_A=\mu\cdot {\bf I}_{E}\ ,
\end{equation}
for some $\mu\in \RBbb$.  The subspace of HYM-connections will be denoted: 
$$\Acal_{\HYM}(E,h, a_J)\subset \Acal^{1,1}(E,h, a_J)\ .$$
\end{defi}

The group of unitary gauge transformations is defined by 
$$
\Gcal(E,h)=\left\{ g\in \Omega^0(X,\End E) \mid gg^\ast={\bf I}_E\ ,\ (\det g)(x)=1\ ,\ \forall x\in X\right\} \ .
$$
Then  $\Gcal(E,h)$ acts on $ \Acal^{1,1}(E,h, a_J)$ (on the right) by conjugation: 
$$d_{g(A)}=g^{-1}\circ d_A\circ g\ ,$$
 and this
induces an action on the curvature given by $ F_{g(A)}=g^{-1}\circ
F_{A}\circ g$. In particular, the subspaces 
$$\Acal_{\HYM}^{\ast}(E,h, a_J)\subset \Acal_{\HYM}(E,h, a_J)\subset \Acal^{1,1}(E,h,a_J)$$
 of
(irreducible) Hermitian-Yang-Mills connections are preserved by $\Gcal(E,h)$.

\begin{defi}
The \emph{moduli spaces} 
of HYM connections and irreducible 
HYM connections connections on $E$ are the sets
\begin{align*}
M_{\HYM}(E,h,a_J)&=\Acal_{\HYM}(E,h,a_J)/\Gcal(E,h) \ ,\\
M_{\HYM}^\ast(E,h,a_J)&=\Acal_{\HYM}^\ast(E,h,a_J)/\Gcal(E,h)\ , \text{ respectively.}
\end{align*}
\end{defi}

Note that by taking the trace of both sides of \eqref{eqn:hym}, integrating over $X$, and using \eqref{eqn:lefschetz} and the normalisation of the volume, one
sees that the constant $\mu$ must be given explicitly by
\begin{equation*} 
\mu =\frac{2\pi }{\vol(X)}\mu(E)= \mu(E)\ ,
\end{equation*}
where the \emph{degree}, $\deg(E)$, and  \emph{slope}, $\mu(E)$, are defined by
\begin{equation} \label{eqn:slope}
 \deg(E) := {\int_{X}c_{1}(E)\wedge \frac{\omega^{n-1}}{(n-1)!}}\quad ,\quad \mu(E):=\frac{\deg(E)}{\rank E}\ ,
\end{equation}
respectively.
For a torsion free coherent sheaf $\Ecal\to X$ of rank $r$, we have a line bundle
 $\det\Ecal=\ddual{(\Lambda^r\Ecal)}$. Setting $c_1(\Ecal)=c_1(\det\Ecal)$,   the definitions of the degree and slope thus make sense for any torsion free sheaf.

\begin{defi} \label{def:slope-stability}
A torsion free coherent sheaf $\Ecal\to X$ is \emph{$\omega $-slope stable} (resp.\ \emph{semistable})
if any coherent  subsheaf $\Scal\subset\Ecal$ with $0<\rank
\mathcal{S}<\rank\Ecal$ satisfies $\mu(S)<\mu(\mathcal{E)}$ (resp.\ $\leq$). 
A torsion free sheaf is called \emph{polystable} if
it is a direct sum of stable sheaves of the same slope.
\end{defi}

As indicated, the notion of $\omega$-slope stability in Definition \ref{def:slope-stability} depends on the cohomology class $[\omega]$ of the polarization as soon as $\dim X\geq 2$.  In this paper, however, the class $[\omega]$ will be fixed throughout, and we shall therefore refer to $\omega$-(semi)slope stability simply as slope (semi)stability, or more often, as \emph{$\mu$-(semi)stability}. 

The  group $\Gcal(E,h)$ has a natural complexification:
$$
\Gcal^\CBbb(E)=\left\{ g\in \End E \mid (\det g)(x)=1\ ,\ \forall x\in X\right\}\ .
$$
Then $\Gcal^\CBbb(E)$ acts on the
space of holomorphic structures $\Acal_{hol}(E,\dbar_J)$ on $E$ by $
g(\bar{\partial}_{E})=g^{-1}\circ \bar{\partial}_{E}\circ g$. Two
holomorphic structures give rise to isomorphic holomorphic vector bundles if
and only if they are related by a complex gauge transformation. Therefore,
the subspaces $$\Acal^{s}_{hol}(E,\dbar_J)\subset \Acal^{ps}_{hol}(E,\dbar_J)\subset \Acal^{ss}_{hol}(E,\dbar_J)\subset \Acal_{hol}(E,\dbar_J)$$ 
of stable, 
polystable, and semistable holomorphic structures on $E$ are preserved by the action of $\Gcal^\CBbb(E)$.

\begin{defi}
The moduli space $M^{s}(E,\Jcal)$ of  stable holomorphic structures
on $E$ is defined as: 
$$M^{ s}(E,\Jcal)= \Acal^{s}_{hol}(E,\dbar_J)/\Gcal^\CBbb(E)\ .$$
\end{defi}

\begin{rem} \label{rem:notation}
In order to lighten the notation if no confusion is likely to arise, for the rest of the paper we shall most often drop $(E,h,a_J)$, $\dbar_J$ and $\Jcal$ from the notation. So for example:
$$M_\HYM:= M_\HYM(E,h,a_J)\ ,\ M_\HYM^\ast:= M_\HYM^\ast(E,h,a_J)\ ,\ M^s:=M^s(E,\Jcal)\ .
$$
\end{rem}

\subsection{The Donaldson-Uhlenbeck-Yau theorem}

Viewed as an equation for the hermitian metric $h$ on a fixed holomorphic bundle $\Ecal$, eq.\
 \eqref{eqn:hym} is a quasi-linear second order elliptic  PDE. It is therefore 
nontrivial  to determine when solutions exist. Famously, the obstruction is related to stability.

\begin{thm}[\sc Donaldson \cite{Donaldson:87}, Uhlenbeck-Yau \cite{UhlenbeckYau:86}] \label{thm:DUY}
Fix $A\in \Acal^{1,1}$. 
There exists a
HYM connection $($resp.\ irreducible HYM connection$)$ in the $\Gcal^\CBbb$ orbit of $A$ if and only if $\Ecal=(E, \dbar_A)\in \Acal_{hol}$ is polystable $($resp.\ stable$)$.
\end{thm}

This key result implies a corresponding statement at the level of moduli spaces.  Namely,
there exists a natural bijection
\begin{equation} \label{eqn:phi}
\Phi: M^{s}\stackrel{\cong}{\xrightarrow{\hspace*{.75cm}}} M^\ast_\HYM\ . 
\end{equation}
Furthermore, $M^{s}$ is a Hausdorff complex analytic space (possibly nonreduced),
and $M_{\HYM}^\ast$ has the structure of a real analytic space, such that this map restricts to a real analytic isomorphism (\cite[Chapter VII]{Kobayashi:87}, \cite[Prop.\ 4.2]{Miyajima:89}, \cite{FujikiSchumacher:87}, and \cite[Thm. 4.1.1]{LubkeTeleman:95}). 

\subsection{The Hilbert polynomial and GM-semistability}
 Recall that $L\to X$ denotes the underlying complex line bundle of $\Ocal_X(1)$ and that $c_1(L)=\lambda\cdot\omega$. For a complex vector bundle $E\to X$ we set 
\begin{equation} \label{eqn:hilbert-poly}
\tau_E(m):=\int_X \ch(E\otimes L^{m}){\rm td}(X) \ .
\end{equation}
Then $\tau_E(m)$ is a polynomial of degree $n$ in $m$.
The first two terms will be  important:
\begin{align}
\begin{split} \label{eqn:expansion}
\tau_E(m)&=m^n(2\pi\lambda^n\rank(E))+ m^{n-1}\lambda^{n-1}\bigl(\deg E+\frac{\rank E}{2}\deg TX\bigr)\\
&\qquad\qquad+O(m^{n-2})
\end{split}
\end{align}
The definition \eqref{eqn:hilbert-poly} and expansion \eqref{eqn:expansion} extend to coherent sheaves $\Ecal$, and by the Hirzebruch-Riemann-Roch Theorem and Kodaira Vanishing, for all sufficiently large natural numbers $m$ the following holds:
$$
h^0(X,\Ecal(m))=\chi_\Ecal(m) :=\chi(\Ecal(m))= \int_X \ch(\Ecal(m)){\rm td}(X) \ .
$$
 We refer to $\chi_\Ecal(m)$ as the \emph{Hilbert polynomial} of $\Ecal$. The Hilbert polynomial is clearly topological: for any holomorphic structure $\Ecal=(E,\dbar_E)$ on $E$ we have  $\chi_\Ecal(m)=\tau_E(m)$. 

An alternative notion of stability that will play an important  role in the subsequent discussion is due to Gieseker and Maruyama. Let the \emph{reduced Hilbert polynomial} be defined by
$$p_{\Ecal}(m):=\frac{\chi (\Ecal(m))}{\rank \mathcal{
E}}\ .$$
 Then we say $\Ecal$ is
\emph{GM-stable} (resp. \emph{GM-semistable}) if for any subsheaf $\mathcal{S}\subset \Ecal$
with $0<\rank\mathcal{S<}\rank\Ecal$ we have $p_{\mathcal{S}}(m)$<$p_{\Ecal}(m)$ (resp. $(\leq
)$) for $m\gg0$.  

The following  relationship between GM- and slope stability follows immediately from \eqref{eqn:expansion}:
\vspace{0.1cm}
\begin{equation*}
\mu \text{-stable}\Longrightarrow \text{GM-stable}\Longrightarrow 
\text{GM-semistable}\Longrightarrow \mu \text{-semistable}\ .
\end{equation*}

\subsection{Jordan-H\"{o}lder filtrations}
If $\Ecal$ is torsion free and
$\mu $-semistable, then it has a \emph{Seshadri filtration} (also called a \emph{Jordan-H\"{o}lder filtration}). This is a
filtration by coherent subsheaves
\begin{equation} \label{eqn:filtration}
0=\Ecal_{0}\subset \Ecal_{1}\subset \cdots \subset \Ecal
_{\ell-1}\subset \Ecal_{\ell}=\Ecal\ ,
\end{equation}
so that the successive quotients $Q_{i}=\Ecal_{i}/\Ecal_{i-1}$
are torsion free and slope stable. Moreover,  $\mu(Q_{i})=\mu(\Ecal)$, $i=1,\ldots, \ell$. We will write
$
\Gr \Ecal= \bigoplus_{i} Q_{i}
$
for the graded object associated to the above filtration.

Such a filtration (and even its associated graded $
\Gr \Ecal$)  is not uniquely determined by $\Ecal$ and $
\omega $ (see \cite {BuchdahlTelemanToma:17} for examples). On the other hand, we will see below that one can extract from it certain natural algebraic-geometric data that is unique.

By analogy with the discussion for $\mu $-semistability, a GM-semistable
sheaf $\Ecal$ has a Jordan-H\"{o}lder filtration by subsheaves whose
successive quotients are torsion free GM-stable with reduced Hilbert polynomial equal to
that of $\Ecal$. Such a filtration is not unique, but the associated
graded $\gr\Ecal$ is unique. Given two GM-semistable sheaves $
\Ecal_{1}$ and $\Ecal_{2}$, we say they are \emph{s-equivalent} if $\gr\Ecal_{1}=\gr\Ecal_{2}$. 

Notice that since a GM-semistable sheaf is also $\mu$-semistable, such a sheaf admits two types of Jordan-H\"{o}lder filtrations and graded objects, which in general will differ.  

\subsection{Cycle spaces, singular sets, and support cycles}

\subsubsection{The cycle space}

Write $\mathscr C_{p}(X)$ for the set of all analytic $p$-cycles 
$
\Ccal=\sum_{i}n_{i}Z_{i}$ on $X$, where the irreducible subvarieties $Z_{i}\subset X
$ all have  dimension $p$ and  $n_{i}\in \mathbb{N}$. Let $\mathscr C
(X)=\cup^n_{p=0}\mathscr C_{p}(X)$. This is a complex space called the \emph{cycle
space}. We define the \emph{degree} of a cycle $\Ccal\in \mathscr C_{p}(X)$ by
\begin{equation*}
\deg(\Ccal):=\sum_{i}n_{i}\int_{Z_{i}}\frac{\omega ^{p}}{p!}\ ,
\end{equation*}
where the integral is performed over the nonsingular locus of $Z_i$. 
In other words, the degree is  the weighted sum of the volumes of the $Z_{i}$ with respect to the K\"ahler metric $\omega$. We will
write $|\Ccal|=\cup _{i}Z_{i}$ for the support of $\Ccal$, and $[\Ccal]=\sum_{i}n_{i}[Z_{i}]$ for the cohomology class in $H^{2(n-p)}(X,\QBbb)$ defined by $\Ccal$. When we  occasionally allow negative integers $n_i$ in the definition such objects will be called \emph{generalised cycles}.
 The following fact will be important later.

\begin{thm}[cf.\ \cite{Bishop:64, Barlet:78}] \label{thm:cycle space}
A subset $\mathscr S\subset \mathscr C(X)$ is relatively compact if and
only if there is $K$ such that $\deg(\Ccal)\leq K$ for all  $\Ccal\in \mathscr S$.
\end{thm}

\subsubsection{Singular sets} \label{subsubsect:singular_sets}

For any coherent analytic sheaf $\Ecal\to X$,
define the set of singular points to be
\begin{equation*}
\sing(\Ecal)=\{x\in X\mid\Ecal\text{ is not locally free at }
x\}.
\end{equation*}
In general, $\sing(\Ecal)$ is an analytic subvariety of codimension at least 1. If $\Ecal$ is torsion free it is of codimension at least 2, and if  $\Ecal$ is reflexive it is of codimenison at least 3. 
Another description of this set is given by
\begin{equation*}
\sing
(\Ecal)=\bigcup\limits_{i>0}\supp\left( \mathcal{E}
xt^{i}(\Ecal,\mathcal{O}_{X})\right) ,
\end{equation*}
(see \cite[Ch.\ 2, Sect.\ 1]{OSS:80}). There is a unique \emph{torsion filtration} (see \cite[Def.\ 1.1.4]{HuybrechtsLehn:10}):
\[
\mathcal{T}_{0}(\Ecal)\subset \mathcal{T}_{1}(\Ecal)\subset
\cdots \subset \mathcal{T}_{d}(\Ecal)=\Ecal\ ,
\]
where $d=\dim(\supp(\Ecal))$, and $\mathcal{T}_{i}(\Ecal)$ is the maximal
subsheaf of $\Ecal$ fulfilling $\dim\supp(\mathcal{T}_{i}(\mathcal{E}))\leq i$. By construction, the support of $\mathcal{Q}_{i}(\mathcal{F)}=\mathcal{T}_{i}(\Ecal)/
\mathcal{T}_{i-1}(\Ecal)$ is  a pure codimension $i$
subvariety if it is nonzero. Notice that if $\Ecal$ is not a torsion sheaf,
then $d=n$, and $\mathcal{T}_{n-1}(\Ecal)\subset \Ecal$ is the
torsion subsheaf. 

Now suppose $\Ecal$ is torsion free, so that we have an injection $\Ecal\hookrightarrow \ddual\Ecal$. Since $
\mathcal{E}^{\vee \vee }$ is reflexive, it has homological dimension strictly less
than $n-1$, and therefore $\mathcal{E}xt^{i}(\mathcal{E}^{\vee \vee },
\mathcal{O}_{X})=0$ for $i\geq n-1$ (see \cite[Prop.\ V.4.14 (b)]{Kobayashi:87} and \cite[Ch.\ 2, Lemma 1.1.1]{OSS:80}). Hence, 
\begin{equation}\label{eq:sings_of_reflexive}
\sing\left( \mathcal{E}^{\vee \vee }\right)
=\bigcup _{i=1}^{n-2}\supp\left( \mathcal{E}xt^{i}(\mathcal{E}^{\vee \vee
},\mathcal{O}_{X})\right) \ .\end{equation}
 We define the \emph{codimension $k$ singular set of $
\mathcal{E}$} to be
\begin{equation*}
\sing_{n-k}(\mathcal{E}):=\supp\left( \mathcal{Q}_{n-k}(\mathcal{E
}^{\vee \vee }/\mathcal{E})\right) \cup \bigcup\limits_{i=1}^{n-2}\supp
\left( \mathcal{Q}_{n-k}(\mathcal{E}xt^{i}(\mathcal{E}^{\vee \vee },
\mathcal{O}_{X}))\right) .
\end{equation*}
This set is exactly the union of all the irreducible components of $\sing(\mathcal{E)}$ of codimension $k$.
Notice also that  
\begin{equation}\label{eq:sings_be_codim}
\sing_{n-k}(\ddual\Ecal\mathcal{)=}\bigcup\limits_{i=1}^{n-2}\supp( 
\mathcal{Q}_{n-k}(\mathcal{E}xt^{i}(\ddual\Ecal,\mathcal{O}_{X})))\ .
\end{equation}
 From \eqref{eq:sings_of_reflexive} and \eqref{eq:sings_be_codim} we therefore obtain the following  third description of the set $\sing (\Ecal)$ when $\Ecal$ is torsion free:
\begin{eqnarray*}
\sing(\Ecal) =\bigcup_{k=2}^{n}\bigcup_{i=1}^{n-2}\supp(\mathcal{Q}_{n-k}(\ddual\Ecal /\mathcal{E)})\cup 
\supp(\mathcal{Q}_{n-k}(\mathcal{E}xt^{i}(\ddual\Ecal,
\mathcal{O}_{X})) \ ,
\end{eqnarray*}
which by rearranging the terms can be seen to be equal to
$
\supp(\ddual\Ecal/\Ecal)\cup \sing(
\ddual\Ecal)$.

\subsubsection{Support cycles} \label{subsubsect:support_cycles}
Consider a general torsion sheaf $\mathcal{T}\to X$ such that $\supp(\Tcal)$ has codimension $p$, and write $Z_{j}$ for the irreducible components of codimension $p$. We will write $\mathcal{I}$ for the ideal sheaf of $Z=\supp(\mathcal{T})$
with the induced reduced structure. Then, there is some power $\mathcal{I}^{N}$ so that 
$\mathcal{I}^{N}\mathcal{T}=0$; this leads to the following filtration of $\mathcal{T}$:
\begin{equation*}
0=\mathcal{I}^{N}\mathcal{T\subset }\cdots \subset \mathcal{I}^{k+1}\mathcal{T}
\subset \mathcal{I}^{k}\mathcal{T}\subset \cdots \subset \mathcal{IT}\subset 
\mathcal{T}\text{.}
\end{equation*}
We will write 
\begin{equation*}
\gr_{\mathcal{I}}(\mathcal{T)} :=\bigoplus _{k=0}^{N-1}\mathcal{I}^{k}\mathcal{T}/\mathcal{I}^{k+1}\mathcal{T}
\end{equation*}
for the associated graded of this filtration. Notice that $\Ical$ annihilates $\gr_{\mathcal{I}}(\mathcal{T)}$, so the latter is an $\Ocal_{Z}$-module, whose first summand $\mathcal{T}/\mathcal{IT}$ is precisely the
restriction $\mathcal{T}|_{Z}$. Regarding the
restriction $\gr_{\mathcal{I}}(\mathcal{T})|_{Z_{j}}$ as a sheaf of $\Ocal_{Z_{j}}$-modules, we note that the fibre
dimension 
\begin{equation}
\dim(\gr_{\mathcal{I}}(\mathcal{T})|_{Z_{j}}(z))=\dim\bigl((\gr_{\mathcal{I}}(\mathcal{T})|_{Z_{j}})_{z}/\mathfrak{m}_{z}((\gr_{\mathcal{I}}(\mathcal{T})|_{Z_{j}})_{z})\bigr)
\end{equation}
is constant on a dense open subset of $Z_{j}$. We call this natural number the \emph{rank}. We may now associate to each $Z_{j}$ the multiplicity 
\begin{equation*}
m_{j}^{\mathcal{T}}=\rank(\gr_{\mathcal{I}}(\mathcal{T})|_{Z_{j}})\ .
\end{equation*}
Thus, associated to the sheaf $\mathcal{T}$ is the effective analytic (algebraic) cycle  
\begin{equation} \label{eqn:cycle-def}
\Ccal_{\mathcal{T}}=\sum_{j}m_{j}^{\mathcal{T}}\cdot Z_{j}\in\mathscr C_{n-p}(X)\ ,
\end{equation}
which we call the \emph{support cycle} of $\Tcal.$

To a semistable sheaf $\Ecal$, we may associate a canonical element $\Ccal_{\Ecal}\in \mathscr C_{n-2}(X)$ as follows. Consider the torsion sheaf 
\begin{equation*}
\Tcal_{\Ecal}:=\ddual{(\Gr\Ecal)}/\Gr\Ecal\ . 
\end{equation*}
Since $\Gr \Ecal$ is torsion free, sing$(\Gr\left( \Ecal\right) )$ has codimension at least 2, and admits the decomposition 
\begin{equation*}
\text{sing}(\Gr\Ecal)=\text{supp}(\ddual{(\Gr\Ecal)}/\Gr
\Ecal)\cup \text{sing}(\ddual{(\Gr\Ecal)})\ .
\end{equation*} In particular, $\supp\left(\ddual{(\Gr \Ecal)}/\Gr \Ecal\right)$
has codimension at least 2 as well. We may therefore define  the \emph{support cycle of $\Ecal$} by
\begin{equation}\label{eq:def_supp_cycle}
\Ccal_{\Ecal}:=\Ccal_{\mathcal{T}_{\Ecal}}\in \mathscr C_{n-2}(X)\, .
\end{equation}
\begin{rem}\label{rem:pair_is_unique}
The pair $(\ddual{(\Gr\Ecal)},\Ccal_{\Ecal})$ is uniquely determined by $\Ecal$ and the polarisation defining semistability (see \cite[Appendix]{BuchdahlTelemanToma:17}).
\end{rem}

As in \cite{SibleyWentworth:15} (see also the latest preprint version), we define the \emph{multiplicities} of $\Ecal$ to be the positive integers $m^{\Ecal}_j=\rank(\gr_{\mathcal{I}}(\mathcal{T}_{\mathcal{E}}))\bigr|_{Z^{\Ecal}_j})$ that appear as the coefficients of $\Ccal_{\Ecal}$.
 In \cite[Rem.\ 5.3]{GrebToma:17} the multiplicities are defined differently. Namely, let $S$ be a general
  complete intersection surface intersecting ${Z^{\Ecal}_j}$ transversally in a finite number of smooth points $\{z_{1},\cdots ,z_{k}\}$, and set
\begin{equation*}
n_{j}^{\Ecal}(S,z_{i})=\ell_{\mathcal{O}_{S,z_{i}}}(\mathcal{T}_{\Ecal}|_{S})_{z_{i}}\ ,
\end{equation*} 
where $\ell$ denotes the length of a module.
Since we will rely on results from both \cite{SibleyWentworth:15} and \cite{GrebToma:17} that were proven using these two definitions, we point out  the following simple lemma.

\begin{lemma} For general $S$ as above,
$
m_{j}^{\Ecal}=n_{j}^{\Ecal}(S,z_{i})$.
 In particular, $n_{j}^{\Ecal}(S,z_{i})$ is independent of $z_{i}$ and S. 
\end{lemma}

\begin{proof}
We begin with the observation that, since the ideal $\Ical$ that cuts out the support of $\Tcal_{\Ecal}$ annihilates $\gr_{\mathcal{I}}(\mathcal{T_{\Ecal})}$, then if $\imath$ is the natural inclusion of the support in $X$, we have $\gr_{\mathcal{I}}(\mathcal{T_{\Ecal})=}$ $\imath _{\ast }(\imath ^{\ast }\gr_{
\mathcal{I}}(\mathcal{T_{\Ecal}))}$. In other words, for a sufficiently small open set $U$ containing $z_i\in Z_j$, $\gr_{\mathcal{I}}(\mathcal{T)}$ may be expressed as an $\Ocal_X$-module as $T\otimes_\CBbb \Ocal_{Z_{j}}$, for a $\CBbb$-vector space $T$ of dimension $m_j^\Ecal$. Furthermore, $T=\bigoplus T_{k}$, where $T_{k}$ is the vector space corresponding to each summand of $\gr_{\mathcal{I}}(\Tcal_{\Ecal})$.  We consider the exact sequence
$$
0\lra T\otimes_\CBbb\Ical_{U\cap Z_j}\lra T\otimes_\CBbb \Ocal_U\lra \gr_{\mathcal{I}}(\Tcal_{\Ecal})\bigr|_{U}\lra 0\ .
$$
By our assumptions, tensoring with $\Ocal_S$ we get an exact sequence of $\Ocal_{S}$-modules:
\[0\lra T\otimes_\CBbb\Ical_{z_i}\lra T\otimes_\CBbb \Ocal_{U\cap S}\lra \gr_{\mathcal{I}}(\mathcal{T_{\Ecal})}\otimes_{\Ocal_U} \Ocal_{U\cap S}\lra 0\ .
\]
The length of the torsion quotient on the right hand side (the dimension of its fibre over $z_{i}$) is by definition $\ell_{\mathcal{O}_{S,z_{i}}}(\gr_{\mathcal{I}}(\mathcal{T_{\Ecal})}|_{S})_{z_{i}}$. The relevant fibre over $z_{i}$ is clearly isomorphic so $T$, so this number is equal to $\dim T$. 

On the other hand, considering the exact sequences
\begin{equation*}
0\rightarrow \mathcal{I}^{k+1}\mathcal{T}|_{U}\mathcal{\rightarrow I}^{k}
\mathcal{T}|_{U}\rightarrow T_{k}\otimes \mathcal{O}_{Z_{j}\cap U}\rightarrow 0\ ,
\end{equation*}
and again using the assumptions on $Z_{j}$ and $S$, we have
 $\mathcal{T}or^{\mathcal{O}_X}_{1}(\mathcal{O}_{Z_{j}},\mathcal{O}_{S})=0,$ and so tensoring by $
\mathcal{O}_{S}$ we obtain
\begin{equation*}
0\rightarrow \mathcal{I}^{k+1}\mathcal{T}|_{U\cap S}
\mathcal{\rightarrow I}^{k}
\mathcal{T}|_{U\cap S}\rightarrow T_{k}\otimes \mathcal{O}_{z_{i}}\rightarrow 0\ .
\end{equation*}
Using additivity of the length in exact sequences, we obtain \[n_j(S, z_i) = \ell_{\mathcal{O}_{S,z_{i}}}(
\mathcal{T}_{\mathcal{E}}|_{S})_{z_{i}}=\sum_{k=0}^{N-1} \dim T_{k}=\dim T=m_{j}^{\mathcal{E}}\ . \qedhere\] \end{proof}

We will repeatedly use the following result from \cite{SibleyWentworth:15} (see also the newest preprint version) which is certainly also well-known in the algebraic geometry literature. It relates the support cycle defined
above to the first nonzero part of the Chern character of $\mathcal{T}$.  

\begin{lemma}\label{lem:cherncharactertorsionsheaf}
Let $\mathcal{T}$ be a torsion sheaf with $\codim(\supp(\mathcal{
T}))=p$. Then $ch_{k}(\mathcal{T})=0$ for $k<p$, and
$\ch_{p}(\mathcal{T})$ coincides with the Poincar\'e dual in $H^{2p}(X, \mathbb{Q})$ of the homology class of the cycle $\Ccal_\Tcal$ defined in \eqref{eqn:cycle-def}.
\end{lemma}

\subsubsection{Generalised cycles}\label{subsubsect:generalised_cycles}

For inductive arguments, the following more general notion of associated cycle is often useful.
\begin{defi} \label{defi:cycle,couple} 
Let $\Ecal\to X$ be  coherent  such that its torsion part $\Tors(\Ecal)$ has support in codimension $\geq 2$. 
  We define the {\em $($generalised$)$ codimension $2$ cycle of} $\Ecal$ as
\begin{equation*}
 \Ccal(\Ecal):= \Ccal_{\ker(\Ecal\to \ddual\Ecal)}-\Ccal_{\coker(\Ecal\to \ddual\Ecal)}\ . 
\end{equation*}
We further set, $\gamma(\Ecal):=(\ddual\Ecal, -\Ccal(\Ecal))$. If $\Ecal$ torsion free, we set $\Qcal_\Ecal:=\ddual\Ecal/\Ecal$ and we put $\widehat \Qcal_\Ecal:= \Qcal_\Ecal/\cT(\Qcal_\Ecal)$, where $\cT(\Qcal_\Ecal)$ , where $\cT (\Qcal_{\cF}) = \cT_{n-3} (\Qcal_{\cF})$ is the maximal subsheaf of $\Qcal_{\cF}$ of dimension less than or equal to $n-3$; cf.\ \cite[Def.\ 1.1.4]{HuybrechtsLehn:10}.
Notice that since $\Qcal_\Ecal$ has support in  codimension at least $2$, $\widehat \Qcal_\Ecal$ is pure of codimension $2$ or vanishes. In this case $\mathcal{C}(\Ecal)=-\mathcal{C}_{\mathcal{Q}_{\Ecal}}=-\mathcal{C}
_{\widehat{\mathcal{Q}}_{\Ecal}}$.  Moreover, we will write $[\gamma(\Ecal) ]$ for the class of $\gamma(\Ecal)$ under the natural equivalence relation taking isomorphism classes in the first component and equality of cycles in the second component.
 \end{defi}

\begin{rem}\label{rem:gamma_campatible_with_earlier_defi}
Note that if $\Ecal$ is polystable, then $\Ecal=\Gr\Ecal$, $\Qcal_\Ecal$ is equal to the torsion sheaf $\mathcal{T}_{\Ecal}$ previously defined, and $\mathcal{C}(\Ecal)=-\mathcal{C}_{\Ecal}$ so that $\gamma(\Ecal)=(\ddual\Ecal, \Ccal_{\widehat{\mathcal{Q}}_{\Ecal}})=(\ddual\Ecal, \mathcal{C}_{\mathcal{Q}_{\Ecal}})=(\ddual\Ecal,\mathcal{C}_{\Ecal})$. 
\end{rem}

The following three results gather properties of codimension $2$ cycles.  

\begin{lemma}\label{lemma:cycles1}
If $0\to \Ecal_1\to  \cdots\to \Ecal_m\to 0$ is an exact sequence of coherent sheaves on $X$ with supports of codimension $\geq 2$, then $\sum_j(-1)^j\Ccal(\Ecal_j)=0$.
\end{lemma}
\begin{proof}
This is easily checked by cutting with general complete intersection surfaces $S\subset X$ and using the fact  that the  multiplicities of $\Ccal(\Ecal_j)$ are equal to the lengths of the skyscraper sheaves $\Ecal_j|_S$.
\end{proof}

\begin{lemma}\label{lemma:cycles2}
Let $\alpha:\Ecal_1\to \Ecal_2$ be a morphism between coherent sheaves on $X$ inducing an isomorphism $\ddual\alpha:\ddual{\Ecal_1}\isorightarrow \ddual{\Ecal_2}$ between the double duals. Then
$$\Ccal(\Ecal_1)=\Ccal(\Ecal_2)+\Ccal(\ker \alpha)-\Ccal(\coker\alpha)\ .$$
\end{lemma}
\begin{proof}
The idea is to reduce the situation to an application of Lemma \ref{lemma:cycles1} on several exact sequences of coherent sheaves of codimension at least 2.
By decomposing the sequence 
$$0\to\ker \alpha\to \Ecal_1\to \Ecal_2\to\coker\alpha\to0$$
into short exact sequences one
 first reduces the question to the cases where $\ker\alpha=0$ or $\coker\alpha=0$. One then compares these short exact sequences with the morphisms $\Ecal_1\to \ddual{\Ecal_1}$ and $\Ecal_2\to \ddual{\Ecal_2}$.  
\end{proof}
 
The next basic result will be needed in Sections \ref{sec:limits1} and \ref{sec:agree-cycle-component}; see ~\cite[Appendix]{BuchdahlTelemanToma:17} and \cite[Cor.~2.23]{ChenSun:19} for very similar results.  
\begin{prop}\label{prop:polystablereflexive}
Let $\mathcal{E}$ be a torsion free sheaf such that $\mathcal{E}^{\vee \vee }$
is polystable, and let  $\mathcal{T}=\mathcal{E}^{\vee \vee }/\mathcal{E}$. Then $\Ecal$ is $\mu$-semistable. Moreover, 
if $\Gr\Ecal$ is the associated graded  to any
Jordan-H\"{o}lder filtration of $\Ecal$, then  $(\Gr\mathcal{E})^{\vee \vee }\cong \mathcal{E}^{\vee \vee }$ and
 $\Ccal_{\mathcal{E}}=\Ccal_{\mathcal{T}}$.
\end{prop}

\begin{proof}
 Since $\mathcal{E}^{\vee\vee}$ is polystable it is in particular 
 $\mu$-semistable.  Suppose $\mathcal{F}\subset \Ecal$ is a proper subsheaf with $\mu(\mathcal{F})>\mu(\Ecal)$. Then $\Fcal$ is also a subsheaf of $\Ecal^{\vee\vee}$ with $\mu(\mathcal{F})>\mu(\Ecal^{\vee\vee} )$, since $\Ecal$ and $\Ecal^{\vee\vee} $ coincide in codimension one and hence have the same slope. This  violates semistability of $\Ecal^{\vee\vee}$, and so $\Ecal$ is $\mu $-semistable, proving the first claim.

Recall the definition of the Jordan-H\"older filtration \eqref{eqn:filtration}.
We claim that taking the saturations $\widehat{\mathcal{E}}_{i}$ of the 
$\mathcal{E}_{i}$ inside $\mathcal{E}^{\vee \vee }$ gives a 
Jordan-H\"{o}lder filtration for $\mathcal{E}^{\vee \vee }$. First of all, 
$
\widehat{Q}_{i}=\widehat{\mathcal{E}}_{i}/\widehat{\mathcal{E}}_{i-1}
$
is torsion free by construction. Moreover, there is an isomorphism 
\begin{equation}\label{eq:dd_coincide}
 (\widehat{Q}_{i})^{\vee \vee }\cong (Q_{i})^{\vee \vee }=(\mathcal{E}_{i}/
\mathcal{E}_{i-1})^{\vee \vee },
\end{equation}
 since these sheaves are isomorphic away from 
$\supp(\mathcal{T})$,
which has codimension greater than or equal to two. In fact, notice that since a saturated subsheaf of
a reflexive sheaf is reflexive, we actually have 
$\widehat{\mathcal{E}}_{i}\cong \mathcal{E}_{i}^{\vee \vee }$,
 by normality. The stability of $
\widehat{Q}_{i}$ then follows since $Q_{i}$ is stable by definition, and a
sheaf is stable if and only if its double dual is, see \cite[Chap.~II, Lemma 1.2.4]{OSS:80}. Therefore $\widehat{\mathcal{E}}_\bullet$ defines a Jordan-H\"{o}lder filtration.

It follows from \eqref{eq:dd_coincide} that 
$\displaystyle
(\Gr\mathcal{E)}^{\vee \vee }\cong \oplus_{i}(\widehat{Q}_{i})^{\vee \vee }
$.
The right hand side of this isomorphism is the double dual of the associated
graded object for the Jordan-H\"{o}lder filtration for $\mathcal{E}^{\vee
\vee }$ by $\widehat{\mathcal{E}}_{i}\cong \mathcal{E}_{i}^{\vee \vee }$.
Since this sheaf is uniquely associated to $\mathcal{E}^{\vee \vee }$ by Remark~\ref{rem:pair_is_unique}, and $
\mathcal{E}^{\vee \vee }$ is polystable with possibly
different Jordan-H\"{o}lder filtration whose successive quotients are the
direct summands in the decomposition of $\mathcal{E}^{\vee \vee }$, we obtain the isomorphism
$
(\Gr\mathcal{E)}^{\vee \vee }\cong \mathcal{E}^{\vee \vee }$.

In order to prove the claim regarding cycles, we start by observing that $(\mathcal{E} / \mathcal{E}_1)^{\vee \vee} \cong \mathcal{E}^{\vee\vee} / \mathcal{E}_1^{\vee \vee}$,
as $\mathcal{E}^{\vee \vee} / \mathcal{E}_1^{\vee \vee}$ is a direct summand of $\Gr(\mathcal{E})^{\vee \vee}$ and hence reflexive. Therefore, we obtain the following diagram with exact rows:
\[\begin{xymatrix}{
0 \ar[r]& \mathcal{E}_1 \ar[d] \ar[r]& \mathcal{E} \ar[r]\ar[d]& \mathcal{E}/\mathcal{E}_1\ar[r]\ar[d]^\alpha & 0\\
0 \ar[r]& \mathcal{E}_1^{\vee\vee} \ar[r]& \mathcal{E}^{\vee\vee} \ar[r]& (\mathcal{E}/\mathcal{E}_1)^{\vee\vee}\ar[r] & 0
}
  \end{xymatrix}
\]
Notice that $\alpha$ is injective, as $\mathcal{E}/\mathcal{E}_1$ is torsion free. An application of the snake lemma  yields an exact sequence
\[0 \to \mathcal{E}_1^{\vee\vee}/\mathcal{E}_1 \to \mathcal{E}^{\vee\vee}/\mathcal{E}  \to (\mathcal{E}/\mathcal{E}_1)^{\vee\vee} / (\mathcal{E}/\mathcal{E}_1) \to 0.\]
Lemma~\ref{lemma:cycles1} therefore implies that $\mathcal{C}(\mathcal{E}) = \mathcal{C}(\mathcal{E}_1) + \mathcal{C}(\mathcal{E}/\mathcal{E}_1)$. Observe that $\mathcal{E}/\mathcal{E}_1$ is again torsion free with polystable double dual, so that Remark~\ref{rem:gamma_campatible_with_earlier_defi} allows us to conclude by induction.
\end{proof}

\subsubsection{Some boundedness results}\label{sec:boundedness}
The following will be used later on.
\begin{prop}\label{prop:Gr-boundedness} 
Let $\mathcal{S}$ be a set of semistable sheaves $\Ecal$ that is bounded in the sense of \cite{Grothendieck:61}. Then the set of all possible Seshadri graduations $\Gr\Ecal$ for $\Ecal\in\Scal$ 
is also bounded. 
\end{prop}

\begin{proof}
Let $r$ be the maximal rank of the sheaves appearing in $\Scal$. Let $\Ecal$ any sheaf in $\Scal$. Then the length $\ell$ of any Seshadri filtration 
\eqref{eqn:filtration} of $\Ecal$ is bounded by $r$. 
We have 
$\Gr\Ecal=\oplus_{i=1}^\ell \Ecal_i/\Ecal_{i-1}$, and it is enough to show that the isomorphism classes of all quotients $ \Ecal_i/\Ecal_{i-1}$ arising from such filtrations form a bounded set. 
If $\Ecal$ is such that $\ell=1$, we have $\Gr\Ecal=\Ecal_1/\Ecal_0=\Ecal$, and the statement is clear.

Consider now the subset of $\Scal$ consisting of those $\Ecal$ with $\ell>1$. Then for such sheaves $\Ecal$, the isomorphism classes of quotients $ \Ecal_\ell/\Ecal_{\ell-1}=\Ecal/\Ecal_{\ell-1}$ belong to a bounded set  by Grothendieck's lemma, \cite[Lemma 1.7.9]{HuybrechtsLehn:10}, since their slope equals to $\mu(\Ecal)$ and is therefore bounded. The sheaves $\Ecal_{\ell-1}$ may now be viewed as kernels of the projections $\Ecal_\ell\to\Ecal_\ell/\Ecal_{\ell-1}$ so their isomorphism classes form a bounded set $\Scal'$. 
We continue by applying the above procedure to the set $\Scal'$ and conclude by descending induction on $r$.
\end{proof}

In the next section we will also require a boundedness result for the codimension $k$ singular sets of a bounded set of polystable reflexive sheaves. We first need the following elementary proposition.  

\begin{prop}\label{prop:boundedness} 
Let $\mathcal{S}$ be a set of polystable reflexive sheaves $\Ecal$ that is bounded in the sense of \cite{Grothendieck:61}. Then for each $k$ and each $i$, the set
\begin{equation*}
\{\mathcal{Q}_{n-k}(\mathcal{E}xt^{i}(\mathcal{E},\mathcal{O}_{X}))\}_{
\mathcal{E}\in \mathcal{S}}
\end{equation*}
is also bounded. 
\end{prop}

\begin{proof}
 Since $\Scal$ is bounded, by \cite[p.~251]{Grothendieck:61}, we know that the sets of sheaves  $
\{\mathcal{E}xt^{i}(\Ecal,\mathcal{O}_{X})\}_{\mathcal{E}\in \Scal}$, are bounded. This means that for any $i,j,k$, the associated sets 
$\{\mathcal{T}_{j}(\mathcal{E}xt^{i}(\Ecal,\mathcal{O}_{X}))\}_{\mathcal{E}\in \mathcal{S}}
$,
 and therefore also $\{\mathcal{Q}_{n-k}(\mathcal{E}xt^{i}(\Ecal,\mathcal{O}_{X}))\}_{\mathcal{E}\in \mathcal{S}}$, are bounded. 
\end{proof}

\begin{Cor} \label{Cor:singsetboundedness}
In the setup of Proposition\ \ref{prop:boundedness},  the set
$
\{\deg\sing_{n-k}(\mathcal{E)\}}_{\mathcal{E}\in \mathcal{S}}
$, 
is finite for each $k$, where we regard $\sing_{n-k}(\mathcal{E)}$ as an element of 
$\mathscr C_{n-k}(X)$ by assigning the weight $1$ to each of its
irreducible components. In particular, the number of irreducible components
of the sets $\sing_{n-k}(\mathcal{E)}$ is bounded. As a result, the
set 
$
\{\sing_{n-k}(\mathcal{E)\}}_{\mathcal{E}\in \mathcal{S}}
$
is relatively compact in $\mathscr C_{n-k}(X)$. 
\end{Cor}

\begin{proof}
By \eqref{eq:sings_be_codim}, we have 
\[
\sing_{n-k}(\mathcal{E})=\bigcup_{i=1}^{n-2}\supp( \mathcal{Q}_{n-k}(
\mathcal{E}xt^{i}(\Ecal,\mathcal{O}_{X})))\ ,
\]
and by the preceding proposition the family of sheaves appearing on the
right hand side is bounded. By \cite[Thm. 2.1]{Grothendieck:61}, this implies in particular
that for each $i$ the set of Hilbert polynomials 
$
\{\chi _{\mathcal{Q}_{n-k}(\mathcal{E}xt^{i}(\Ecal,
\mathcal{O}_{X}))}(m)\}_{\Ecal\in\Scal}
$, is finite. 
The $(n-k)$-th coefficients of these polynomials are precisely
\[
\int_{X}\ch_k(\mathcal{Q}_{n-k}(\mathcal{E}xt^{i}(\Ecal,\mathcal{O}_{X}))\wedge 
\frac{\omega ^{n-k}}{(n-k)!}\ ,
\]
and these are therefore finite in number. On the other hand,
\[
\ch_k(\mathcal{Q}_{n-k}(\mathcal{E}xt^{i}(\Ecal,\mathcal{
O}_{X}))={\rm PD}[\Ccal_{\mathcal{Q}_{n-k}(\mathcal{E}xt^{i}(\Ecal,\mathcal{O}_{X}))}]
\]
are the (Poincar\'e duals of) the support cycles of these torsion sheaves (see Lemma \ref{lem:cherncharactertorsionsheaf}). Therefore, the coefficients are given by the degrees of the cycles 
$\Ccal_{\mathcal{Q}_{n-k}(\mathcal{E}xt^{i}(\Ecal,\mathcal{O}_{X}))}$,
and so there are only finitely many such degrees. Since the sum of these cycles is exactly the cycle associated to $\sing_{n-k}(\mathcal{E})$, except
with possibly larger multiplicities, the set of degrees $\deg(\sing_{n-k}(\mathcal{E)})$ is also finite. The second statement follows
directly from Theorem \ref{thm:cycle space}. 
\end{proof}
\subsection{The Quot scheme, natural subschemes, and convergence}\label{subsect:quotschemes_and_bundles}
The starting point for forming moduli spaces of sheaves is that the set of
slope (or GM-) semistable sheaves with fixed Chern classes $c_{i}(E)$ is a
bounded family (see for example \cite{Maruyama:81}). In practice this means that there is a natural number $
m_{0}\gg0$, such that for any $m>m_{0}$ and any such sheaf $\Ecal$, $H^{i}(X,\Ecal
(m))=0$ for $i>0$, and $\Ecal(m)$ is globally generated.

Write $V=\CBbb^{\tau _{E}(m)}$ and set
\begin{equation} \label{eqn:H}
\Hcal=V\otimes \mathcal{O}_{X}(-m)\ .
\end{equation}
If we choose an isomorphism $H^{0}(X,\Ecal(m))\cong V$ with $m$ chosen as above, then we have  a surjection $\Hcal\to \Ecal\to 0$. Notice that the map we obtain depends on our choice of isomorphism.
Hence, we see  that to any slope or GM-semistable sheaf with Chern classes $c_{i}(E)$  we can associate a point  in the \emph{Quot scheme} $\Quot(\Hcal, \tau_E)$, which is defined as the set of equivalence classes of quotients $q_{\Ecal}: \Hcal\to \Ecal\to 0$, where $\Ecal\to X$ is a coherent sheaf with Hilbert polynomial equal to $\tau_E$. Quotients $q_1$ and $q_2$ are equivalent if $\ker q_1=\ker q_2$. This is equivalent to the existence of a commutative diagram
$$
\xymatrix{
\Hcal \ar[r]^{q_{1}} \ar@{=}[d] &\Ecal\ar[d]^{\varphi}\ar[r] &0 \\
\Hcal \ar[r]^{q_{2}} &\Ecal \ar[r]&0
}
$$
where $\varphi$ is an isomorphism.

Observe that there is an action $\GL(V) \curvearrowright \Quot(
\mathcal{H},\tau)$, namely, the map $g\cdot q_{\Ecal}:\mathcal{H}
\rightarrow \Ecal\rightarrow 0$ is given by composing $q_{\Ecal}$
with  $g$.
This action amounts to the fact that there is an ambiguity due to the choice
of isomorphism $H^{0}(X,\Ecal(m))\cong V$ (i.e.\ a choice of basis for $
H^{0}(X,\Ecal(m))$). Since $\mathbb{C}^{\ast }$ acts trivially
(rescaling the basis vectors by the same constant results in the same
kernel), it will suffice to consider the action of $\SL(V)$. 

 By the construction in \cite{Grothendieck:61}, $\Quot(\Hcal, \tau_E)$ is a projective scheme. It is furthermore a fine moduli space, so in particular
there is a universal quotient sheaf $q_{\mathscr U}:\Hscr\rightarrow \mathscr U\rightarrow 0$, where $\Hscr$ is the pullback of $\mathcal{H}$ via the projection map,
and $\mathscr U\rightarrow X\times \Quot(\mathcal{H},\tau _{E})$ is a
flat family so that $\mathscr U|_{X\times \{q_{\Ecal}\}}=\Ecal$,
and $q_{\mathscr U}|_{X\times \{q_{\Ecal}\}}=q_{\Ecal}$. There are subspaces $\Quot(\mathcal{H},(c_{1},\cdots ,c_{\min
(r,n)}))\subset \Quot(\mathcal{H},\tau _{E})$ consisting of
those quotients with fixed Chern classes $(c_{1},\cdots ,c_{\min (r,n)})$. Since the Chern classes of a flat family (in particular those of $\mathscr U$) are locally constant, we have a decomposition:
\[
\Quot(\mathcal{H},\tau _{E})=\coprod \Quot(\mathcal{H}
,(c_{1},\cdots , c_{\min (r,n)}))\ ,
\]
where the union is over all tuples of Chern classes whose associated Hilbert
polynomial is $\tau _{E}$. We will write  
\[
\Quot(\mathcal{H},c(E)):=\Quot(\mathcal{H},(c_{1}(E),\cdots
,c_{\min (r,n)}(E))\ .
\]

Let $R^{\muss}$ (resp.\ $R^{\Gss}$) 
$\subset \Quot(\mathcal{H},c(E))$ denote the subscheme of quotients $q:\Hcal \to \Ecal$ satisfying:
\begin{enumerate}
\item $\Ecal$ is torsion free;
\item $\det\Ecal\simeq \Jcal$;
\item $\Ecal$ is $\mu$-semistable (resp.\ GM-semistable);
\item $q$ induces an isomorphism $V\isorightarrow H^0(X,\Ecal(m_0))$.
\end{enumerate}
The spaces $R^{\muss}$ and $R^{\Gss}$ are preserved by the action of $\SL(V)$, and there is an inclusion $R^{\Gss}\hookrightarrow R^{\muss}$.

In Section \ref{sec:gauge-theory} we will need a result concerning the meaning of convergence of a sequence in the space $\Quot (\mathcal{H},\tau_E)$ in the analytic topology. For the following,  fix a hermitian structure on $\Hcal$.

\begin{lemma} \label{lem:quot-topology}
Let $q_i: \Hcal\to \Ecal_i\to 0$, $q:\Hcal\to \Ecal\to 0$ be points in $\Quot (\mathcal{H},\tau )$. suppose each $\Ecal_i$ is locally free with underlying $C^\infty$-bundle smoothly isomorphic to $E$. Let $\pi_i$ denote the orthogonal projections to $\ker q_i$, and $\pi$ the orthogonal projection to $\ker q$ on the open set $X\backslash\sing\Ecal$ where $\Ecal$ is locally free.  
If $q_i\to q$ in the analytic topology of $\Quot (\mathcal{H},\tau )$, then
\begin{enumerate}
\item $\pi_i\to \pi$ smoothly on $X\backslash\sing\Ecal$, and
\item on $X\backslash\sing\Ecal$, the underlying $C^\infty$-bundle of $\Ecal$ is smoothly isomorphic to $E$.
\end{enumerate}
\end{lemma}

\begin{proof}
Recall from \cite{Grothendieck:61} that $\Quot(\Hcal,\tau_E)$ admits an embedding into a Grassmannian as follows. By choosing $k\gg0$ we may assume $$H^1(X, \ker q\otimes \Ocal_X(k))=\{0\}$$ for all $q\in \Quot(\Hcal,\tau)$. Hence, $q$ defines a point  $P_q\in \Gr(r,N)$, where 
$N=\dim W$, where $W:= H^0(X, \Hcal(k))$, and $r=N-\tau(k)$. By choosing hermitian structures, we may regard $P_q$ as the element of $\End W$ given by  orthogonal projection to $K_q:=H^0(X, \ker q\otimes \Ocal_X(k))$. 
Applying this to the situation in the statement of the lemma, convergence $q_i\to q$ in the analytic topology implies smooth convergence $P_{q_i}\to P_q$ in $\End W$. This gives a sequence of smooth bundle morphisms $W\otimes \Ocal_X\to \Hcal(k)$ gotten from the composition
$$
W\otimes \Ocal_X
\stackrel{P_{q_i}}{\xrightarrow{\hspace*{.75cm}}}W\otimes \Ocal_X\lra \Hcal(k)\ ,
$$
whose images are precisely $\ker q_i\otimes \Ocal_X(k)$. Since the bundle morphisms converge smoothly, items (1) and (2)  clearly follow.
\end{proof}

\subsection{Flat families and line bundles over Quot schemes}
 The line bundles in question come from the following general construction (see \cite[Section~8.1]{HuybrechtsLehn:10} for more details). 
 Let $S$ be a scheme over $\CBbb$. Given an $S$-flat family of coherent sheaves $\Escr \rightarrow X\times S$, the determinant of cohomology $\lambda _{\Escr}:K(X)\rightarrow \Pic(S)$ is defined by
\begin{equation}\label{eq:def_of_det_bundle}
 \lambda_{\Escr}\left( u\right) :=\det\left((p_{S})_!(p_{X}^{\ast }(u)\otimes 
{\Escr})\right).
\end{equation}
Here, $K(X)=K_{0}(X)=K^{0}(X)$ is the Grothendieck group of holomorphic vector bundles (and of coherent sheaves) on $X$, $p_{X}:X\times S\rightarrow X$ and $
p_{S}:X\times S\rightarrow S$ are the projections, and $(p_{S})_!:K^{0}(X\times
S)\rightarrow K^{0}(S)$ is defined on a class $[F]$ represented by a sheaf $\mathscr{F}$ as 
\[(p_S)_!([\mathscr{F}]) = \sum_{i=0}^{\dim X} (-1)^i [R^i(p_S)_*\mathscr{F}] \in K^0(S) ,\]
and is then extended to $K^0(X)$ by linearity. Moreover, in our situation, it can be shown that $(p_{S})_!(p_{X}^{\ast }(u)\otimes 
{\Escr}) \in K_0(S)$, so that applying the determinant in \eqref{eq:def_of_det_bundle} makes sense; see the discussion preceding \cite[Cor.~2.1.11]{HuybrechtsLehn:10} for details. Therefore, we may construct line bundles on the parameter space $S$ by specifying classes in $K(X)$. One natural way to do this is to consider classes arising from complete intersections as follows; cf.~\cite[Section 3.2]{GrebToma:17}.

Let $H\subset X\subset \PBbb^N$ be a hyperplane section and consider the class
$[\mathcal{O}_{H}]\in K(X)$. 
We fix a class $c\in K(X)_{{\rm num}
}:=K(X)/\sim $, where $u\sim v$ iff the difference $u-v$ lies in the kernel of the quadratic form induced by the Euler characteristic $\chi$, i.e., $u-v\in \ker ((a,b)\mapsto \chi
(a\cdot b))$. By the Riemann-Roch Theorem, the numerical behaviour of such a class is hence determined by its associated rank $r$ and its Chern classes $c_{i}\in H^{2i}(X,
\mathbb{Z})$. As we consider sheaves with fixed determinant, we furthermore fix a line bundle $\Jcal \in \Pic(X)$ such that $c_{1}(\Jcal )=c_{1}$. Now, for any integer $1\leq i\leq n-1$, let $X_{i}\in |H|$ and write $X^{(0)}=X$ and $X^{(l)}=\cap _{i=1}^{l}X_{i}$ with $1\leq l\leq n-1$. We
will assume that $X^{(l)}$ is smooth for each $l$.  Finally, fix a basepoint $x\in X^{(n-1)}$, whose structure sheaf $\mathcal{O}_x$ defines a class $[\mathcal{O}_x]$ in each $K(X^{(l)})$. 

Then, for each $1\leq l \leq n-1$ we
define a class in $ K(X^{(l)})$ by
$$
u_{n-1-l}(c|_{X^{(l)}}, [\mathcal{O}_{H}]_{X^{(l)}}) 
:=-r([\mathcal{O}_{H}]|_{X^{(l)}})^{n-1-l}+\chi (c|_{X^{(l)}}\cdot
([\mathcal{O}_{H}]|_{X^{(l)}})^{n-1-l})[\mathcal{O}_{x}]\ .
$$
Here, multiplication "$\cdot$" and powers "$\bullet\, ^{n-1-l}$" are taken with respect to the ring structure of $K(X^{(l)})$, which is derived from the operation of taking tensor product of locally free sheaves.

Supposing furthermore that the family $\Escr|_{X^{(l)}\times S} \rightarrow X^{(l)}\times S$ is flat over $S$ for each $l$, by using the previous construction, we can produce a sequence of line
bundles 
\[
\mathscr{L}_{\Escr,n-1-l}:=\lambda _{\Escr
}\bigl(u_{n-1-l}(c|_{X^{(l)}},[\mathcal{O}_{H}]_{X^{(l)}})\bigr)\ .
\]

The main case of interest in this article  is $l=0$, $r = \rank (E)$, $c_i=c_i(E)$, and $\Escr=\Uscr$, the universal sheaf on $\Quot(\mathcal{H},\tau
_{E})$. By the previous discussion we obtain a line bundle $\mathscr{L}_{
\Uscr,n-1}\rightarrow \Quot(\mathcal{H},\tau _{E})$ which we abbreviate
by $\mathscr{L}_{n-1}$. We will discuss properties of this line bundle in the following subsection.

\subsection{Compactifications by sheaves} \label{sec:moduli}
By Theorem \ref{thm:DUY},  the moduli space $M_{\HYM}^\ast$ is identified with the complex analytic space $M^{s}$. By a result of Miyajima \cite{Miyajima:89}, the latter is isomorphic as a complex analytic space to both $M_{an}^{s}$ and $M_{alg}^{s}$, which are moduli spaces corepresenting the appropriate moduli functors (namely those associating to a complex space or $\mathbb{C}$-scheme the set of isomorphism classes of flat families of slope stable holomorphic or algebraic bundles which are smoothly isomorphic to $E$ over this space).\footnote{In the case of  $M_{alg}^{s}$ we really refer to its analytification.} We will sometimes identify all three spaces and use the notation $M^{s}$. 

Since $M_{alg}^{s}$ corepresents a subfunctor of the moduli functor for the Gieseker moduli space $M^{\Gss}$ of GM-semistable sheaves, and forms a Zariski open subspace, $M^{\Gss}$ compactifies $M^{s}$, and therefore also $M_{\HYM}^\ast$. Below we briefly recall  the construction in \cite{GrebToma:17} of another compactification of $M_{\HYM}^\ast$ based  instead on slope semistable sheaves. Both of these compactifications arise as the image of a  scheme in projective space under sections of line bundles produced by the construction of the previous section.

Denote by $\mathscr{L}_{\Gss}\in\Pic(\Quot(\mathcal{H},\tau_E))$ the bundle $\lambda _{\mathscr{U}}(\mathcal{O}_{X}(l))$
associated to the universal sheaf $\mathscr{U}\rightarrow X\times \Quot(\mathcal{H},c(E))$. This
 turns out to be ample for $l$ sufficiently large, and it posseses
an $\SL(V)$ linearisation. The reason for the subscript $Gss$ is that the
subspace $R^{\Gss}\subset \Quot(\mathcal{H},c(E))$, is exactly the set of
GIT semistable points of $\mathscr{L}_{\Gss}$ on the Zariski closure $
\overline{R^{\Gss}}$ with respect to this linearisation. The
\emph{Gieseker moduli space} is the GIT quotient \[
M^{\Gss}:=\overline{R^{\Gss}}\doublequotient_{\mathscr{L}_{\Gss}}\SL(V)\ .
\]
By definition this is the Proj of the invariant section ring of $\mathscr{L}_{\Gss}$:   
\[
M^{\Gss}={\rm Proj}\left(  {\bigoplus }_{k}H^{0}\left( \overline{R^{\Gss}},
\mathscr{L}_{\Gss}^{k}\right) ^{\SL(V)}\ \right) \ .
\]

 A point in $M^{\Gss}$ is represented by a GM-semistable quotient $q_{\Ecal}:\mathcal{H}\rightarrow \Ecal$ with $\det\Ecal\simeq \Jcal$. Moreover GIT gives a way to understand the geometry of $M^{\Gss}$. More precisely, two quotients $q_{\Ecal_{1}}$,$q_{\Ecal_{2}}\in R^{\Gss}$ represent the same point of $M^{\Gss}$ if and only if they are $s$-equivalent. In particular, the space $M_{alg}^{s}\cong M^{s}$ embeds as a Zariski open set.

\begin{defi} The \emph{Gieseker compactification} 
$\overline{M}^{\GMC}$ of ${M}^{s}$ is defined as the Zariski closure
 $\overline{M^{s}}\subset M^{\Gss}$. 
\end{defi}

 There is no
linearised ample line bundle for which the GIT semistable points coincide with the slope semistable sheaves. Nevertheless, several authors have considered the line bundle 
$$\mathscr{L}_{n-1}=\lambda_{\mathscr{U}}(u_{n-1}(c,[\mathcal{O}_{H}]))\in \Pic(\Quot(\mathcal{H},c(E))
)\ .$$
 When $\dim X=2$, the restriction to $R^{\muss}$ was studied
by Le Potier and Li (\cite{LePotier:92, Li:93}, and see also \cite{HuybrechtsLehn:10}) and it was found to be
equivariantly semiample. 

For technical reasons explained in \cite[Sect.\ 1.2]{GrebToma:17},  the situation is more subtle in higher dimensions. One is more or less forced to consider the
weak normalisation  $(R^{\muss})^{wn}$. A \emph{weakly normal} complex space $\Zcal$ is one for which every locally defined
continuous function on $\Zcal$ which is holomorphic on restriction to the smooth
points $\Zcal_{\mathrm{reg}}$ of $\Zcal$ is in fact holomorphic. Every
 complex space $\Zcal$ has a weak normalisation $\Zcal^{wn}$ which is a reduced
weakly normal complex space homeomorphic to $\Zcal
$ (see \cite[Sect.\ 2.3]{GrebToma:17} for a summary of the relevant theory and for references). Henceforth, we  write $\Zcal=(R^{\muss})^{wn}$. 

We consider pull back $\widetilde{\mathscr{U}}\rightarrow X\times
\Zcal$ of the universal sheaf $\mathscr{U}\rightarrow
X\times R^{\muss}$ via the map $X\times $ $\Zcal \rightarrow X\times R^{\muss}$, 
and the line bundle $\lambda_{
\widetilde{\mathscr{U}}}(u_{n-1}(c,[\mathcal{O}_{H}]))\in \Pic(\Zcal)$,
which we continue to denote by $\mathscr{L}_{n-1}$. In  \cite[Thm.\ 3.6]{GrebToma:17} it is 
shown that this line bundle is equivariantly semiample with respect to the $
\SL(V)$-action on $\mathcal{Z}$ induced by the natural $\SL(V)$-action on $\Quot(\mathcal{H},c(E))$. Furthermore, the equivariant
section ring of $\mathscr{L}_{n-1}$ is finitely generated in degree 1
(perhaps after passing to  a sufficiently large power, see \cite[Prop.\ 4.3]{GrebToma:17}). In
other words, a sufficiently large power of $\mathscr{L}_{n-1}$ gives a map from 
$\Zcal$ to a projective space, and furthermore the images
of these maps stabilise. Formally imitating the GIT construction we may
define the space $M^{\muss}$ to be the image, that is: 
\begin{equation} \label{eqn:mu-moduli}
M^{\muss}={\rm Proj}\left( {\bigoplus }_{k}H^{0}(\Zcal,\mathscr L_{n-1}^{k})^{\SL(V)}\right) .
\end{equation}
We will write $\pi:\Zcal \rightarrow M^{\muss}$ for the natural surjective map. We note that
by a recent result \cite[Appendix A.4]{BuchdahlTelemanToma:17}, the
map $\pi $ is in fact a quotient map. 

The space $M^{\muss}$ is by construction a projective scheme. Furthermore, it comes with a distinguished ample line bundle $\mathcal{O}_{M^{\muss}}(1)$ and enjoys a  subtle universal property (for details, see 
\cite[Sect.\ 4]{GrebToma:17}). In particular, any flat family $\Escr \rightarrow X\times S$ of slope semistable sheaves 
over a weakly normal complex space $S$ yields a classifying morphism $\psi _{\Escr}:S\rightarrow M^{\mu
ss}$, so that $\psi _{\Escr}^{\ast }(\mathcal{O}_{M^{\mu
ss}}(1))=\lambda _{\Escr}(u_{n-1}(c,[\mathcal{O}_{H}]))^{\otimes N}$, for some power $N$. 
Again, when we refer to $M^{\muss}$ in the sequel we will always mean its analytification, which is a weakly normal complex space by {\cite[Thm.\ 4.7]{GrebToma:17}}.

A point in $M^{\muss}$ is represented by a quotient sheaf $q_{
\Ecal}:\mathcal{H}\rightarrow \Ecal\rightarrow 0$ with $\det\Ecal\simeq \Jcal$. Just as for
the Gieseker moduli space, one seeks a characterisation of the points of
this space in terms of properties of $\Ecal$. Clearly, quotients
corresponding to isomorphic sheaves $\Ecal_{1}$ and $\Ecal_{2}$ give
rise to the same point in $M^{\muss}$ since this means that $q_{
\Ecal_{1}}$ and $q_{\Ecal_{2}}$ belong to the same $\SL(V)$ orbit, and
therefore they cannot be separated by any $\SL(V)$ invariant section of the
line bundle $\mathscr{L}_{n-1}$. Hence, given a $\mu $-semistable sheaf $
\Ecal$ with the appropriate Chern classes we may speak unambiguously
of the point $\pi (q_{\Ecal})=[\Ecal]\in M^{\muss}$ corresponding to $\Ecal$. Conversely, every point of $M^{\muss}$ is represented by such
a sheaf $\Ecal.$ In fact, \cite[Thm.\ 5.10]{GrebToma:17} shows that the weak normalisation of  $M_{an}^{s} \cong M_{alg}^{s}\cong M^{s}$ embeds into $M^{\muss}$ as a Zariski open set, similar to the case of $M^{\Gss}$. This motivates the following definition.
\begin{defi}\label{def:defining_the_closure}
Let $\overline M^\mu$ denote the Zariski closure of $(M^s)^{wn}\subset  M^{\muss}$. We will call this the \emph{slope compactification}. 
\end{defi}

It is important to clarify when two sheaves correspond to the same
point in $M^{\muss}$. We have the following result from \cite{Li:93} ($\dim X=2$) and \cite{GrebToma:17} (higher dimensions).

\begin{prop}  \label{prop:GT}
 If $\{q_i:\Hcal\to \Ecal_i\} \in \Zcal$ , $i=1,2$ satisfy $[\Ecal_{1}]=[\Ecal_{2}]\in M^{\muss}$, then
$\ddual{\Gr(\Ecal_{1})}\cong \ddual{\Gr(\Ecal_{2})}$ 
and $\Ccal_{\Ecal_1}=\Ccal_{\Ecal_2}$, where  $q_i:\Hcal\to \Ecal_i$, $i=1,2$. If $\dim X=2$, the converse also holds.
\end{prop}

The proposition means that for $\dim X=2$, the points of $M^{\muss}$ are
in bijection with isomorphism classes of pairs $(\mathcal{E},\Ccal)$ consisting of a polystable
reflexive sheaf and a $0$-cycle satisfying
 $\ch_{2}(\mathcal{E})=\ch_{2}(E)+[\Ccal]$. When $\dim X\geq 3$, one needs more information
to characterise the geometry of $M^{\muss}$. Suppose $\Ecal_{1}$
and $\Ecal_{2}$ with  $q_{\Ecal_{1}},q_{\Ecal_{2}}\in 
\Quot(\mathcal{H}, c(E)$) satisfy $\ddual{\Gr(\Ecal_{1})}\cong \ddual{\Gr(
\Ecal_{2})}:=\mathcal{E}$ and $\Ccal_{\Ecal_{1}}=\Ccal_{
\Ecal_{2}}=\Ccal$. Consider the Quot scheme of torsion quotients $\Quot(
\mathcal{E},\tau _{\mathcal{E}}-\tau _{E})$. Then we have the following
\emph{Quot-to-Chow map}: 
$$
\chi :\Quot(\mathcal{E},\tau _{\mathcal{E}}-\tau _{E})\lra
 \Cscr_{n-2}(X)
\ , 
$$
given by taking a quotient $\mathcal{T}$ \ to its cycle $\Ccal_{\mathcal{T}}.$
The  conditions imposed on $\Ecal_{1}$ and $\Ecal_{2}$ imply that
the sheaves $\mathcal{T}_{\Ecal_{1}}$ and
 $\mathcal{T}_{\Ecal_{2}}$ are in $\Quot(\mathcal{E},\tau _{\mathcal{E}}-\tau _{E})$ and lie in
the same fibre of $\chi $.  

\begin{prop}[{\cite[Prop.\ 5.8]{GrebToma:17}}] \label{prop:fibres}
If $\mathcal{T}_{\Ecal_{1}}$ and $\mathcal{T}_{\Ecal_{2}}$ lie
in the same connected component of $\chi^{-1}(\Ccal)$, then $[\Ecal_{1}]=[\Ecal_2]$.
 In particular, for any $q_{\Ecal}\in \Quot(\mathcal{H},c(E))$ giving a point $[\Ecal]$ of $M^{\muss}$, there
are at most finitely many different points 
$\{[\Ecal_{i}\mathcal{]\}}_{i=l}^{k}$ of $M^{\muss}$ with $\ddual{(\Gr\Ecal)}=\ddual{(\Gr\Ecal
_{i})}$ and $\Ccal_{\Ecal}=\Ccal_{\Ecal_{i}}$.  
\end{prop}

\begin{rem}
The phenomenon of finite to oneness described above is a genuine one.
There are indeed examples where $M^{\muss}$ is a finite set of cardinality larger than one, and yet all representatives have graded objects with
the same double dual and cycle (see \cite[Example 3.3]{GrebRossToma:16}). In principle, it is also
possible for two sheaves to represent the same point even when their
associated torsion sheaves lie in different components of $\chi^{-1}(\Ccal)$.
The reason that this issue doesn't arise in the case $\dim X=2$
is that in this case the fibres $\chi^{-1}(\Ccal)$ are connected \cite{Li:93}, and even irreducible \cite{EllingsrudLehn:99}. 
\end{rem}

We finish this section by summarising the relationship between $M^{\Gss}$ and $M^{\muss}$.
Writing $\Zcal^{\Gss}$ for the space $(R^{\Gss})^{wn}$, the pullback
of $\widetilde{\mathscr{U}}$ to $X\times \Zcal^{\Gss}$  under the natural map $\Zcal^{\Gss} \to \Zcal$ induces a classifying map $\Zcal^{\Gss}\rightarrow M^{\muss}$ which factors through a map $(M^{\Gss})^{wn}\rightarrow M^{\muss}$, as $(M^{\Gss})^{wn}$ is a good
quotient of $(\Zcal^{\Gss})^{wn}$ (see \cite[Sect.\  5.3]{GrebToma:17}). Composing with the natural map $(\overline{M}^\GMC)^{wn} \to  (M^{\Gss})^{wn}$ induced by the inclusion, we obtain a map \begin{equation}\label{eq:Xi}
\Xi : (\overline{M}^\GMC)^{wn}\lra \overline{M}^\mu\
.\end{equation} On the level of points, this map can be expressed more explicitly in terms of sheaf theory as follows. The space $(M^{\Gss})^{wn}$ is in bijection with the set 
$\{\gr\Ecal\mid q_{\Ecal}\in \Zcal^{\Gss}\}$. Being GM-polystable, the sheaf $\gr
\Ecal$ is $\mu$-semistable (with the same Chern classes as $\mathcal{
E}$). Therefore, writing the quotient map as $\pi _{G}:\Zcal^{\Gss}\rightarrow (M^{\Gss})^{wn}$, we have $\Xi (\pi _{G}(q_{\Ecal}))=[\gr\Ecal]$. 

Together with \cite[Thm.\ 5.10]{GrebToma:17}, this shows that $\Xi$ is an isomorphism when restricted to $(M^{s})^{wn} \subset (\overline{M}^\GMC)^{wn}$. Since $\Xi$ is in particular continuous, this means it is also surjective, and hence in fact birational, when restricted to the respective compactifications.

\section{The gauge theoretic compactification} \label{sec:gauge-theory}
As we have seen, the moduli spaces $M^s$ and $M^\ast_\HYM$ are not  
compact in general, even after adjoining strictly semistable bundles and reducible connections, respectively. Two compactifications of $M^s$, $\overline M^\GMC$ and $\overline M^\mu$,   are obtained algebro-geometrically  by adding torsion free sheaves, and this was described briefly in Section \ref{sec:moduli}. In this section, we describe a third, gauge theoretic compactification associated to $M_\HYM^\ast$. 
In the case $\dim X=2$, this construction is due to 
 Donaldson \cite{Donaldson:86,DonaldsonKronheimer:90} based on work of Uhlenbeck \cite{Uhlenbeck:82b} and also Sedlacek \cite{Sedlacek:82}. A version of limiting Yang-Mills connections on higher dimensional manifolds appeared in Nakajima \cite{Nakajima:88}. 
Tian \cite{Tian:00} and Tian-Yang \cite{TianYang:02} generalise the  method and obtain key results on the structure of the singular sets. Since the construction will play a central role in this paper, and since some of the details presented here differ from those in the references above, we will provide a complete description of the compactification and some of its important features over the next few subsections.

\subsection{Uhlenbeck limits and admissible connections}
We begin with  a key definition (cf.\ \cite{BandoSiu:94} and \cite[Sect.\ 2.3]{Tian:00}).
\begin{defi}\label{def:admissible}
 Let  $(E,h)$ be a hermitian vector bundle on  a K\"{a}hler manifold $X$, not necessarily compact, with  $n=\dim X$. Then by an \emph{admissible connection} we mean a pair $(A,S)$ where 
 \begin{enumerate}
 \item  $S\subset X$ is a closed subset of locally finite Hausdorff $(2n-4)$-measure;
 \item $A$ is a smooth integrable unitary connection on $E\bigr|_{X\backslash S}$;
 \item $\int_{X\backslash S} |F_A|^2\, dvol_X < +\infty$;
 \item $\sup_{X\backslash S}  | \Lambda F_A| < +\infty$.
 \end{enumerate}
 An admissible connection is called \emph{admissible HYM} if there is a constant $\mu$ such that  $\sqrt{-1}\Lambda F_A=\mu\cdot \Ibold$ on $X\backslash S$.
 
 We will sometimes abuse terminology by saying "A is an admissible connection on (E,h)", when S is understood. 
 \end{defi}
 
 \begin{rem}
  The preceding definition is closely related to, but should not be confused with, the notion of an \emph{admissible Hermitian-Yang-Mills metric} in a given polystable reflexive sheaf $\mathcal{E}$ on $X$ as in \cite{BandoSiu:94}. We will elaborate on the relationship between the two notions in Sections \ref{subsect:idealconnections} and \ref{sec:map} below. 
 \end{rem}

The fundamental weak compactness result is the following.

\begin{thm}[Uhlenbeck, \cite{UhlenbeckPreprint}] \label{thm:uhlenbeck}
In the setup of Definition\ \ref{def:admissible}, let $(A_i, S_i)$ be a  sequence of admissible HYM connections with a uniform bound on the $L^2$-norm of curvature.
Assume there is a closed set $S_{\infty}'$ of finite Hausdorff $(2n-4)$-measure such that $S_i\to S_{\infty}'$ on compact sets in the Hausdorff sense.
 Then there is
\begin{enumerate}
\item a subsequence $($still denoted $A_{i}$$)$,
\item a closed subset $S_{\infty}\subset X$ of locally finite $(2n-4)$-Hausdorff measure containing $S_{\infty}'$,
\item a HYM connection $A_\infty$ on a hermitian bundle $E_\infty\to X\backslash S_{\infty}$, and
\item local isometries $E_\infty \simeq E$ on compact subsets of $X\backslash S_{\infty}$
\end{enumerate}
such that with respect to the local isometries, and  modulo unitary gauge equivalence, $A_{i}\to A_\infty$ in $
C_{loc}^{\infty }(X\backslash S_{\infty})$.
\end{thm}

We call the limiting connection $A_\infty$ an \emph{Uhlenbeck limit}.
 The set $S_{\infty}$, which we call the 
\emph{(analytic) singular set}, is the union of $S_{\infty}'$ with the set
$$
\bigcap_{\sigma_0\geq \sigma>0}\Bigl\{
	x\in X\backslash S_{\infty}' \mid \liminf_{i\to\infty} \sigma^{4-2n}\int_{B_\sigma(x)}|F_{A_i}|^2 \frac{\omega^n}{n!}\geq \varepsilon_0\Bigr\}\ ,
$$
where $\sigma_0$ and $\varepsilon_0$ are universal constants depending only on the geometry of $X$.

For the definition of a gauge theoretic compactification, it will be important that the Uhlenbeck limits
of smooth HYM connections be admissible HYM connections in the sense of Definition \ref{def:admissible}.  This will be true if $E_\infty$ is isometric to $E$ on the complement of the singular set. We formulate the precise statement as follows. 

\begin{prop} \label{prop:admissible} 
Let $(E,h)\to X$ be a hermitian vector bundle on a compact K\"ahler manifold. Then
any Uhlenbeck limit of a sequence of smooth HYM connections on $(E,h)$ is an admissible HYM connection. Moreover, the corresponding singular set is a holomorphic subvariety of codimension at least $2$.
\end{prop}
The remainder of this section is devoted to the description of the proof of Proposition \ref{prop:admissible}, which follows from work of Tian, Bando-Siu, and Tao-Tian.
 In the following, let
\begin{equation}\label{eqn:singset}
S(A_\infty):=\biggl\{ x\in X \,\biggr |\, \lim_{\sigma\downarrow
0}\sigma^{4-2n}\int_{B_{\sigma}(x)}\left\vert F_{A_\infty}\right\vert
^{2}\frac{\omega^n}{n!}\neq 0\ \biggr\}.
\end{equation}
This is a closed set with zero $(2n-4)$-dimensional  Hausdorff measure. 
\begin{thm}[{Tian, \cite[Thm.\ 4.3.3]{Tian:00}}] \label{thm:tian-singset}
 Let $A_i$ be a sequence of smooth integrable HYM connections (i.e.\ $S_i=\emptyset$) converging in the sense of Theorem \ref{thm:uhlenbeck} to an Uhlenbeck limit $(A_\infty,S_\infty)$. Then $S_\infty$ admits a decomposition into closed sets
$
S_\infty=S_b\cup \widetilde S_{\infty}
$,
where $S_b$ is a pure codimension $2$ holomorphic subvariety, 
and $\widetilde S_{\infty}$ has zero $(2n-4)$-dimensional Hausdorff measure.
\end{thm}
The set $S_b$ appearing in the above theorem is called the {\emph {blow-up locus}} of the sequence $A_i$.
We will also need two results on "removable" singularities for admissible connections.

\begin{thm}[{Bando--Siu, \cite{BandoSiu:94}}] \label{thm:bando-siu}
Let $\Ecal$ be a hermitian holomorphic vector bundle on $X\backslash S$, where $S$ has finite Hausdorff $(2n-4)$-measure, and suppose the  Chern connection of $\Ecal$
is admissible.
Then $\Ecal$  extends uniquely as a reflexive sheaf 
 $\widehat\Ecal$ on $X$. If the Chern connection of $\Ecal$ is
HYM, the hermitian structure extends smoothly to $X\backslash
\sing(\widehat\Ecal)$, the Chern connection is HYM there, and
$\Ecal$ is polystable.
\end{thm}

\begin{thm}[{Tao--Tian, \cite{TaoTian:04}}]  \label{thm:tao-tian}
Let $(A, S)$ be an admissible connection on the trivial bundle over a ball $B_{\sigma_0}(x)\subset X$. 
Suppose  $x\not \in S(A)$, where $S(A)$ is defined in \eqref{eqn:singset}. 
Then for $0<\sigma<\sigma_0$ sufficiently small  there is a
unitary gauge transformation $g$ on $B_\sigma(x)\backslash S$, such that $
g(A)$ extends to a smooth HYM connection  on $B_\sigma(x)$. 
\end{thm}

We note the following important fact, which is implicit in \cite[Thm.\ 4.3.3]{Tian:00}.
\begin{lemma}\label{lem:singset1}
In Theorem \ref{thm:tian-singset}, 
we may take $\widetilde S_{\infty}=S(A_\infty)$; i.e., the singular set decomposes into closed sets as follows: $S_\infty=S_b\cup S(A_\infty)$.
\end{lemma}
\begin{proof}
Let $\mu_i=\left| F_{A_i}\right|^2 \omega^n/n!$ be the Yang-Mills energy densities. Then we may assume without loss of generality that we have a convergence of finite Radon measures: 
\begin{equation} \label{eqn:measures}
\mu_i \to \mu=\left| F_{A_\infty}\right|^2 \frac{\omega^n}{n!} +\nu\ ,
\end{equation}
 where the measure  $\nu$ is absolutely continuous with respect to  $H^{2n-4}_{S_\infty}$, where $H^{2n-4}_{S_{\infty}}(A):=H^{2n-4}(S_{\infty}\cap A)$ is the $(2n-4)$-dimensional Hausdorff measure on $S_\infty$. We will write 
$\nu(x) = \Theta(x)\cdot \Hcal^{2n-4}_{S_{\infty}}$,
the density function $\Theta$ being
$$
\Theta(x):=\lim_{\sigma\downarrow 0}\sigma^{4-2n}\mu(B_\sigma(x))\ ,
$$
 where the limit exists due to the monotonicity formula (see \cite[Lemma 3.1.4 (a)]{Tian:00}).
 Then $x\in S_{\infty}$ if and only if $\Theta(x)\neq 0$ (see \cite[p.\ 222]{Tian:00}).
On the other hand, if $x\not \in S_b$, then for $\sigma>0$ sufficiently small, $\Hcal^{2n-4}(S_{\infty}\cap B_\sigma(x))=0$, so $\nu(S_{\infty}\cap B_\sigma(x))=0$. If $x\not\in S(A_\infty)$, then by \eqref{eqn:singset} and \eqref{eqn:measures} we must then also have $\Theta(x)=0$. The result follows.
\end{proof}

Let $\Ecal_\infty$ be the reflexive sheaf on $X$ extending the holomorphic bundle $(E_\infty, \dbar_{A_\infty})$, which comes from  Theorem \ref{thm:bando-siu}. The following
result is a restatement of
  \cite[Thm.\ 1.4]{TianYang:02}. For completeness and for the convenience of the reader, we provide a condensed proof of this result.  

\begin{lemma} \label{lem:singset2}
The equality $S(A_\infty)= \sing\mathcal{E}_{\infty }$ holds. In particular,  $S_\infty$ is a holomorphic subvariety.
\end{lemma}

\begin{proof}
Fix $x\in S_{\infty}$. 
If $x\not\in S(A_\infty)$, then there is $r_0>0$ such that $S_{\infty}\cap B_{r_0}(x)\subset S_b$. As $S_b$ is analytic by Theorem \ref{thm:tian-singset}, $S_\infty \cap B_{r_0}(x)$ has a neighbourhood that admits a deformation retraction to $S_\infty \cap B_{r_0}(x)$. It follows (cf.\ \cite[Ch.\ 7]{Wehrheim:04}) that $E_\infty\bigr|_{B_{r_0}(x)}$ is isometric to $E\bigr|_{B_{\sigma_0}(x)}$. In particular, $A_\infty$ is identified with an admissible connection on a trivial bundle on $B_{r_0}(x)$. We may therefore apply Theorem \ref{thm:tao-tian} to conclude that the holomorphic bundle $(E,\dbar_{A_\infty})$ extends to a holomorphic bundle on $B_{r}(x)$ for some $r>0$. Since the reflexive extension $\Ecal_\infty $ is unique, we conclude that  $x\not\in \sing(\Ecal_\infty)$. Conversely, if $x\not\in \sing(\Ecal_\infty)$, then the Chern connection is smooth at $x$, and it follows that $x\not\in S(A_\infty)$. 
\end{proof}

We can now give the

\begin{proof}[Proof of Proposition \ref{prop:admissible}] 
By Lemma \ref{lem:singset2}, $S_{\infty}$ is a union of irreducible subvarieties, and so there is an exhaustion of $X\backslash S_{\infty}$ by compact subsets which are deformation retracts of $X\backslash S_{\infty}$. It follows that in the patching argument for the construction of an Uhlenbeck limit one can find global gauges, and so the bundle  $E_\infty$ obtained in Theorem \ref{thm:uhlenbeck} is isometric to $(E,h)$ on $X\backslash S_{\infty}$.
 We refer to \cite{Wehrheim:04} for further details.
\end{proof}

\subsection{Analytic multiplicities}
\label{sec:analytic-multiplicities}
In this section  we discuss the multiplicities that are associated to the irreducible components of the blow-up locus. For an admissible connection $(A,S)$ on $(E,h)\to X$, the result in  \cite[Prop.\ 2.3.1]{Tian:00} states that integration against the form
\begin{equation}\label{eq:Chern_character_of_connection}
\ch_2(A) =-\frac{1}{8\pi^2} \tr(F_A\wedge F_A)
\end{equation}
on $X\backslash S$ defines a closed $(2,2)$-current on $X$.

\begin{thm}[{Tian,  \cite[Thm.\ 4.3.3]{Tian:00}}] \label{thm:tian}
Suppose $A_{i}$ is a sequence of smooth HYM connections on $E$,
and $A_{i}$ has Uhlenbeck limit $A$ with blow-up locus $S_b$. Then to each irreducible
codimension $2$ component $Z^{an}_{j}\subset S_b$ there is a
positive integer $m^{an}_{j}$ such that
\begin{equation} \label{eqn:current}
\ch_2(A_i)\lra \ch_2(A_\infty)-\sum_{j}m^{an}_{j}{Z^{an}_{j}}
\end{equation}
in the sense of currents.
\end{thm}
We define the $(n-2)$-cycle associated to the sequence $\{A_i\}$ by $\Ccal^{an}=\sum_{j}m^{an}_{j}{Z^{an}_{j}}$, so that $|\Ccal|=S_b$.
From \eqref{eqn:current} we see that $\ch_2(A_\infty)$ represents $\ch_2(E)+[\Ccal^{an}]$ in $H^4(X,\QBbb)$. In this context we will also refer to the triple $(A_\infty,\Ccal^{an},S(A_\infty))$ as an \emph {Uhlenbeck limit} of ${A_i}$. 

To elaborate on  the origin of the integer multiplicities above,
we
 recall two slicing lemmas from \cite{SibleyWentworth:15} which will be needed later on.   We begin with the following definition.

\begin{defi} Let $z$ be a
smooth point of a codimension $2$ subvariety $Z\subset X$. We say that $\Sigma$ is a {\em transverse slice} to $Z$ at $z$ if $\Sigma \cap Z=\{z\}$
and $\Sigma $ is the restriction of a linear subspace
 $\mathbb{C}^{2}\hookrightarrow \mathbb{C}^{n}$ to some coordinate ball centred at $z$
that is transverse to $T_{z}Z$ at the origin.
\end{defi}

\begin{lemma}[{\cite[eq.\ (4.1)]{SibleyWentworth:15}}]\label{lem:slicing1} Let  $T$ be a smooth,
closed $(2,2)$ form satisfying the equation 
$
T=m_Z Z+dd^{c}\Psi$, 
where $\Psi $ is a $(1,1)$-current, smooth away from $Z$, $m_Z Z$ is the
current of integration over the nonsingular points of $Z$ with multiplicity $
m_Z$, and the equation is taken in the sense of distributions. Then for a
transverse slice, 
\begin{equation*}
m_Z=\int_{\Sigma }T-\int_{\partial \Sigma }d^{c}\Psi\ .
\end{equation*}
\end{lemma}

The next result shows that the analytic
multiplicities may also be calculated by restricting to transverse slices.

\begin{lemma}[{\cite[Lemma 4.1]{SibleyWentworth:15}}]\label{lem:slicing2}
Let $A_{i}$ be a sequence of Hermitian-Yang-Mills connections on a fixed
hermitian vector bundle $E\rightarrow X$, with Uhlenbeck limit $(A_{\infty },\Ccal^{an},S(A_\infty))$ and corresponding blow-up locus $S_{b},$ and let $Z$ be an irreducible codimension 2 subvariety of $X$. For a transverse slice $
\Sigma $ at a generic smooth point $z\in Z$, we have: 
\begin{equation*}
\lim_{i\rightarrow \infty }\frac{1}{8\pi ^{2}}\int_{\Sigma
}\left\{\tr(F_{A_{i}}\wedge F_{A_{i}})-\tr(F_{A_{\infty
}}\wedge F_{A_{\infty }})\right\} = m_j^{an}
\end{equation*}
if $Z=Z_j^{an}\subset S_{b}^{an}$, and 
\begin{equation*}
\lim_{i\rightarrow \infty }\frac{1}{8\pi ^{2}}\int_{\Sigma }\left\{ \tr(F_{A_{i}}\wedge F_{A_{i}})-\tr(F_{A_{\infty }}\wedge
F_{A_{\infty }})\right\} =0
\end{equation*}
otherwise.
\end{lemma}

It will be useful to have a more explicit description of the multiplicity. First, we will need
the next result, which is  an elementary computation that we omit.
\begin{lemma}\label{lem:Chern-Simons}
Let $E\to \Sigma$ be a bundle over a smooth $4$-manifold $\Sigma$ with boundary $\partial \Sigma$. For connections $A,B$ on $E\to U$, $U\subset \Sigma$ open, write $
\CS(A,B)$ for the Chern-Simons 3-form on $U$ satisfying
\[
d\CS(A,B)=\frac{1}{8\pi^2}\tr\{(F_A\wedge F_A)-(F_{B}\wedge
F_{B})\}\ .
\] 
Then we have the following:
\begin{enumerate}
	\item
For $A_{0}$, $A_{1}$, $A_{2}$
 be smooth connections on $E\to U$, there is a $2$-form $\Omega $
so that on $U$,
\begin{equation*}
\CS(A_{2},A_{0})=\CS(A_{1},A_{0})+\CS(A_{2},A_{1})+d\Omega .
\end{equation*}
\item
If $A$ is a connection on $E\to \Sigma$ and $g$ is a smooth gauge transformation defined in a neighbourhood $U$ of $\partial \Sigma$, 
then
$$
\frac{1}{8\pi^2}\int_{\partial\Sigma} \CS(A,g(A)) =
\frac{1}{24\pi^2}\int_{\partial\Sigma} \tr((g^{-1}dg)^3)=: \deg(g)\in \ZBbb\ ,
$$
where  the right hand side 
is the evaluation on $\partial \Sigma$ of the pullback by $g$ of the Cartan $3$-form generating $H^3(\SU(r), \ZBbb)$. 
\item
If $E_{\infty}\rightarrow \Sigma $ is a smooth vector bundle and $\phi_{\infty}
:E|_{\partial \Sigma }\rightarrow E_{\infty}|_{\partial \Sigma }$ is a smooth
isomorphism, and $A_{\infty},B_{\infty}$ are connections on $E_{\infty}$, then
\begin{equation*}
\int_{\partial \Sigma }\CS(\phi_{\infty} ^{\ast }A_{\infty},\phi_{\infty} ^{\ast
}B_{\infty})=\int_{\partial \Sigma }\CS(A_{\infty},B_{\infty}).
\end{equation*}
In particular, for $g$ as in $(2)$ 
\begin{equation*}
\int_{\partial \Sigma }CS(g(A),g(B))=\int_{\partial \Sigma }CS(A,B)\ .
\end{equation*}
\item
If $g_{1}$,$g_{2}$ are gauge transformations as in $(2)$ then
\begin{equation}
 \deg(g_{1}g_{2})=\deg g_{1}+\deg g_{2}\ .  
\end{equation}
\end{enumerate}
\end{lemma}

Returning to the situation in Lemma \ref{lem:slicing2}, let $z\in Z$.
 Without loss of generality assume $z\not\in S(A_\infty)$. Then
by Proposition \ref{prop:admissible} and Theorem \ref{thm:tao-tian},
   $A_\infty$ locally extends to a  connection on a bundle $E_\infty$ that is isometric to $E$ away from $Z$. Along the slice $\Sigma $ we choose local unitary frames $\ebold$ and $\ebold_\infty$ of $E$ and $E_\infty$, respectively. 
Let $D_E$ and $D_{E_\infty}$ denote the connections on $E$ and $E_\infty$ that make the frames $\ebold$ and $\ebold_\infty$ parallel. 
We have:
\begin{align}
\begin{split} \label{eqn:cs}
\frac{1}{8\pi^2}\int_\Sigma  \tr(F_{A_{i}}\wedge F_{A_{i}}) &= \frac{1}{8\pi^2}\int_{\partial \Sigma} \CS(A_i, D_E) \\
\frac{1}{8\pi^2}\int_\Sigma  \tr(F_{A_{\infty}}\wedge F_{A_{\infty}}) &= \frac{1}{8\pi^2}\int_{\partial \Sigma} \CS(A_\infty,D_{E_\infty})
\end{split}
\end{align}
Away from $z\in \Sigma$ there is a isometry $\phi_{\infty} : E\to E_\infty$. By Theorem \ref{thm:uhlenbeck} there are gauge transformations $g_i$ defined away from $z\in \Sigma$ such that $g_i(A_i)\to \phi_{\infty}^\ast A_\infty$. Furthermore,  since $\phi_{\infty}^\ast D_\infty$ and $D_E$ are flat connections on $E$ over a simply connected manifold $\Sigma\backslash\{z\}$, there is a gauge transformation $h$ on $\Sigma\backslash\{z\}$ such that $h(\phi_{\infty}^\ast D_{E_\infty})= D_E$.  Set $h_i=hg_i$.
  Now, using \eqref{eqn:cs} and  Lemma \ref{lem:Chern-Simons},
\begin{align*}
\frac{1}{8\pi ^{2}}&\int_{\Sigma
}\Bigl\{\tr(F_{A_{i}}\wedge F_{A_{i}})-\tr(F_{A_{\infty
}}\wedge F_{A_{\infty }})\Bigr\}
=
\int_{\partial \Sigma} \CS(A_i, D_E)-\CS(A_\infty, D_{E_\infty}) \\
&=
\int_{\partial \Sigma} \CS(g_i(A_i), g_i(D_E))-\CS(\phi_{\infty}^\ast A_\infty, \phi_{\infty}^\ast D_{E_\infty}) \\
&=
\int_{\partial \Sigma} \CS(g_i(A_i), \phi_{\infty}^\ast A_\infty)
+\CS(\phi_{\infty}^\ast A_\infty, g_i(D_E))
-\CS(\phi_{\infty}^\ast A_\infty, \phi_{\infty}^\ast D_{E_\infty}) \\
&=
\int_{\partial \Sigma} \CS(g_i(A_i), \phi_{\infty}^\ast A_\infty)
+\CS(\phi_{\infty}^\ast D_{E_\infty},g_i(D_E)) \\
&=
\int_{\partial \Sigma} \CS(g_i(A_i), \phi_{\infty}^\ast A_\infty)
+\CS(D_E, h_i(D_E)) \ .
\end{align*}
The first term on the right hand side above vanishes as $i\to\infty$, whereas by part (2) of Lemma \ref{lem:Chern-Simons}, $\deg (h_i)$ is an integer which must stabilise to give the multiplicity  for $i$ sufficiently large.

\subsection{Ideal HYM connections}\label{subsect:idealconnections}
Observe from Lemma \ref{lem:singset1} that limits $A_i\to A_{\infty}$ of smooth HYM connections satisfy the following property: any removable singularity of $A_\infty$ lies in the codimension $2$ cycle associated to $\{A_i\}$.
This result motivates the next definition, which is slightly more restrictive than the one used in \cite{Tian:00}, but well-adapted to our purposes.

\begin{defi} \label{def:ideal-connection}
An {\em ideal HYM connection}  is a triple  $(A,\Ccal, S(A))$ satisfying the following conditions:
\begin{enumerate}
\item $\Ccal\in\Cscr_{n-2}(X)$;
\item the pair $(A, |\Ccal| \cup S(A))$ is an  admissible HYM connection on the hermitian vector bundle $(E,h)\to X$, where $S(A)$ is 
given by eq.\ \eqref{eqn:singset};
\item $[\ch_2(A)]=\ch_2(E)+[\Ccal]$, in $H^4(X,\QBbb)$;
\item $A$ induces the connection $a_J$ on $J=\det E$.
\end{enumerate} 
Moreover, we say that ideal connections $(A_1, \Ccal_1,S(A_1))$ and $(A_2, \Ccal_2,S(A_2))$  are \emph{gauge equivalent} if $\Ccal_1=\Ccal_2$ as cycles (so in particular $|\Ccal_1|=|\Ccal_2|=:Z$), and if there is a smooth unitary gauge transformation $g$ on $X\backslash (Z\cup S(A_1)\cup S(A_2))$ such that $g(A_1)=A_2$.
\end{defi}

Given an ideal HYM connection $(A,\Ccal,S(A))$ on $(E,h)$, by Theorem 
\ref{thm:bando-siu} and Lemma \ref{lem:singset2} there is a polystable reflexive extension $\mathcal{E}\rightarrow X$ of the holomorphic  bundle 
$(E,\dbar_E)$ defined in the complement of $Z\cup S(A)$, with
$S(A)=\sing(\mathcal{E)}$. Conversely, let
$(\mathcal{E},\Ccal)$ be a pair consisting of a polystable
reflexive sheaf and a codimension $2$ holomorphic cycle so
that $\Ecal$ is smoothly isomorphic to $E$ on $X\backslash
(Z\cup \sing(\mathcal{E}))$, where we put $Z=|\Ccal|$. By
\cite[Thm.\ 3]{BandoSiu:94}, $\Ecal$ admits an admissible
Hermitian-Einstein metric $h_{\Ecal}$ that is unique up to a
constant. Let $g$ be a complex gauge transformation on
$X\backslash (Z\cup \sing(\mathcal{E}))$  such that
$g^{*}(h_{\Ecal})=h$. Then a simple calculation shows (see the
discussion in Section \ref{sec:map} below), that if we write
$\dbar_{E}$ for the holomorphic structure on $E$ associated to
$\Ecal$, the Chern connection $A=(\dbar_{g^{*}(E)},h)$ gives an admissible HYM connection $(A,Z\cup \sing(\mathcal{E)})$ on $(E,h)$ in the sense of Definition \ref{def:admissible}. By construction, the sheaf $\Ecal$ is the reflexive extension associated to this admissible connection, and in particular $S(A)=\sing(\mathcal{E)}$.
The current $\ch_2(A)$ defined in eq.\ \eqref{eq:Chern_character_of_connection} is closed and represents $\ch_2(\mathcal{E})$, see the proof of \cite[Prop.\ 3.3]{SibleyWentworth:15}. Hence, if we assume furthermore that $\ch_{2}(\mathcal{E})=\ch_{2}(E)+[\Ccal]$, the triple $(A,\Ccal,S(A))$ is an ideal connection.

By construction, if $(A_{1},\Ccal_{1,}S(A_{1}))$ and $(A_{2},\Ccal_{2},S(A_{2}))$ are
gauge equivalent, their associated holomorphic bundles are isomorphic  
away from
the analytic set $Z\cup S(A_{1})\cup S(A_{2})$, which has codimension $\geq 2$ in $X$. Hence, the respective reflexive extensions $\mathcal{E}_{1}$ and $\mathcal{E}_{2}$ coming from Theorem \ref{thm:bando-siu} are isomorphic, and in particular we conclude with the help of Lemma \ref{lem:singset2} that $S(A_{1})=S(A_{2})$. Conversely, 
if $\Ecal_1$ and $\Ecal_2$ admit admissible HYM metrics in the sense of Bando--Siu, and
 if $\mathcal{E}_{1}\cong \mathcal{E}_{2}$ and $\Ccal_{1}=\Ccal_{2}$,  then   $(A_{1},\Ccal_{1},S(A_{1}))$ and 
$(A_{2},\Ccal_{2},S(A_{2}))$ are gauge equivalent. We may therefore regard an
isomorphism class of ideal connections as equivalent to  an isomorphism class of pairs
$(\mathcal{E},\Ccal)$, where $\Ecal$ is a polystable reflexive sheaf whose underlying $C^\infty$ vector bundle on the complement of $|\Ccal|\cup \sing{\Ecal}$ is isomorphic to $E$. This is compatible with the notation introduced in Definition\ \ref{defi:cycle,couple}. We will use this description in Section \ref{sect:equivalence_relation}.  

Gauge equivalence defines an equivalence relation $\sim $ on the space of
ideal HYM connections on $(E,h)$. We define the \emph{moduli set of ideal HYM
connections} on $(E,h)$ to be
\begin{equation} \label{eqn:ideal-moduli}
\widehat{M}_{\HYM}=\widehat{M}_{\HYM}(E,h, a_J):=\{\text{ideal HYM connections on } (E,h)\}/\sim \ .
\end{equation}
Notice that there is a natural inclusion $M_{\HYM}^{\ast }\subset \widehat{M}_{\HYM}$. 

In order to obtain a compactification, we will apply Theorem \ref{thm:uhlenbeck} to sequences of ideal HYM connections as well. Given any $[(\mathcal{E},\Ccal)]\in \widehat{M}_{\HYM}$, observe that owing to polystability of $\mathcal{E}$ the
Bogomolov inequality, see for example \cite[Cor.\ 3]{BandoSiu:94}, applies to give:
 \begin{equation} \label{eqn:bogomolov}
\int_{X}\left( c_{2}(\mathcal{E})-\frac{r-1}{2r}c_{1}^{2}(\mathcal{E}
)\right) \wedge \omega ^{n-2}\geq 0\ .
\end{equation}
Using the relations $\ch_{2}(\mathcal{E})=\ch_{2}(E)+[\Ccal]$ and $c_{1}(\mathcal{E})=c_{1}(E)=c_1(J)$, we obtain a bound
\[
\deg(\Ccal) =\sum_{i}n_{i}\int_{Z_{i}}\frac{\omega ^{n-2}}{(n-2)!}\leq C
\]
that is independent of $(\mathcal{E},\Ccal)$. This means that if $
[(A_{i},\Ccal_{i},S(A_{i}))]$ is sequence in $\widehat{M}_{\HYM}$, there is a
uniform bound $\deg\Ccal_{i}\leq C$, and so $\Ccal_{i}$ converges subsequentially as
cycles by Theorem \ref{thm:cycle space}.

Note that since $X$ is  compact the sets $S(A_i)$ always converge subsequentially in the   Hausdorff sense to some compact subset $S^{\prime}_{\infty}$ of $X$. On the other hand, 
to maintain the structure of an ideal connection we want the
 $S(A_{i})$ to converge in the union of cycle spaces $\cup_{k=0}^{n-3}\mathscr C_{k}(X)$. 
 This is guaranteed by the next lemma, which follows\footnote{We are grateful to Carlos Simpson for suggesting that Lemma \ref{lem:bounded-admissible} should be a consequence of Maruyama's result.} from   \cite[Main Theorem]{Maruyama:81} together with Corollary \ref{Cor:singsetboundedness}.

\begin{lemma} \label{lem:bounded-admissible}
There is a constant $K$ with the following significance.
For any ideal HYM connection  $(A, \Ccal, S(A))$ on $(E,h)$ and each $k=0, 1, \ldots, n-3$, the $k$-dimensional stratum of $S(A)$ has degree less than or equal to $K$, when considered as a cycle in $X$.
\end{lemma}

\begin{proof}
A consequence of the statement in \cite{Maruyama:81} is that a set of isomorphism classes of slope semistable reflexive sheaves of fixed rank, $c_1$, and $c_2\cup [\omega]^{n-2}$,  is bounded. For sheaves $\mathcal{E}$ associated to points in $\widehat M_\HYM$ as above,
 the rank and $c_1$ are fixed, whereas
$c_2(\Ecal)\cup [\omega]^{n-2}$
is bounded: from below by \eqref{eqn:bogomolov}, and from above by 
$c_2(E)\cup [\omega]^{n-2}$. The aforementioned result therefore applies to our family of reflexive sheaves. We may then use Corollary \ref{Cor:singsetboundedness} applied to the sets $\sing(\Ecal)=\cup _{k=3}^{n}\sing_{n-k}(\mathcal{E)}$. Since these algebraic singularity sets coincide with the analytic singularity sets $S(A)$ by Lemma \ref{lem:singset2}, the claim follows.
\end{proof} 

The following  is the main result of this subsection.

\begin{thm} \label{thm:ideal-convergence}
Let $(A_i, \Ccal_i, S(A_i))\in \widehat M_\HYM$.
Then there is a subsequence $($also denoted by $\{i\}$), and an ideal HYM connection $(A_\infty, \Ccal_\infty, S(A_\infty))$
such that $\Ccal_i$ converges to a subcycle of $\Ccal_\infty$, and (up to gauge transformations) $A_i\to A_\infty$ in $C^\infty_{loc}$ on $X\backslash (Z_\infty\cup S(A_\infty))$ where $Z_\infty:=|\Ccal_\infty|$. Moreover, 
\begin{equation} \label{eqn:currents-converge}
\ch_2(A_i)-\Ccal_i\lra \ch_2(A_\infty)-\Ccal_\infty
\end{equation}
in the mass norm; in particular, also
in the sense of currents.
\end{thm}

The rest of this section is devoted to the proof of Theorem \ref{thm:ideal-convergence}.  We proceed in several steps. 

\medskip\noindent{\bf Step 1.} 
By Lemma  \ref{lem:bounded-admissible} and Theorem \ref{thm:cycle space}, we may first extract a subsequence (also denoted $\{i\}$), such that
\begin{enumerate}
\item the Yang-Mills densities converge weakly $|F_{A_{i}}|^2 dvol_\omega \to \mu_\infty$ to a Radon measure $\mu_\infty$;
\item there is a cycle $\Ccal'_\infty$ with $|\Ccal'_\infty|=Z_\infty'$ such that  $\Ccal_i\to \Ccal'_\infty$ as cycles;
\item we have Hausdorff convergence $S(A_i)\to S'_\infty$, where $S'_\infty$ is a subvariety of codimension at least $3$.
\end{enumerate}

\medskip\noindent{\bf Step 2.}
We may apply Theorem \ref{thm:uhlenbeck} to see that the following holds: There is a  closed subset $\widetilde S_{\infty}\subset X\backslash (Z'_\infty \cup S'_\infty)$ of locally finite Hausdorff $(2n-4)$-measure, and a  smooth HYM connection $A_\infty$ on $X\backslash (Z'_\infty \cup S'_\infty\cup\widetilde S_{\infty})$, such that $A_i\to A_\infty$ up to gauge, smoothly in the local $C^\infty$ topology. Moreover, by Theorem \ref{thm:tian-singset}, we may write $\widetilde S_{\infty}=\widetilde Z_{\infty}\cup S(A_\infty)$, where $S(A_\infty)$ is defined as in \eqref{eqn:singset}, and $\widetilde Z_{\infty}$ is a closed pure $(n-2)$-dimensional analytic  subvariety of $X\backslash (Z'_\infty \cup S'_\infty)$.
 
\medskip\noindent{\bf Step 3.} We claim that $\widetilde Z_{\infty}\subset X\backslash (Z'_\infty \cup S'_\infty)$ extends to $X$. We shall use the Bishop-Stoll removable singularities theorem. Choose a point $p\in Z'_\infty$, $\sigma>0$, such that $B_{2\sigma}(p)\subset X\backslash (S'_\infty\cup S(A_\infty))$. Then for $i$ sufficiently large, the connections $A_i$ are smooth  on $B_{2\sigma}(p)\backslash Z_i$ and extend to smooth connections $\widehat A_i$ on $B_\sigma(p)$ by Theorem \ref{thm:tao-tian}. Hence, applying Theorem \ref{thm:uhlenbeck} we conclude that after passing to a subsequence and up to gauge $\widehat A_i$ converges to a limit $\widehat A_\infty$ in the local $C^\infty$ topology on $B_\sigma(p)$ away from a singular set $\widehat Z_\infty\cup S(\widehat A_\infty)$ that is an analytic subvariety of $B_\sigma(p)$. Since $\widehat A_\infty$ and $A_\infty$ agree up to gauge off a codimension $2$ set, they agree up to gauge on their common domain of definition. By our choice of $p$ it  follows that $S(\widehat A_\infty)\cap B_\sigma(p)=\emptyset$, and $\widehat Z_\infty\cap B_\sigma(p)$ and $\widetilde Z_{\infty}\cap B_\sigma(p)$ agree on the complement of $Z'_\infty \cap B_\sigma(p)$. In particular, the intersection ${\rm cl}(\widetilde Z_{\infty})\cap Z'_\infty \cap B_{\sigma}(p)$ has codimension at least $3$. Since $S'_\infty\cup S(A_\infty)$ also has codimension at least $3$, the same is true for ${\rm cl}(\widetilde Z_{\infty})\cap (Z'_\infty\cup S'_\infty)$. It now follows from \cite[Lemma 9]{Bishop:64}, that $\widetilde Z_{\infty}$ extends as a holomorphic $(n-2)$-dimensional
  subvariety $Z''_\infty$ on $X$.

\medskip\noindent{\bf Step 4.} Set $Z_\infty=Z'_\infty\cup Z''_\infty$ as pure $(n-2)$-dimensional subvarieties. 
We have local $C^\infty $ convergence $A_i\to A_\infty$ on $X\backslash (Z_\infty\cup S(A_\infty)\cup S'_\infty)$. Hence, by the same argument as in the proof of Proposition \ref{prop:admissible}, $A_\infty$ is an admissible HYM connection. To prove that we have an ideal HYM connection, we need to show that $ S'_\infty\subset Z_\infty\cup S(A_\infty)$, and we have to assign multiplicities to the components of $Z_\infty$. The latter part will be discussed in Step 5 below. The former statement is a consequence of the next lemma, which is also absolutely crucial for the argument in Section \ref{sect:equivalence_relation}.
  
  \begin{lemma} \label{lem:sing-limit}
  In the situation above, $ S'_\infty\subset S(A_\infty)\cup Z_\infty$.
  \end{lemma}
\begin{proof}
Consider the reflexive sheaves $\Ecal_{\infty}$, $\Ecal_i$ obtained by extending the holomorphic bundles $(E, \dbar_{A_{\infty}})$ and $(E, \dbar_{A_i})$, respectively. By Lemma \ref{lem:singset2}, the loci where the $\Ecal_i$ are not locally free are precisely the codimension at least 3 singular sets $S(A_i)$. 

Suppose that $p\in S'_\infty$, $p\not\in S(A_\infty)\cup Z_\infty$.
Then we can find $x_i\in S(A_i)$ converging to $p$. 
 We may also  find a coordinate ball $B_{2\sigma}(p)$ whose closure lies in the complement of 
$S(A_\infty)\cup Z_\infty$, and such that $(\Ecal_\infty)|_{B_{2\sigma}(p)}$ is a holomorphically trivial vector bundle. Notice that by definition of $Z_\infty$,  $Z_i\cap B_{2\sigma}=\emptyset$ for sufficiently large $i$. 
Let $\Delta_\delta^{(k)}$ denote the
polydisk of radius $(\delta, \dots, \delta)$ in $\CBbb^k$ and let $B^{(k)}_\sigma(0)$ the ball of radius $\sigma$ in $\CBbb^k$. 
We may find an annular region $U:=U(\sigma, \varepsilon; \delta)\subset B_{2\sigma}(p)$ centred about $p$ given in coordinates by 
$$(B_{\sigma}^{(3)}(0)\backslash \overline B_{\varepsilon}^{(3)}(0)) \times \Delta_\delta^{(n-3)}\ ,$$
 such that $S_\infty'\cap U=\emptyset$, and hence also $S(A_i)\cap U=\emptyset$,  for sufficiently large $i$.
 For the purposes of this proof, we choose the euclidean product metric on $U$, and the standard hermitian structure on $\Ecal_\infty$ with respect to its trivialisation over $U$, scaled
 by an appropriate weight as explained in \cite{Shaw:85}, whose results we will apply later on in the proof. The notation $\Vert\cdot \Vert$ will refer to the $L^2$-norm of bundle valued forms on $U$ with respect to these metrics, and $\ast$ will denote the associated hermitian conjugate. 

Change notation slightly, and let $\dbar_{A_i}$ denote the $\dbar$-operator on
$\End E$ inducing the holomorphic structure $\Ecal^\vee_\infty\otimes \Ecal_i$.
Writing $\dbar_{A_i}=\dbar_{A_\infty}+a_i$ for some $\End E$-valued $(0,1)$-form $a_i$, we have $a_i\to 0$ smoothly on $U$. We claim that for $i$ sufficiently large there are smooth endomorphisms $u_i$ of $E$ and a constant $C$ such that\begin{enumerate}
\item $\dbar_{A_i}u_i=-a_i\,$;
\item $\Vert u_i\Vert \leq C \Vert a_i\Vert\, $.
	\end{enumerate}
In order to prove this we verify the basic estimate; namely, that there exists a constant $C$ such that
\begin{equation} \label{eqn:basic-estimate}
\Vert \phi\Vert\leq C\left( \Vert \dbar_{A_i}\phi\Vert +\Vert \dbar^\ast_{A_i}\phi\Vert \right)
\end{equation}
for all $\phi$,  a smooth $(0,1)$-form with values in $\End E$ satisfying the $\dbar$-Neumann boundary  conditions (cf.\ \cite{Hormanader:65,FollandKohn:72}).
We first note that \eqref{eqn:basic-estimate} is valid if $A_i$ is replaced by $A_\infty$:
\begin{equation} \label{eqn:basic-estimate0}
\Vert \phi\Vert\leq C\left( \Vert \dbar_{A_\infty}\phi\Vert +\Vert \dbar^\ast_{A_\infty}\phi\Vert \right)\ .
\end{equation}
Indeed, since the bundle and metrics are trivialized, the problem reduces to the scalar $\dbar$-estimate, and hence \eqref{eqn:basic-estimate0} follows using standard results such as \cite[Lemma 2.1]{ChakrabartiShaw:11} from
\cite[Thm.\ 2.2]{Shaw:10} in the case $\dim X=3$ and 
 \cite[Thm.\ 1.1, Cor.\ 6.3]{ChakrabartiShaw:11} in higher dimensions\footnote{The authors thank Mei-Chi Shaw for pointing us towards these references.}. 
On the other hand, there is some numerical constant $c_0$ independent of $A_i$ such that
\begin{align*}
\Vert \dbar_{A_i}\phi\Vert&\geq \Vert \dbar_{A_\infty}\phi\Vert-c_0\sup|a_i|\Vert\phi\Vert\ , \\
\Vert \dbar^\ast_{A_i}\phi\Vert&\geq \Vert \dbar^\ast_{A_\infty}\phi\Vert-c_0\sup|a_i|\Vert\phi\Vert\ .
\end{align*}
Since $\sup|a_i|$ is arbitrarily small on $U$ for sufficiently large $i$, the estimate \eqref{eqn:basic-estimate0} for $A_\infty$ can be parlayed into one for $A_i$. This proves \eqref{eqn:basic-estimate}.

Given \eqref{eqn:basic-estimate}, it follows as in  \cite[Lemma 3.2]{Shaw:85}
that we can find $u_i$ satisfying (1) and (2).  Let $\varphi_i=\mathrm{Id_E}+u_i$. Then by (1),
$$
\dbar_{A_i}\varphi_i=\dbar_{A_i}\mathrm{Id}_E+\dbar_{A_i}u_i=(\dbar_{A_\infty}+a_i)\mathrm{Id}_E-a_i=0\ .
$$
We thus have produced a holomorphic map $\varphi_i:\Ecal_\infty\to \Ecal_i$ on $U$. Using interior elliptic estimates for $\dbar$, along with (2),
 it follows that for sufficiently large $i$,   $\sup |u_i|$ is arbitrarily  small on a subannular region,
   so that in particular $\det \varphi_i\neq 0$ on $U':=U(\sigma/2, \varepsilon'; \delta/2)$, say, for all sufficiently large $i$. 
   Since $\Ecal_\infty$ is holomorphically trivial and  since $\Ecal_i$ is reflexive, we may use a realisation of $\mathcal{E}_i$ on the interior $\widehat U:= B^{3}_\sigma(0)\times \Delta_\delta^{n-3}$ of $U$ of the form stated in \cite[Chap.\ V, Prop.\ 4.13]{Kobayashi:87} as well as a classical Hartogs-type theorem for holomorphic functions on annular regions to see that $\varphi_i$ extends as a map of sheaves to $\widehat U$. 
   
   Let $D_i$ be the zero locus of $\det \varphi_i$ in $\widehat U$. Notice that $x_i\in D_i$, for if not, $\Ecal_\infty$ and $\Ecal_i$ would be isomorphic in a neighbourhood of $x_i$  by reflexivity, and hence $\Ecal_i$ would be locally free at $x_i$, thus contradicting the assumption that $x_i\in S(A_i)$. Then, all the  divisors $D_i$ have to intersect the annular region $U'$, in particular those with index $i$ sufficiently large. This contradicts our choice of the set $U'$ and completes the proof. 
\end{proof}

\medskip\noindent{\bf Step 5.} Finally, we explain how to assign multiplicities to each irreducible component of $Z_\infty$ to obtain the cycle $\Ccal_\infty$. Since $Z''_\infty$ is the extension of $\widetilde Z$,  the irreducible components of $Z''_\infty$  have assigned multiplicities from Theorem \ref{thm:tian}, giving a cycle $\Ccal''_\infty$. Since we have chosen  $\Ccal_j\to \Ccal'_\infty$ as cycles, each component $Z\subset Z'_\infty$ carries a multiplicity $m'_Z$. It may occur that  there is additional bubbling along $Z$. Let $z\in Z$ be a point such that $B_{2\sigma}(z)\subset X\backslash S(A_\infty)$ intersects $Z_\infty$ only in the smooth locus of $Z$, and let $\widehat A_j$ be the extended connections  from Step 3.  Let $\Omega$ be smooth $(n-4)$-form compactly supported on $B_{\sigma}(z)$. Then there is an integer $m''_Z$ such that
\begin{equation} \label{eqn:excess}
\lim_{j\to \infty} \int_{B_{\sigma}(z)} \left( \ch_2(A_\infty)-\ch_2(\widehat A_j)\right)\wedge\Omega=m''_Z\cdot\int_{Z\cap {B_{\sigma}(z
)}}\Omega\ .
\end{equation}
It follows that we should assign the multiplicity of the component $Z$ to be $m_Z=m'_Z+m''_Z$, and we write $\Ccal'_\infty$ for the cycle whose summands are the irreducible components of $Z'_\infty$ with multiplicities defined in this way.
Finally, we are in a position to prove the following
\begin{lemma}
Define $\Ccal_\infty$ to be the cycle $\Ccal'_\infty+\Ccal''_\infty$. Then with this definition eq.\ \eqref{eqn:currents-converge} holds. 
\end{lemma}

\begin{proof}
Choose a smooth $(2n-4)$-form $\Omega$, and fix $\varepsilon>0$. First, we choose $r_0>0$ such that 
\begin{equation*} \label{eqn:cycle-estimate-1}
0\leq \int_{\Ncal_{r_0}(S'_\infty\cup S(A_\infty))} \Ccal_\infty\wedge\frac{\omega^{m-2}}{(m-2)!}\leq \varepsilon/2\ .
\end{equation*}
Since the $\Ccal_j$ converge to a subcycle of $\Ccal_\infty$, for $j$ sufficiently large it follows that
\begin{equation} \label{eqn:cycle-estimate-2}
0\leq \int_{\Ncal_{r_0}(S'_\infty\cup S(A_\infty))} \Ccal_j\wedge\frac{\omega^{m-2}}{(m-2)!}\leq \varepsilon\ .
\end{equation}
Next, since $S'_\infty\cup S(A_\infty)$ is a subvariety of codimension at least $3$, in particular it has zero Hausdorff $(2n-4)$-measure. Hence, we may find finitely many $x_i\in S'_\infty\cup S(A_\infty)$, $i=1, \ldots, M$, and $0< r_i\leq r_0/2$, such that 
\begin{align}
S'_\infty\cup S(A_\infty)&\subset \bigcup_{i=1}^M B_{r_i}(x_i)=: U_1\ ;\\
\sum_{i=1}^M r_i^{2n-4}&\leq \varepsilon \cdot 2^{4-2n}\ .
\end{align}
Set $U_2=\cup_{i=1}^M B_{2r_i}(x_i)$, and $r=\min\{r_1, \ldots, r_M\}$.  Note that if $y\not\in U_2$, then
$
B_r(y)\cap U_1=\emptyset
$.
Find finitely many $y_j\in \overline U_2^c\cap Z'_\infty$, $j=1,\ldots, N$, and $0< s_j\leq r$, such that 
$$
U_2^c\cap Z'_\infty\subset \bigcup_{j=1}^N B_{s_j}(y_j)=: V\ .
$$
Taking  a partition of unity subordinate to the cover 
$$\{B_{2r_i}(x_i), B_{s_j}(y_j), \Ncal_r((\overline U_2\cup\overline V)^c)\}\ ,$$
it suffices to consider the cases where $\Omega$ is compactly supported in each of the elements of the cover. 
From the monotonicity formula referred to previously,
 there is a constant $\Lambda$ independant of $i$, such that 
$$
\int_{B_\sigma(x)} |F_{A_i}|^2\frac{\omega^m}{m!}\leq \Lambda\cdot \sigma^{2n-4}\ ,
$$
for $0<\sigma\leq \sigma_0$.  It follows that there is a constant $\Lambda_1$ such that
\begin{equation} \label{eqn:c-estimate}
\left| \int_{B_\sigma(x)} \ch_2(A_i)\wedge \Omega\, \right|\leq \Lambda_1\sup|\Omega|\cdot \sigma^{2n-4}\ .
\end{equation}
The constant $\sigma_0$ only depends on the geometry of $X$, and so we may assume without loss of generality that $r_0\leq\sigma_0$. By Fatou's lemma, \eqref{eqn:c-estimate} also holds for $A_\infty$.  It follows from \eqref{eqn:c-estimate} that
\begin{align}
\left| \int_{U_2} \ch_2(A_i)\wedge \Omega\, \right|&\leq \sum_{i=1}^M \left| \int_{B_{2r_i}(x_i)} \ch_2(A_i)\wedge \Omega\, \right|\leq 2^{2n-4}\Lambda_1\sup|\Omega| \sum_{i=1}^M r_i^{2n-4} \notag\\
&\leq \Lambda_1\sup|\Omega|\cdot\varepsilon\ .\notag
\end{align}
The same holds for $A_\infty$ in place of $A_i$. Since $U_2\subset \Ncal_{r_0}(S'_\infty\cup S(A_\infty))$, and using \eqref{eqn:cycle-estimate-2}, we have for sufficiently large $i$,
\begin{equation}
\left|
\int_{U_2} (\ch_2(A_i)-\Ccal_i)\wedge \Omega - \int_{U_2} (\ch_2(A_\infty)-\Ccal_\infty)\wedge \Omega\, 
\right|
\leq 2\varepsilon\cdot\sup|\Omega| 
(\Lambda_1+1)\ . \label{eqn:u1}
\end{equation}
Next, suppose $\Omega$ is compactly supported in $B_{s_j}(y_j)$. For $i$ sufficiently large we have $S(A_i)\subset U_1$, and hence $S(A_i)\cap B_{s_j}(y_j)=\emptyset$.  As in Step 3, we have connections $\widehat A_i$ on $B_{s_j}(y_j)$ that converge (up to gauge) to $A_\infty$ smoothly on compact subsets in the complement of $Z_\infty\cap B_{s_j}(y_j)$. 
By definition of the excess multiplicity in \eqref{eqn:excess}, we have
$$
\lim_{i\to \infty}\int_{B_{s_j}(y_j)}(\ch_2(A_\infty)-\ch_2(\widehat A_i)+\Ccal_i)\wedge \Omega
=\int_ {B_{s_j}(y_j)} \Ccal'_\infty\wedge\Omega\  .
$$
Finally, on $\Ncal_{r}(\overline U_2\cap V)^c$, the $A_i$ are smooth HYM connections and Theorem \ref{thm:tian} applies directly to show that 
$$
\lim_{i\to \infty}\int_{\Ncal_{r}(\overline U_2\cap V)^c}(\ch_2(A_\infty)-\ch_2(A_i))\wedge \Omega
=\int_{\Ncal_{r}(\overline U_2\cap V)^c} \Ccal''_\infty\wedge\Omega\ .
$$
Since $\varepsilon$ was arbitrary, the lemma follows.
\end{proof}

\subsection{Diagonalisation} \label{subsect:diagonalisation}
The goal of this section is to prove  the crucial result.
\begin{prop} \label{prop:diagonal}
Suppose that we have a bounded sequence of  ideal HYM connections that converges  $(A_i, \Ccal_i, S(A_i))\to (A_\infty, \Ccal_\infty, S(A_\infty))$ as in the previous section. 
If each $(A_i, \Ccal_i, S(A_i))$ is an Uhlenbeck limit of smooth HYM connections on $(E,h)\to X$, then so is $(A_\infty, \Ccal_\infty, S(A_\infty))$.
\end{prop}

\begin{proof}
We have sequence of gauge transformations $s_i$ such that $s_i(A_i)\to A_\infty$ in $C^\infty_{loc}$ away from $Z_\infty\cup S(A_\infty)$. 
For each $i$, let $\{A_{i,k}\}$ be a sequence of smooth HYM connections and gauge transformations $g_{i,k}$ on $X\backslash Z_i\cup S(A_i)$ and Uhlenbeck limits $g_{i,k}(A_{i,k})\to A_i$ as $k\to \infty$.  By a diagonalisation argument, and using the fact that $Z_\infty\cup S(A_\infty)$ is a holomorphic subvariety, we can find a subsequence $k_i\to \infty$ such that $s_ig_i(A_{i,k_i})\to A_\infty$ in $C^\infty_{loc}$ on $X\backslash Z_\infty\cup S(A_\infty)$, where $g_i=g_{i, k_i}$. Applying Theorem \ref{thm:uhlenbeck}, after passing to a further subsequence which we continue to denote by $i$, we may assume there is an admissible connection $B_\infty$ on $(E,h)\to X$ and gauge transformations $\widetilde g_i$ such that $\widetilde g_i(A_{i,k_i})\to B_\infty$ up to gauge in $C^\infty_{loc}$ away from a singular set $\widetilde Z_\infty\cup S(B_\infty)$. By Theorem \ref{thm:bando-siu}, the holomorphic bundles $(E, \dbar_{A_\infty})$ and $(E, \dbar_{B_\infty})$ extend as reflexive sheaves $\Ecal_{A_\infty}$ and $\Ecal_{B_\infty}$ on $X$. Since $(Z_\infty, S(A_\infty))$ and $(\widetilde Z_\infty, S(B_\infty))$ have codimension at least $2$, and $\Ecal_{A_\infty}\cong \Ecal_{B_\infty}$ on the complement, then in fact $\Ecal\cong \Ecal_{A_\infty}\cong \Ecal_{B_\infty}$ everywhere on $X$. Furthermore, by uniqueness of the Bando-Siu HYM connection, $A_\infty=B_\infty$ up to gauge, and so $S(A_\infty)=S(B_\infty)$. By further modifying $\tilde g_i$, we may assume $A_\infty=B_\infty$. 

 It remains to consider the blow-up loci $Z_\infty$ and $\widetilde Z_\infty$ and multiplicities. Since up to gauge,  $A_{i,k_i}\to A_\infty$ on $X\backslash Z_\infty\cup S(A_\infty)$, it follows the blow-locus for $A_{i,k_i}$ is contained in $Z_\infty$. This means that if $\widetilde Z_j$ is an irreducible component of $\widetilde Z_\infty$, then $\widetilde Z_j$ must be equal to some component $Z_j\subset Z_\infty$. For simplicity, call this component $Z$.

Then to prove equality of sets we only need to know that $Z_\infty$ and $\widetilde Z_\infty$ have the same number of irreducible components. This will follow immediately from the fact that the cohomology classes of $[\Ccal_\infty]$ and $[\widetilde \Ccal_\infty]$ agree (and are equal to $[\ch_2(A_\infty)]-\ch_2(E)=[\ch_2(B_\infty)]-\ch_2(E)$), if we show that the multiplicity associated to $Z$ in the cycle $\Ccal_\infty$  agrees with its the multiplicity in the cycle $\widetilde \Ccal_\infty$.

In other words, to prove that $\Ccal_\infty$ is equal to the cycle $\widetilde \Ccal_\infty$ associated to $\widetilde Z_\infty$, it suffices to prove that $\widetilde m_{Z}=m_{Z}$.   
Let $\Sigma$ be a generic slice to $Z$. Then by Lemma \ref{lem:slicing2}, 
$$
\widetilde m_Z=\lim_{i\to \infty} \frac{1}{8\pi^2}\int_\Sigma \left( \tr (F_{A_{i,k_i}}\wedge F_{A_{i,k_i}})-\tr( F_{A_\infty}\wedge F_{A_\infty})\right)\ .
$$
By the discussion following Lemma \ref{lem:Chern-Simons}, we have $\widetilde m_Z=\deg \widetilde h_i$, where $\widetilde{h}_{i}=h\circ \widetilde{g}_{i}$, and $h$ is the gauge transformation on $\Sigma \backslash \{z\}$ defined there.
Write $a_{i}=A_{\infty }-\widetilde{g}_{i}A_{i,k_{i}}$ and $b_{i}=A_{\infty
}-s_{i}g_{i}A_{i,k_{i}}$ so that 
$s_{i}g_{i}A_{i,k_{i}}-\widetilde{g}_{i}A_{i,k_{i}}=a_{i}-b_{i}$. Then 
 by Lemma \ref{lem:Chern-Simons} we have 
\begin{align*}
\deg \widetilde{g}_{i}-\deg(s_{i}g_{i}) &=\int_{\partial \Sigma
}\CS(s_{i}g_{i}(A_{i,k_{i}}),A_{i,k_{i}})-
\CS(\widetilde{g}_{i}(A_{i,k_{i}}),A_{i,k_{i}}) \\
&=\int_{\partial \Sigma }\CS(s_{i}g_{i}(A_{i,k_{i}}),
\widetilde{g}_{i}(A_{i,k_{i}})) \\
&=\frac{1}{8\pi^2}\int_{\partial \Sigma }\tr\bigl\{(a_{i}-b_{i})\wedge D_{\widetilde{g}_{i}(A_{i,k_{i}})}(a_{i}-b_{i})\\
&\qquad\qquad+\frac{2}{3}(a_{i}-b_{i})^{3} +2(a_{i}-b_{i})
\wedge F_{\widetilde{g}_{i}(A_{i,k_{i}})}\bigr\}\ . 
\end{align*}
 
 Now $s_i g_i (A_{i, k_i})$ and $\widetilde g_i(A_{i, k_i})$ converge smoothly on $\partial \Sigma$ to the same connection $A_{\infty}$, so we have $a_{i},b_{i}\rightarrow 0$, and therefore the right hand side above
converges to zero. Hence,  for sufficiently large $i$, $|\deg s_{i}g_{i}-\deg 
\widetilde{g}_{i}|<1$. Since the degree is an integer this implies that for all such $i$ we have $\deg(s_i g_i) =\deg \widetilde g_i$ and therefore $\deg \widetilde{h}_{i}=\deg hs_{i}g_{i}$. 

We claim that $\deg(hs_{i}g_{i})=m_{Z}$. Recall that $\mathcal{C}_{\infty }$
breaks up into two subcycles $\mathcal{C}_{\infty }^\prime$ and $
\mathcal{C}_{\infty }^{\prime\prime}$, and so there are two cases according to whether $
Z$ is an irreducible component of 
$Z_{\infty }^{\prime }=|\mathcal{C}_{\infty }^{\prime }|$ or $Z_{\infty }^{\prime \prime }=|\mathcal{C}_{\infty }^{\prime\prime}|$. In the second case the connections $A_{i}$ bubble at the
point $z\in Z$ through which the slice $\Sigma $ is taken. Since $\Sigma $
may be chosen to be away from $Z_{\infty }^{\prime }$; and the supports $
Z_{i}=|\mathcal{C}_{i}|$ are contained in a neighbourhood of $Z_{\infty
}^{\prime }$ for large enough $i$, then since $Z_{i}$ is the bubbling set
for the sequence of connection $A_{i,k}$ converge without bubbling to $A_{i}$
on $\Sigma $ as $k\rightarrow \infty .$ This furthermore implies that if we
write $E_{i}$ for the topological bundle underlying the reflexive extension $
\mathcal{E}_{i}$, there are isomorphisms $E|_{\Sigma }\simeq E_{i}|_{\Sigma }$.
By the argument following Lemma \ref{lem:Chern-Simons}, we see that $\deg g_{i}$ calculates
the bubbling multiplicity for $A_{i,k}$ along $Z$ for sufficiently large $i$
and $k$, and this number is zero by construction, so $\deg g_{i}=0$ and $
\deg(hs_{i}g_{i})=\deg(hs_{i})$. Again by the discussion after Lemma \ref{lem:Chern-Simons} , $
\deg(hs_{i})=m_{Z}$. 

Now consider the case that $Z$ is an irreducible component of $Z_{\infty
}^{\prime }$. In this case $m_{Z}=m_{Z}^{\prime }+m_{Z}^{\prime
\prime }$ where $m_{Z}^{\prime }$ is the multiplicity assigned $Z$ by
the convergence of the cycles $\mathcal{C}_{i}$, and $m_{Z}^{\prime \prime
}$ is the excess multiplicity defined in Step 5 eq.\ \eqref{eqn:excess}. Notice that by Lemma \ref{lem:Chern-Simons} 
we have $\deg(hs_{i}g_{i})=\deg(hg_{i})+\deg s_{i}$. We claim that $\deg
(hg_{i})=m_{Z}^{\prime }$ and $\deg s_{i}=m_{Z}^{\prime \prime }$. The
second equality follows again from the argument after Lemma \ref{lem:Chern-Simons}  (here we
choose a single frame for the trivial bundle on $\Sigma ,$ so no additional
gauge transformation $h$ is necessary). To prove the second equality, we
note that by assumption, for each $i$ sufficiently large $Z_{i}$ has an
irreducible component $Z_{i,l}$ converging to $Z$ in the Hausdorff topology
as $i\rightarrow \infty.$ Assume we have chosen $\Sigma $ so that it is
also a transverse slice to all $Z_{i,l}$ through points $z_{i}$ for $i$
sufficiently large. Since $A_{i,k}$ bubbles along $Z_{i,l}$ with Uhlenbeck
limit $A_{i}$, there are gauge transformations $h_{i}$ of $E$ defined on $
\Sigma \backslash \{z_{i}\}$ and isomorphisms $\phi _{i}:E|_{\Sigma
\backslash \{z_{i}\}}\rightarrow E_{i}|_{\Sigma \backslash \{z_{i}\}}$ so
that $h_{i}\phi _{i}^{\ast }D_{E_{i}}=D_{E}$ (using the same notation as in
the argument following Lemma \ref{lem:Chern-Simons} ), and such that $\deg h_{i}$ calculates
the associated multiplicity. Then for sufficiently large $i$, $\deg
h_{i}g_{i}=m_{Z}^{\prime }$. We claim that $\deg h_{i}=\deg h$ for $i$
sufficiently large. We have
\begin{eqnarray*}
\deg h_{i}-\deg h &=&\int_{\partial \Sigma }\CS(h_{i}\phi _{i}^{\ast
}D_{E_{i}},\phi _{\infty }^{\ast }D_{E_{\infty }})-\CS(h\phi _{\infty }^{\ast
}D_{E_{\infty }},\phi _{\infty }^{\ast }D_{E_{\infty }}) \\
&=&\int_{\partial \Sigma }\CS(D_{E},\phi _{i}^{\ast }D_{E_{i}})-\CS(D_{E},\phi
_{\infty }^{\ast }D_{E_{\infty }}) \\
&&+\int_{\partial \Sigma }\CS(\phi _{i}^{\ast }D_{E_{i}},\phi _{\infty }^{\ast
}D_{E_{\infty }})\ .
\end{eqnarray*}
Since $A_{i}\rightarrow A_{\infty }$ on $\partial \Sigma $, $\phi _{i}^{\ast
}D_{E_{i}}\rightarrow \phi _{\infty }^{\ast }D_{E_{\infty }}$ on $\partial
\Sigma $ and therefore for large $i$ this difference is zero. This implies $
\deg(hg_{i})=m_{Z}^{\prime }$.  
\end{proof}

\subsection{The  compactification} \label{sec:compactification}

We are now in the position to define the desired compactification of the moduli set $M_\HYM^\ast$.

\begin{defi} \label{def:convergence}
We say that a sequence $[(A_{i},\mathcal{C}_{i},S(A_{i})]\in 
\widehat{M}_{\HYM}$ converges to $[(A_{\infty },\mathcal{C}_{\infty },S(A_{\infty
})]$ in $\widehat{M}_{\HYM}$ if the conclusion of Theorem \ref{thm:ideal-convergence} holds. 
\end{defi}

 We define a topology on $\widehat M_{\HYM}$ by giving a basis $\{\Ucal_{\vec\varepsilon}([A,\Ccal, S(A)])\}$ of open neighbourhoods of an ideal HYM connection depending on a $4$-tuple $\vec\varepsilon=(\varepsilon_1, \varepsilon_2, \varepsilon_3, \varepsilon_4)$, $\varepsilon_i>0$.
The set $\Ucal_{\vec\varepsilon}([A,\Ccal, S(A)])$ consists of gauge equivalence classes of all ideal HYM connections $(A_1, \Ccal_1, S(A_1))$ satisfying:
\begin{enumerate}
\item There exists a subcycle $\widehat \Ccal_1$ of $\Ccal$ such that $\widehat\Ccal_1$ and $\Ccal_1$ are within $\varepsilon_1$ with respect to the mass norm.
 \item  $S(A_1)$ is contained in the $\varepsilon_2$-neighborhood of $S(A)\cup|\Ccal| \subset X$;  
\item On the complement of a $2\varepsilon_2$-neighbourhood of $S(A)\cup |\Ccal| \subset X$, there is a gauge transformation $g$ such that $g(A_1)$ and $A$ are within $\varepsilon_3$ with respect to the $C^\infty$-topology;
\item the currents $\ch_2(A_1)-\Ccal_1$ and $\ch_2(A)-\Ccal$ are within $\varepsilon_4$ with respect to the mass norm.
\end{enumerate}
 
\begin{rem}
By definition, a sequence $[(A_{i},\mathcal{C}_{i},S(A_{i})]$ converges to an ideal connection $[(A_{\infty },\mathcal{C}_{\infty },S(A_{\infty })]$ in this topology if and
only if it converges in the sense of Definition \ref{def:convergence}. Namely a sequence converges
the sense of Definition \ref{def:convergence} to an ideal connection $[(A_{\infty },\mathcal{C}_{\infty
},S(A_{\infty })]$ if and only if for every basic open set
 $\mathcal{U}_{\varepsilon }([(A_{\infty },\mathcal{C}_{\infty },S(A_{\infty })])$, there
exists $I\gg0$ such that $[(A_{i},\mathcal{C}_{i},S(A_{i})]\in
 \mathcal{U}_{\varepsilon }([(A_{\infty },\mathcal{C}_{\infty },S(A_{\infty })])$ for
all $i>I$. 
\end{rem}

First of all,  we have the following.
\begin{thm}\label{thm:ideal_connections_compact}
The set $\widehat{M}_{\HYM}$ defined in \eqref{eqn:ideal-moduli} and endowed with the topology described above is a first countable Hausdorff space and moreover sequentially compact. 
\end{thm}
\begin{proof}
The first statement follows directly from the definition. Compactness follows from
Theorem \ref{thm:ideal-convergence}: Indeed, any sequence of ideal connections has a subsequence
converging to an ideal HYM connection $(A_\infty,\Ccal_\infty, S(A_\infty))$ with respect to the topology defined above 
(regarding point (2) above, note that by Lemma \ref{lem:sing-limit}, the singular sets $S(A_i)$ converge in the Hausdorff sense into $S(A_\infty)\cup|\Ccal_\infty|$). 
\end{proof}
With this understood we make the following definition. 
\begin{defi} \label{def:Uhlenbeck-compactification}
The {\em gauge theoretic compactification}  $\overline M_{\HYM}$ 
is the (sequential) closure $\overline{M^\ast_\HYM}\subset\widehat M_{\HYM}$, that is, the set of all Uhlenbeck limits of sequences in ${M^\ast_\HYM}$.
\end{defi}
\begin{rem}
In a first countable space the notions of closure and sequential closure coincide. Hence, in the preceding definition it does not make a difference which closure we take. 
\end{rem}

\begin{rem}\label{rem:twodimensions}
 If $\dim X=2$, $\overline M_{\HYM}$ coincides with the Donaldson-Uhlenbeck compactification (cf.~\cite[III(iii)]{Donaldson:86} as well as the summary of the construction provided in \cite[p.~444]{Li:93}).
\end{rem}

\begin{rem}\label{rem:sequentially_closed}
By Proposition \ref{prop:diagonal}, $\overline M_{\HYM}$ is sequentially closed; notice that this fact requires an argument and is not automatic. This means precisely that any limit $[(A_\infty,\Ccal_\infty, S(A_\infty))]$ of elements in $\overline M_{\HYM}$ is itself the limit of connections in ${M}_{\HYM}^\ast$. 
\end{rem}

The main result of this section is the following.
\begin{thm}
The space $\overline{M}_{\HYM}$ is Hausdorff and sequentially compact.
\end{thm}

\begin{proof}
The Hausdorffness is clear, since $\overline M_{\HYM}$ is a subspace of a Hausdorff space. The sequential compactness follows from the fact $\overline M_{\HYM}$ is a sequentially closed subspace of a sequentially compact space by Remark \ref{rem:sequentially_closed} and Theorem \ref{thm:ideal_connections_compact} above.  
\end{proof}

\begin{rem}
The results of the next section will give a slightly different proof of sequential compactness for $\overline M_{\HYM}$. Namely, we will see that the reflexive sheaves associated to Uhlenbeck limits of sequences in ${M}_{\HYM}^\ast$ are  the double duals of sheaves in a fixed Quot scheme, and therefore the boundedness in Lemma \ref{lem:bounded-admissible} actually follows for this particular subset of $\widehat M_{\HYM}$ without appealing to the results of \cite{Maruyama:81}. 
\end{rem}


\section{Comparison of analytic and algebraic moduli spaces} \label{sec:comparison}
In this section we define the principal object of study in this paper: namely, a map 
$
\overline{\Phi }:\overline{M}^\mu\rightarrow \widehat{M}_{\HYM}
$ from the closure of $(M^s)^{wn}$ inside the slope
compactification $M^\muss$ to $\widehat{M}_{\HYM}$, extending
\eqref{eqn:phi}. We will give the definition in Section
\ref{sec:map}. The main result of this section is that $\overline\Phi$ maps onto $\overline M_\HYM$ and is continuous. More precisely, recall from Section \ref{sec:moduli}  that $\overline{M}^\mu$ is constructed using the ring of invariant sections of some determinant line bundle on $\mathcal{Z}$, the weak normalisation of a locally closed subscheme of some Quot scheme parametrising slope semistable sheaves. We will study a map $\overline\Psi : \overline \Zcal_{\circ}\to \widehat{M}_{\HYM}$, defined on a certain closed subvariety $\overline \Zcal_{\circ}$ of $\Zcal$, that descends to $\overline\Phi$. 
  Both $\overline \Zcal_{\circ}$ and $\overline{M}^\mu$ are (quasi-)projective and hence admit metrics. In particular, in order to prove continuity it suffices to prove sequential continuity.

Recall from Definition \ref{def:ideal-connection} that
 a point of $\widehat{M}_{\HYM}$ consists of two pieces of data:  an admissible HYM connection and a holomorphic cycle. The proof of continuity consists of showing that the respective limits coincide in both these components. We  first check on sequences in the interior of $\overline \Zcal_{\circ}$ (that is, for sequences of locally free stable sheaves) that the image of $\overline\Psi$ and therefore of $\overline\Phi$ lies in $\overline M_\HYM$. Using Proposition \ref{prop:diagonal} we will then reduce to this case. Equality in the sheaf component  will follow from the fact that any Uhlenbeck limit of smooth HYM connections may be identified with the double dual of a sheaf appearing as a quotient of $\Hcal$. This will be proven in  Section \ref{sec:limits1}. In Section \ref{sec:agree-cycle-component} we give a small extension of the singular Bott-Chern formula from \cite{SibleyWentworth:15}. Using this, equality of the cycle components is proven, and the proof of continuity is then completed in Section \ref{subsect:continuity} using the diagonalisation argument of Section \ref{subsect:diagonalisation}. 
\subsection{Definition of the comparison between moduli spaces} \label{sec:map}
Suppose first that we have a torsion free sheaf $\Fcal\to X$ satisfying the conditions:
\begin{enumerate}
	\item $\Fcal$ is $\mu$-semistable and $\det\Fcal\simeq \Jcal$;
\item  $\ch_2(\Fcal)=\ch_2(E)$ in rational cohomology; 
\item there is a smooth bundle isomorphism $\Fcal\bigr|_{X\backslash\sing \Fcal}\simeq E\bigr|_{X\backslash\sing \Fcal}$, and so a $C^{\infty}$ isomorphism 
$$\ddual{(\Gr\Fcal)}\bigr|_{X\backslash\sing(\Gr\Fcal)}=(\Gr\Fcal)\bigr|_{X\backslash\sing(\Gr\Fcal)}\simeq E\bigr|_{X\backslash\sing(\Gr\Fcal)}\ .$$
\end{enumerate}
We will associate to this data a gauge equivalence class of ideal HYM connections $[(A,\Ccal, S(A))]$ as follows. 

First, note that by \cite[Thm.\ 3]{BandoSiu:94} there exists an admissible HE metric $h_{\Fcal}$ on $\ddual{(\Gr\Fcal)}$.  By (3) we may write $\ddual{(\Gr\Fcal)}=\Gr\Fcal=(E,\dbar_E)$ on $X\backslash\sing(\Gr\Fcal)$.  Let $g$ be a complex gauge transformation defined on $X\backslash\sing(\Gr\Fcal)$ such that $g(h_{\Fcal})=h$ in the sense that for any $e_1, e_2\in E$,
$$
\langle e_1, e_2\rangle_{h_{\Fcal}}=\langle g^{-1}e_1, g^{-1} e_2\rangle_h\ .
$$
Then a straightforward calculation regarding the curvatures of the Chern connections $A =(g\cdot \dbar_{E}, h)$ and $(\dbar_E, h_{\Fcal})$ yields that 
$$
F_{(g\cdot \dbar_{E}, h)} = g^{-1}\circ F_{  (\dbar_E, h_{\Fcal})   }\circ g\ .
$$
In particular, since $h_{\Fcal}$ is HE and admissible, the pair  $(A, \sing(\Gr\Fcal))$ defines an admissible HYM connection on $(E, h)$ in the sense of Definition\ \ref{def:admissible}.

Second, consider the support cycle  $\Ccal\in\Cscr_{n-2}(X)$ of the torsion sheaf $\Tcal_{\Fcal}=\ddual{(\Gr\Fcal)}/\Gr\Fcal$ defined in Section \ref{subsubsect:support_cycles}. By \cite[Prop.\ 2.3.1]{Tian:00}, $\ch_2(A)$  is a closed current which is easily seen to represent $\ch_2(\ddual{(\Gr\Fcal)})$ in rational cohomology (cf.\  \cite[proof of Prop.\ 3.3]{SibleyWentworth:15}).  From the exact sequence
 $$
 0\lra \Gr\Fcal\lra \ddual{(\Gr\Fcal)}\lra \Tcal_{\Fcal}\lra 0
 $$
 and the fact that $\ch_2(\Tcal_\Fcal)=[\Ccal]$ (see \cite[Prop.\ 3.1]{SibleyWentworth:15}), we conclude using assumption (2) above that
 \begin{equation} \label{eqn:cohom}
[ \ch_2(A) ]=\ch_2(E)+[\Ccal] \ .
 \end{equation}
Hence, if we define $S(A)$ to be equal to $\sing(\ddual{(\Gr\Fcal)})$, the conditions in Definition \ref{def:ideal-connection} are satisfied. 

In summary, from a sheaf $\Fcal$ satisfying (1)--(3) above we obtained an ideal HYM connection $(A, \Ccal, S(A))$ on $(E,h)$. Different choices of $g$ give gauge equivalent ideal connections, and isomorphic sheaves $\Fcal$ give rise to the same class $[(A,\Ccal,S(A))]\in \widehat M_{\HYM}$. 

Let $\overline {R_{\circ}}\subset R^{\muss}$ denote the Zariski closure of the Zariski open set $R_\circ$ 
 consisting of locally free $\mu$-stable quotients $q:\Hcal \to \Fcal$ with $\det\Fcal\simeq \Jcal$ such that the underlying smooth bundle of $\Fcal$ is $C^\infty$-isomorphic to $E$. By an application of Ehresmann's Theorem to the associated family of unitary bundles, this is a finite union of connected components of the Zariski open subset of locally free  $\mu$-stable quotients $q:\Hcal \to \Fcal$ with $\det\Fcal\simeq \Jcal$. Let $\overline{\Zcal_\circ}\subset \Zcal$ its preimage under the weak normalisation map $\mathcal{Z} \to R^{\mu ss}_{red}$, which recall is a homeomorphism between the underlying topological spaces. Note that the Zariski closure of $(M^s)^{wn}$ defined in Definition\ \ref{def:defining_the_closure} is equal to the image of $\overline{\Zcal_\circ}$ under the classifying map $\pi: \Zcal \to M^\muss$. By Lemma \ref{lem:quot-topology}, if $q:\Hcal \to \Fcal$ is in $\overline{R_\circ}$, then $\Fcal$ satisfies conditions (1)--(3) above. Hence, from the previous discussion we have a well-defined map 
 \begin{equation*}\overline \Psi : \overline {\Zcal_\circ}  \to \widehat M_\HYM, \;\;
q_{\mathcal{F}} \mapsto[(A,\Ccal,S(A))]\ .
\end{equation*}
Notice that the gauge equivalence class $[(A,\Ccal,S(A))]$ depends only on $\ddual{(\Gr
\Fcal)}$ and $\Ccal_{\mathcal{F}}$. Consequently, by Proposition \ref{prop:GT}, the  map $\overline{\Psi}$ descends to a map $\overline \Phi : \overline M^\mu\to \widehat M_\HYM$ satisfying
\begin{equation}\label{eq:relation_phi_psi}
 \overline{\Psi} = \overline{\Phi} \circ \pi\, \bigr|_{\overline{\Zcal_\circ}}\ 
\end{equation}
and extending \eqref{eqn:phi} to the respective compactifications. 

\begin{rem}
 By construction, if $\Fcal$ is polystable, then Remark\ \ref{rem:gamma_campatible_with_earlier_defi} implies that with the notations introduced in Definition\ \ref{defi:cycle,couple} we have
$\overline{\Phi}([\Fcal]) = [\gamma(\Fcal)]$.
\end{rem}

\subsection{Identification of the limiting sheaf} \label{sec:limits1}
The main goal of this section is to prove the following result, whose proof identifies an Uhlenbeck limit of smooth HYM connections with the double dual of a certain point in $\Quot(\Hcal,\tau)$ (in fact in $\Quot(\Hcal,c(E))$).

\begin{prop} \label{prop:quot-limit}
Let $A_i$ be a sequence of HYM connections on $(E,h)$, giving holomorphic bundles $\Ecal_{i}$. Then for any Uhlenbeck limit $A_\infty$ with reflexive extension $\Ecal_\infty$, there is a point $\widehat q_\infty: \Hcal\to \widehat \Ecal_\infty\to 0$ in $\Quot(\Hcal,c(E))$, such that $\ddual{(\widehat \Ecal_\infty)}$ is isomorphic to $\Ecal_\infty$.

\end{prop}

Let $\Ecal_i=(E, \dbar_{A_i})$. We realize $\Ecal_i(m)$ as a quotient by choosing a basis $\ebold_i$ of holomorphic sections that is orthonormal with respect to the $L^2$ metric induced  by the hermitian structure $h\otimes h_L^m$on $\Ecal_i(m)$. Denote the similar construction of a choice of a basis for a further twist by:
 \begin{equation} \label{eqn:E-quot}
 \ebold_i^{(k)} : \Ocal_X^{\tau(m+k)}\lra \Ecal_i(m+k)\ .
 \end{equation}
We begin with a preliminary result.

\begin{lemma} \label{lem:rank-limit}
There is a subsequence, also denoted $\{i\}$, such that for any $k\geq0$, the basis $\ebold_i^{(k)}$ converges smoothly away from $S_\infty$ to a nonzero holomorphic map 
$$\ebold_\infty^{(k)}   : \Ocal_X^{\tau(m+k)}\lra \Ecal_\infty(m+k)\ .$$
 Moreover, if $ \Ecal_\infty^{(k)}\subset \Ecal_\infty$ denotes the image of $\ebold_\infty^{(k)}$, twisted by $\Ocal_X(-m-k)$,
then there is a coherent subsheaf $\widetilde\Ecal_\infty\subset \Ecal_\infty$, such that
\begin{enumerate}
\item 
 for $k$  sufficiently large, $\Ecal^{(k)}_\infty= \widetilde\Ecal_\infty$;
 \item $\ddual{(\widetilde\Ecal_\infty)}\cong \Ecal_\infty$;
 \item 
 $\chi_{ \widetilde\Ecal_\infty}(\ell)\geq \tau(\ell)$, for $\ell$ sufficiently large.
 \end{enumerate}
\end{lemma}

\begin{proof}
For $k$ large, we denote the orthonormal basis ${\bf e}_i^{(k)}=\{ e^{(k)}_{i,j}\}_{j=1}^{\tau(m+k)}$. 
Let $B_i$ denote the connection on $\Ecal_i(m+k)$ induced by $A_i$ on $\Ecal_i$ and the connection $a_L$ on $L$.
Notice that (see Section \ref{sec:prelim})
$$\sqrt{-1}\Lambda F_{B_{i}}=\sqrt{-1}\Lambda F_{A_{i}}+2\pi (m+k)\lambda\omega\cdot \Ibold\ .$$
Then since $e^{(k)}_{i,j}$ is $B_{i}$-holomorphic, the Bochner formula gives 
\begin{eqnarray*}
\sqrt{-1}\bar{\partial}\partial\langle
e_{i,j}^{(k)},e_{i,j}^{(k)}\rangle  &=&\sqrt{-1}\bar{\partial}
\langle \partial _{B _{i}}e_{i,j}^{(k)},e_{i,j}^{(k)}\rangle
  \\
&=&\sqrt{-1}(\langle\bar{\partial}_{B _{i}}\partial
_{B_{i}}e_{i,j}^{(k)},e_{i,j}^{(k)}\rangle +\langle
\partial _{B _{i}}e_{i,j}^{(k)},\partial
_{B_{i}}e_{i,j}^{(k)}\rangle)  \\
&=&\sqrt{-1}(\langle F_{B
_{i}}e_{i,j}^{(k)},e_{i,j}^{(k)}\rangle -\langle \partial
_{B _{i}}e_{i,j}^{(k)},\partial _{B
_{i}}e_{i,j}^{(k)}\rangle)\ ,
\end{eqnarray*}
Hence,
\begin{eqnarray*}
\Delta (|e_{i,j}^{(k)}|^{2}) 
&=&2\Delta _{\partial }(|e_{i,j}^{(k)}|
^{2}) =\Lambda(( 2\sqrt{-1}\bar{\partial}\partial
( |e_{i,j}^{(k)}|^{2})) \\
&=&2\sqrt{-1}\langle \Lambda F_{B
_{i}}e_{i,j}^{(k)},e_{i,j}^{(k)}\rangle-2|\partial _{B
_{i}}e_{i,j}^{(k)}|^{2} \\
&\leq &2\sqrt{-1}\langle \Lambda F_{B
_{i}}e_{i,j}^{(k)},e_{i,j}^{(k)}\rangle \ ,
\end{eqnarray*}
where we have used the identity $\sqrt{-1}\Lambda\langle \beta
,\beta \rangle =|\beta |^{2}$ for any  $\beta\in \Omega^{1,0}(X, E\otimes L^{m+k})$. We therefore have

\begin{equation*}
\Delta |e_{i,j}^{(k)}|^{2}\leq C(\mu(E)+2\pi n\lambda(m+k))|e_{i,j}^{(k)}|^{2}\leq C|e_{i,j}^{(k)}|^{2}\ ,
\end{equation*}
for a constant $C$ independent of $i$.  
By a result of Morrey (cf.\ \cite[Thm.\ 9.20]{GilbargTrudinger:83}),  there is a uniform bound
 $$\sup_X
|e^{(k)}_{i,j}|\leq C\cdot \Vert e_{i,j}^{(k)}\Vert^2\leq C\ ,$$
independent of $i$ and for each $j$. Since the analytic singular set $S$ has measure zero, we may find a cover $
\{B_{\sigma }(x_{\alpha})\}_{x_{\alpha}\in S}$ such that 
\begin{equation*}
C\sum_{i}\vol(B_{\sigma }(x_{\alpha}))<\frac{1}{4}\ .
\end{equation*}
Write 
$
K_{\sigma }=X\backslash \cup _{\alpha}B_{\sigma }(x_{\alpha})
$,
and note that this together with the $L^\infty$ estimate given above we have
that 
$$
||e_{i,j}^{(k)}||_{L^{2}(X)} =||e_{i,j}^{(k)}||_{L^{2}(\cup _{\alpha}B_{\sigma
}(x_{\alpha}))}+||e_{i,j}^{(k)}||_{L^{2}(K_{\sigma })} 
\leq ||e_{i,j}^{(k)}||_{L^{2}(K_{\sigma })}+\frac{1}{2}\ ,
$$
so that $||e_{i,j}^{(k)}||_{L^{2}(K_{\sigma })}\geq \frac{1}{2}$. 

By Theorem \ref{thm:uhlenbeck}, the $A_{i}$ converge smoothly outside of a holomorphic subvariety $S$ to a connection $A_{\infty }$ on a bundle $E_{\infty }$ which is smoothly isometric to $(E,h)$, and correspondingly the
connections $B_{i}$ converge smoothly to a connection $B_{\infty }$ on the complement of this set. Write $
\dbar_{B_\infty}=\dbar_{B_i}+\beta_i$ for $\beta_i\in \Omega ^{0,1}(K_\sigma, \End E)$.
Then holomorphicity of the sections implies $
\dbar_{B_\infty}e_{i,j}^{(k)} =\beta _{i}e_{i,j}^{(k)}$.
Since $\beta_i\to 0$ smoothly, elliptic regularity gives $C^\infty$ estimates on  $e_{i,j}^{(k)}$, and 
 we may extract a
convergent subsequence  to a limit $e_{\infty ,j}^{(k)}$, which is holomorphic with
respect to $B_{\infty }$. By the lower bound on $
||e_{i,j}^{(k)}||_{L^{2}(K_{\sigma })}$ it follows that $e_{\infty ,j}^{(k)}\neq 0$.
 Repeating this argument for an exhaustion of $
X\backslash S$ by sets $K_{\sigma }$ constructed by taking the radii of the
balls $B_{\sigma }$ to zero, and using the fact that the set $S$ is a
holomorphic subvariety (and so we may take this exhaustion to be by
deformation retracts of $X\backslash S$), it follows from a diagonalisation
argument that we obtain   nonzero holomorphic sections of $\mathcal{E}_{\infty
}(m+k)$ on $X\backslash S$. Since $\mathcal{E}_{\infty }$ is reflexive,
these extend to all of $X$ and so we obtain for each $j$ a nontrivial
holomorphic section $e_{\infty ,j}^{(k)}\in H^{0}(\mathcal{E}_{\infty }(m+k))
$. 

For each pair $(j,l)$,  apply the dominated convergence theorem to conclude that
\begin{equation*}
\delta _{jl}=\lim_{i\rightarrow \infty }\int_{X}\langle
e_{i,j}^{(k)},e_{i,l}^{(k)}\rangle \frac{\omega ^{n}}{n!}=\int_{X}\langle
e_{\infty ,j}^{(k)},e_{\infty ,l}^{(k)}\rangle \frac{\omega ^{n}}{n!}\ ,
\end{equation*}
so that the $\{e_{\infty ,j}^{(k)}\}_{j=1}^{\tau (m+k)}$ is an\ 
$L^{2}$-orthonormal subset of
$H^{0}(\mathcal{E}_{\infty }(m+k))$. In other words, we have constructed the promised nonzero holomorphic map\textbf{\ }
\begin{equation*}
\mathbf{e}_{\infty}^{(k)}:\mathcal{O}_{X}^{\oplus \tau (m+k)}\rightarrow \mathcal{E}_{\infty }(m+k)\ .
\end{equation*}  
By a further diagonalisation argument, we can arrange for the convergence along a subsequence, also denoted $\{i\}$, for all $k$ sufficiently large. We now check the stated properties of $\mathbf{e}^{(k)}_{\infty}$.

Let $\Ecal^{(k)}_\infty\subset\Ecal_\infty$ denote the coherent subsheaf 
$$\im \bigl(  \Ocal_X^{\oplus \tau(m+k)}\stackrel{{\bf e}_\infty^{(k)}}{\lra} \Ecal_\infty(m+k)\bigr)\otimes \Ocal_X(-m-k)\ .$$
We claim that for all $k$ sufficiently large, $\Ecal^{(k)}_\infty\subset \Ecal^{(k+1)}_\infty$. Choose a point $p\in X$, 
$U$ a neighbourhood of $p$, a local section $s$ of $\Ecal^{(k)}_\infty$ on $U$, and a global section $\sigma$ of $\Ocal_X(1)$ that is nonvanishing at $p$. By definition, we may write  
\begin{equation} \label{eqn:s}
s=\sum_{j=1}^{\tau(m+k)} f^j e^{(k)}_{\infty,j}\otimes \sigma^{-m}\ ,
\end{equation}
for some $f^j\in \Ocal_U$.
Since $\{ e^{(k+1)}_{i,j}\}$ is an $L^2$-orthonormal basis for $H^0(X, \Ecal_i(m+k+1))$, for each $i$ we may write
\begin{equation} \label{eqn:e}
e^{(k)}_{i,j}\otimes\sigma=\sum_{q=1}^{\tau(m+k+1)} \langle e^{(k)}_{i,j}\otimes \sigma, e^{(k+1)}_{i,q}\rangle e^{(k+1)}_{i,q}\ ,
\end{equation}
where $\langle \cdot, \cdot\rangle$ is the $L^2$-inner product.  By the same argument as above,
$$
\lim_{i\to \infty}  \langle e^{(k)}_{i,j}\otimes \sigma, e^{(k+1)}_{i,q}\rangle\lra  \langle e^{(k)}_{\infty,j}\otimes \sigma, e^{(k+1)}_{\infty,q}\rangle\ .
$$
The right and left hand sides of \eqref{eqn:e} therefore converge smoothly away from $Z^{an}$, 
and therefore $e^{(k)}_{\infty,j}\otimes\sigma$ lies in the image of  ${\bf e}_\infty^{(k+1)}$.  This, in turn means that $s$ in \eqref{eqn:s} may be written 
$$
s=\sum_{j=1}^{\tau(m+k+1)} g^j e^{(k+1)}_{\infty,j}\otimes \sigma^{-m-1}\ ,
$$
for some $g^j\in \Ocal_U$.
Since the right hand side is in $\Ecal^{(k+1)}_\infty$ by definition, the claim follows.

By the Noetherian property of coherent subsheaves of a coherent sheaf \cite[\href{http://stacks.math.columbia.edu/tag/01Y7}{Tag 01Y7}]{stacks-project}, the chain
$
\Ecal^{(k)}_\infty\subset \Ecal^{(k+1)}_\infty\subset \cdots\subset \Ecal_\infty
$,
must stabilise for sufficiently large $k$. 
We set $\widetilde \Ecal_\infty=\Ecal^{(k)}_\infty$, $k\gg0$, and 
this proves part (1) of the lemma. For part (3), notice that for $\ell\gg0$,
\begin{equation} \label{eqn:hilbert-inequality}
\chi_{\widetilde \Ecal_\infty}(\ell)=\chi_{ \Ecal^{(\ell-m)}_\infty}(\ell)
=h^0( \Ecal^{(\ell-m)}_\infty(\ell))\geq h^0( \Ecal_i(\ell))=\tau_E(\ell) \ .
\end{equation}
From the Hirzebruch-Riemann-Roch Theorem and  \eqref{eqn:expansion}, we see that $\rank \widetilde \Ecal_\infty=\rank E=\rank \Ecal_\infty$.
Hence, $\Tcal=\Ecal_\infty/\widetilde \Ecal_\infty$ is a torsion sheaf, and $\det\Ecal_\infty\cong \det\widetilde\Ecal_\infty\otimes \det \Tcal$.  
But then \eqref{eqn:hilbert-inequality} and  \eqref{eqn:expansion} also imply
$$\deg \widetilde\Ecal_\infty \geq\deg E= \deg\Ecal_\infty=\deg \widetilde\Ecal_\infty+\deg(\det \Tcal)\ .
$$
Since the last term on the right hand side above is nonnegative, it must be that $\deg(\det \Tcal)=0$, and hence $\codim(\supp(\Tcal))\geq 2$. Therefore,
$\ddual{(\widetilde\Ecal_\infty)}\cong \ddual{\Ecal_\infty}$.
Since $\Ecal_\infty$ is reflexive, (2)  holds.
This completes the proof of the lemma.
\end{proof}

Henceforth, we assume $k$ has been chosen sufficiently large according to Lemma \ref{lem:rank-limit}, and we drop $k$ from the notation. 
Since $\Quot(\Hcal, c(E))$ is a projective scheme, in the analytic topology the sequence $\ebold_i$ in \eqref{eqn:E-quot} converges to a limiting quotient. By definition of the equivalence of quotients, this means there is a sequence of quotients $\widehat q_i: \Hcal\to \widehat \Ecal_i\to 0$, isomorphisms $\varphi_i: \widehat \Ecal_i\isorightarrow \Ecal_i$, and a commutative diagram
$$
\xymatrix{
\Hcal \ar[r]^{\widehat q_i} \ar@{=}[d] &\widehat \Ecal_i \ar[d]^{\varphi_i}\ar[r] &0 \\
\Hcal \ar[r]^{\ebold_i} &\Ecal_i \ar[r]&0
}
$$
Furthermore, there is $\widehat q_\infty \in \Quot(\Hcal, c(E))$ such that $\widehat q_i\to \widehat q_\infty$.

\begin{lemma} \label{lem:quot-limit}
There is a map $\varphi_\infty : \widehat \Ecal_\infty\to \widetilde\Ecal_\infty$ making the following diagram commute:
$$
\xymatrix{
\Hcal \ar[r]^{\widehat q_\infty\ } \ar@{=}[d] &\widehat \Ecal_\infty \ar[d]^{\varphi_\infty}\ar[r]&0 \\
\Hcal \ar[r]^{\ebold_\infty\ } &\widetilde\Ecal_\infty\ar[r]&0
}
$$
Moreover, $\varphi_\infty$ is an isomorphism.
\end{lemma}

\begin{proof}
Since $\widehat q_\infty$ is surjective, to define the map $\varphi_\infty$ it suffices to show that $\ker \widehat q_\infty\subset \ker \ebold_\infty$.  In fact, it is enough to prove the inclusion away from a proper subvariety. For then the image $\ebold_\infty(\ker \widehat q_\infty)$ would be a torsion subsheaf of $\widetilde\Ecal_\infty$, and since $\widetilde\Ecal_\infty \subset \Ecal_\infty$ is torsion free, such a sheaf must vanish.
Now away from a set of codimension  at least 2 we have smooth convergence: $\ebold_i\to\ebold_\infty$, and $\widehat\pi_i\to \widehat\pi_\infty$ (see Lemma \ref{lem:quot-topology}). Hence,
$
0=\ebold_i\circ \widehat\pi_i\to\ebold_\infty\circ\widehat\pi_\infty
$,
and the result follows.
The map $\varphi_\infty$ is therefore well-defined. We claim it is in fact an isomorphism. First, since
$$\rank \imag \widehat q_\infty=
\rank \widehat \Ecal_\infty =\rank\Ecal_\infty=\rank\widetilde\Ecal_\infty=\rank\imag \ebold_\infty\ ,
$$
we have $\rank(\ker \widehat q_\infty)=\rank(\ker \ebold_\infty)$. 
By the discussion above, $\ker \widehat q_\infty\subset \ker \ebold_\infty$, and so 
\begin{equation} \label{eqn:torsion}
\ker \ebold_\infty/\ker \widehat q_\infty\cong \widehat q_\infty(\ker \ebold_\infty)\subset \Tcal\subset \widehat\Ecal_\infty\ ,
\end{equation}
where $\Tcal\subset \widehat\Ecal_\infty$ is the torsion subsheaf. Since $\widetilde\Ecal_\infty$ is torsion free, $\varphi_\infty$ extends to a map $\overline\varphi_\infty : \widehat\Ecal_\infty/\Tcal\to \widetilde\Ecal_\infty$.  On the other hand, from \eqref{eqn:torsion} there is a well-defined map: $\eta: \widetilde\Ecal_\infty\to \widehat\Ecal_\infty/\Tcal$ that is a generic inverse to $\varphi_\infty$.  
Both maps $\eta$ and $\overline\varphi_\infty$ are injective, since their kernels would be torsion sheaves, and $\widetilde\Ecal_\infty, \widehat\Ecal_\infty/\Tcal$ are torsion free. Hence, there is an exact sequence
\begin{equation*} \label{eqn:torsion-exact}
0\lra \widetilde\Ecal_\infty
\stackrel{\eta}{\xrightarrow{\hspace*{.75cm}}} \widehat\Ecal_\infty/\Tcal\lra \widehat \Tcal\lra 0\ ,
\end{equation*}
for a torsion sheaf $\widehat \Tcal$. It follows that, 
\begin{align*}
\chi_{\widetilde\Ecal_\infty}(\ell)&=\chi_{\widehat \Ecal_\infty}(\ell)-\chi_\Tcal(\ell)-\chi_{\widehat \Tcal}(\ell) \\
\chi_{\widetilde\Ecal_\infty}(\ell)-\tau(\ell)&=-\chi_\Tcal(\ell)-\chi_{\widehat \Tcal}(\ell)\ .
\end{align*}
By Lemma \ref{lem:rank-limit} (3), the left hand side of the last equation above is nonnegative, whereas the right hand side must be nonpositive. So $\chi_\Tcal(\ell)=\chi_{\widehat \Tcal}(\ell)= 0$, and 
therefore $\Tcal$ and $\widehat \Tcal$ are zero. Hence, $\varphi_\infty $ is an isomorphism.
\end{proof}

\begin{proof}[Proof of Proposition \ref{prop:quot-limit}] Immediate from Lemmas \ref{lem:quot-limit} and \ref{lem:rank-limit} (2).
\end{proof}

As a consequence of the main result above, the following proves the first step in the continuity of $\overline\Phi$.
\begin{prop} \label{prop:cont1} Let $q_i\in\overline \Zcal_{\circ}$ be a sequence whose quotients $\Fcal_i$  are locally free and $\mu$-stable. Assume that $q_{i}\rightarrow q_{\infty }\in \overline \Zcal_{\circ}$   
in the analytic topology, so that in particular the corresponding quotient $\Fcal_\infty$ is torsion free and $\mu$-semistable.   Let $A_i$ be the corresponding sequence of HYM connections on $E$. Then for any Uhlenbeck limit $A_\infty$ of $\{A_i\}$ 
 extending to a polystable reflexive sheaf $\Ecal_\infty$, we have $\ddual{\Gr(\Fcal_\infty)}\cong \Ecal_\infty$.
\end{prop}

\begin{proof}
As in Proposition \ref{prop:quot-limit}, write $q_{i}:\mathcal{H}\rightarrow \mathcal{E}_{i}\rightarrow 0$ for the sequence in $\overline \Zcal_{\circ}$ corresponding to the connections $A_{i}$ in the $\mathcal{G}^{\mathbb{C}}$ orbit of $\Fcal_i$. 
We have $ \Ecal_\infty\cong\ddual{(\widehat \Ecal_\infty)}$ for some quotient $\widehat{q}_{\infty }:\mathcal{H}\rightarrow \widehat{\mathcal{E}}_{\infty}\rightarrow 0$, which is a limit of $\widehat{q}_{i}:\mathcal{H}\rightarrow \widehat{\mathcal{E}}_{i}\rightarrow 0$ with $\widehat{\mathcal{E}}_{i} \cong \mathcal{F}_i$.

Note that we have proven in Lemma \ref{lem:quot-limit} above that 
\begin{equation}\label{eq:organisation_of_inclusions}
 \widehat{\mathcal{E}}
_{\infty } \cong \widetilde{\mathcal{E}}_{\infty } \subset \mathcal{E}_\infty.
\end{equation}
 Since $\mathcal{E}_{\infty }$ is polystable (having an admissible
Hermitian-Yang-Mills connection), and hence in particular $\mu $-semistable, this
implies that $\widehat{\mathcal{E}}_{\infty }$ is $\mu $-semistable, cf.\ Proposition~\ref{prop:polystablereflexive}. Note also that because $\widehat{\mathcal{E}}_{\infty }$ is isomorphic to $\Ecal_{\infty}$ away from a codimension two subvariety, their determinant line bundles are equal, and in particular the determinant line bundle of $\widehat{\mathcal{E}}_{\infty }$ is the same as that of $\widehat{\mathcal{E}}_{i} \cong \mathcal{F}_i$. Therefore, $\widehat{\mathcal{E}}_{\infty }$ defines a point in $
M^{\mu ss}$, and hence by construction also in $\overline{M}^{\mu }$.

 By construction, we have $[\widehat{\mathcal{E}}_{i}]=[\mathcal{F}_{i}]$ in $\overline M^\mu$. By continuity of the map $\pi :\Zcal\rightarrow M^{\muss}$, we therefore have $[\widehat{\mathcal{E}}_{\infty }]=[\mathcal{F}_{\infty }]$ in $
\overline M^\mu$. It hence follows from Proposition \ref{prop:GT} that $\ddual{\Gr(\Fcal_\infty)}\cong \Gr(\widehat{\mathcal{E}}_{\infty })^{\vee \vee }$. On the other hand, the fact that $\mathcal{E}_\infty$ is polystable together with \eqref{eq:organisation_of_inclusions} allows to apply Proposition \ref{prop:polystablereflexive} to obtain an isomorphism $\Gr(\widehat{\mathcal{E}}_{\infty })^{\vee \vee }\cong \mathcal{E}_{\infty }$, which completes the proof. 
\end{proof}

\subsection{Identification of the limiting cycle} \label{sec:agree-cycle-component}

 Having identified the reflexive extension of an Uhlenbeck limit of a sequence of smooth HYM connections, we now wish to determine the singular set as well. We will use the following singular version of the usual Bott-Chern formula, which is a slight generalisation of \cite[Thm.\ 1.3]{SibleyWentworth:15}.

\begin{prop} \label{prop:bc-formula}
Let $q_\Ecal, q_{\Fcal}\in \Quot(\Hcal,\tau)$, with $\Fcal$ torsion free and $\Ecal $ locally free. 
Assume that the underlying smooth vector bundles $F$ and $E$ of $\Fcal$ and $\Ecal$  on $X\backslash\sing\Fcal$ are isomorphic.
Set $\Ecal_\infty=\ddual\Fcal$, and let $\Ccal_\Fcal$ be the support cycle. Then for any smooth hermitian metric $h$ on $\Ecal$ and admissible metric $h_\infty$ on $\Ecal_\infty$, there is an equation of currents
\begin{equation} \label{eqn:bc}
\ch_2(\Ecal_\infty, h_\infty)-\ch_2(\Ecal, h)=\Ccal_\Fcal+ dd^c\Psi\ ,
\end{equation}
where $\Psi$ is a $(2,2)$-current, smooth outside the support of $\Ccal_\Fcal$.
\end{prop}

\begin{proof}
Choose representatives of the points in the Quot scheme:
\begin{align*}
q_\Ecal &: 0\lra \Kcal\lra \Hcal\lra\Ecal\lra 0\ , \\ 
q_\Fcal &: 0\lra \Kcal_\infty\lra \Hcal\lra\Fcal\lra 0 \ .
\end{align*}
Let $H, k$ be  fixed hermitian metrics on $\Hcal$ and $\Kcal$, respectively. By the Bott-Chern formula applied to the sequence $q_\Ecal$, there is a smooth $(1,1)$-form $\Psi_1$ such that
\begin{equation} \label{eqn:bc1}
\ch_2(\Kcal\oplus\Ecal, k\oplus h)-\ch_2(\Hcal, H)=dd^c\Psi_1\ .
\end{equation}
 Let $k_\infty$ be an admissible metric on the reflexive sheaf
$\Kcal_\infty$. Clearly, $k_\infty\oplus h_\infty$
 is then an admissible metric on $\Kcal_\infty\oplus \Ecal_\infty$.
By \cite[Thm.\ 1.3]{SibleyWentworth:15} applied to the sequence $q_\Fcal$, there is a $(1,1)$-form $\Psi_2$ such that
\begin{equation} \label{eqn:bc2}
\ch_2(\Kcal_\infty\oplus\Ecal_\infty, k_\infty\oplus h_\infty)-\ch_2(\Hcal, H)=dd^c\Psi_2+\Ccal_\Fcal\ .
\end{equation}
By the assumption on $F$ and $E$, the Chern connection for $(\dbar_{\Ecal_\infty}, h_\infty)$ on $X\backslash\sing\Fcal$ defines an admissible connection on the smooth bundle $E$.
It follows from the proof of \cite[Prop.\ 3.3]{SibleyWentworth:15} that $\ch_2(\Ecal_\infty, h_\infty)$ is a closed current representing $\ch_2(\Ecal_\infty)$ in rational cohomology. The essential point here is that although the proof contained in \cite[Prop.\ 3.3]{SibleyWentworth:15} does not apply to an arbitrary torsion free sheaf, it does apply as long as the underlying smooth bundle extends smoothly to all of $X$, which is guaranteed by the hypotheses. From \eqref{eqn:bc2}, it then follows that $\ch_2(\Kcal_\infty, k_\infty)$ is also a closed current. Appealing once more to the proof of \cite[Prop.\ 3.3]{SibleyWentworth:15} (recall that
$\Kcal_\infty$ is reflexive), $\ch_2(\Kcal_\infty, k_\infty)$ represents $\ch_2(\Kcal_\infty)$ in rational cohomology. Since $\Kcal$ and $\Kcal_\infty$ have the same Chern classes, it follows that there is a $(1,1)$-form $\Psi_3$ such that
\begin{equation} \label{eqn:bc3}
\ch_2(\Kcal_\infty, k_\infty)-\ch_2(\Kcal, k)=dd^c\Psi_3\ .
\end{equation}
Finally, \eqref{eqn:bc} follows by combining eqs.\ \eqref{eqn:bc1}, \eqref{eqn:bc2}, and \eqref{eqn:bc3}.
\end{proof}

The following result may be viewed as the analog of the main result in \cite{SibleyWentworth:15}.

\begin{prop} \label{prop:singset-limit}
In the notation of Proposition \ref{prop:cont1}, the cycle $\Ccal_{\Fcal_\infty}$ is equal to the analytic cycle $\Ccal^{an}$ associated to the Uhlenbeck limit $A_\infty$. Moreover, $\sing(\Ecal_\infty)=S(A_\infty)$.
\end{prop}

The rest of this section is devoted to the proof of Proposition \ref{prop:singset-limit}.
To simplify the notation, write: 
$
\Ccal^{alg} := \Ccal_{\Fcal_\infty}=\sum_j m_j^{alg} Z^{alg}_j
$.

\begin{lemma}\label{lem:cohomology}
In rational cohomology, $[\Ccal^{alg}]=[\Ccal^{an}]$.
\end{lemma}

\begin{proof}
By Proposition \ref{prop:cont1},  the HYM connection on $\ddual{\Gr(\mathcal{F}_{\infty })}$ is the Uhlenbeck limit $A_{\infty}$, and therefore eq.\ \eqref{eqn:cohom} implies that $[\ch_{2}(A_{\infty })]=\ch_{2}(E)+[\Ccal_{\mathcal{F}_{\infty }}]$. But recall also that \eqref{eqn:current} implies that $[\ch_{2}(A_{\infty })]=\ch_{2}(E)+[\Ccal^{an}]$ and the result follows.  
\end{proof}

Now we wish to go further and prove equality as cycles. The first observation
is that by the proof of Proposition \ref{prop:cont1} and Lemma
\ref{lem:quot-limit}, in $\overline{
M}^\mu$ we have $[\mathcal{F}_{\infty }]=[\mathcal{\tilde{E}}_{\infty }]$,
where $\widetilde{\mathcal{E}}_{\infty }\cong
 \widehat{\mathcal{E}}_{\infty }$ is the $\mu$-semistable sheaf that appears
in Lemma \ref{lem:rank-limit}. Then it follows from Proposition \ref{prop:GT} that
$\Ccal^{alg}=\mathcal{C}_{\widetilde{\mathcal{T}}_{\infty }}$,where $\widetilde{\mathcal{T}}_{\infty }$ is the torsion sheaf $Gr(\widetilde{\mathcal{E}}_{\infty })^{\vee \vee }/Gr(\widetilde{\mathcal{E}}_{\infty })$. 
By Proposition \ref{prop:polystablereflexive}, we obtain $\Ccal^{alg}=\mathcal{C}_{\widetilde{\mathcal{T}}_{\infty }}=\Ccal_{\mathcal{T}_{\infty}}$ where $\mathcal{T}_{\infty }=
\mathcal{E}_{\infty }/\widetilde{\mathcal{E}}_{\infty }$. Hence, it suffices to
show $\Ccal^{an}=\Ccal_{\mathcal{T}_{\infty }}$. 
The main technical result of this subsection is the following, which together with Lemma \ref{lem:cohomology} will prove Proposition \ref{prop:singset-limit}. 

\begin{prop}\label{prop:multiplicity} Let $Z\subset\supp(\mathcal{T}_{\infty })$ be an irreducible component of codimension 2. Write $m_{Z}$ for the multiplicity of $Z$ in the cycle $\Ccal_{\mathcal{T}_{\infty }}$. Choose a generic slice $\Sigma$. Then in the notation of Proposition \ref{prop:cont1}, we have:
\begin{equation*}
m_{Z}=\lim_{i\rightarrow \infty }\frac{1}{8\pi ^{2}}\int_{\Sigma }\left\{ 
\tr(F_{A_{i}}\wedge F_{A_{i}})-\tr(F_{A_{\infty }}\wedge
F_{A_{\infty }})\right\} .
\end{equation*}
In particular, since $m_{Z}$ is by definition positive, then by Lemma \ref{lem:slicing2},  $Z\subset S_b$. Therefore, $Z$ is a component of  the cycle $\Ccal^{an}$ with 
$
m_{Z}^{an}=m_{Z}
$.
\end{prop}

The proof of Proposition \ref{prop:multiplicity} will be an
adaptation of the proof of Proposition 4.2 of
\cite{SibleyWentworth:15}. It will use the Bott-Chern formula
of Section \ref{sec:agree-cycle-component}
 and the slicing lemmas of Section
\ref{sec:analytic-multiplicities}. Additionally we will require the following lemma, whose proof is identical to that of Lemma 4.3 of \cite{SibleyWentworth:15}.
\begin{lemma}\label{lem:blowup}
Let $\mathcal{V}\rightarrow X$ \ be a holomorphic vector bundle, $\mathcal{V}
\rightarrow \mathcal{F}\rightarrow 0$ a quotient, and $\mathcal{K}\subset 
\mathcal{V}$ the kernel. Write $\Gr(\mathcal{V})=\mathcal{K}\oplus \mathcal{F}
$. Let $Z\subset \supp(\ddual{\Gr(\mathcal{V})}/$ $\Gr(\mathcal{V}))$ be
an irreducible component of codimension $2$. Then there exists a sequence of
blow-ups along smooth complex submanifolds $\pi :\hat{X}\rightarrow X$ having
centre\textbf{\ }$\mathbf{C}$ and exceptional divisor $\mathbf{E}=\pi ^{-1}(
\mathbf{C)}$, and a subsheaf $\mathcal{\hat{K}}\subset \pi ^{\ast }\mathcal{V
}$ giving rise to an associated graded object $\Gr(\pi ^{\ast }\mathcal{V)=
\hat{K}}\oplus \mathcal{\hat{F}}$ with singular set $\sing(\Gr(\pi
^{\ast }\mathcal{V))}$ such that the following properties are satisfied:
\begin{enumerate}
\item $\Gr(\pi^{\ast }\mathcal{V)}\cong \Gr(\mathcal{V)}$ on $\hat{X}-
\mathbf{E}=X-\mathbf{C}$;
\item $\codim(Z\cap \mathbf{C)}\geq 3$;
\item $\codim(\pi (\sing(\Gr(\pi ^{\ast }\mathcal{V))}-Z))\geq 3$.
\end{enumerate}
\end{lemma}

\begin{proof}[Proof of Proposition \ref{prop:multiplicity}] 
Recall that $\mathcal{E}_{\infty }$ was an Uhlenbeck limit of
a sequence of
Hermitian-Einstein vector bundles
 $\mathcal{E}_{i}=(E,\bar{\partial}_{A_{i}}) $, which  in Section \ref{subsect:idealconnections}
 we realised as quotients
$\mathcal{H}\rightarrow \mathcal{E}_{i}\rightarrow 0$. In Lemma \ref{lem:quot-limit} we saw that there was 
a resulting quotient which we have been calling $\widetilde{\Ecal}_\infty$, and which we showed to be isomorphic to a certain limiting quotient which we called $\widehat{\Ecal}_\infty$. For notational reasons, during this proof we will rename this quotient
$q_{\mathcal{Q}_{\infty }}:\mathcal{H}\rightarrow \mathcal{Q}_{\infty }\rightarrow 0$, so that in particular $\ddual{\mathcal{Q}_{\infty }}\cong \mathcal{E}_{\infty }.$
We may therefore consider the
exact sequences
\begin{align*}
0\longrightarrow \mathcal{K}_{i}\longrightarrow &\mathcal{H}\longrightarrow 
\mathcal{E}_{i}\longrightarrow 0 \ ,\\
0\longrightarrow \mathcal{K}_{\infty }\longrightarrow &\mathcal{H}
\longrightarrow \mathcal{\Qcal}_{\infty }\longrightarrow 0\ .
\end{align*}
If we write $\Gr_{\infty }(\mathcal{H})$ for the associated graded object of
this latter sequence, then notice that $\supp(\mathcal{T}_{\infty })=\supp(\ddual{(\Gr_{\infty }( \mathcal{H})})/\Gr(\mathcal{H)})$, since $\mathcal{K}_{\infty }$ is saturated and therefore reflexive.

Now we are in a position to apply Lemma \ref{lem:blowup} and we obtain first of all an exact sequence
\begin{equation*}
0\longrightarrow \widehat{\mathcal{K}}_{\infty }\longrightarrow \pi ^{\ast }
\mathcal{H\longrightarrow }\widehat{\Qcal}_\infty\longrightarrow 0
\end{equation*}
having the properties stated there. 

Consider $\sum_{j}\hat{m}_{j}^{{alg}}\hat{W}_{j}$, the analytic cycle associated with the support of the sheaf $\ddual{\Gr_{\infty }(\pi^{\ast }\Hcal)}/\Gr_{\infty }(\pi ^{\ast }\mathcal{H)}$. By the first two parts of the Lemma \ref{lem:blowup}, we know that in fact, one of these irreducible components, say $\hat{W}_{1}$
is the proper transform $\widehat{Z}$ of $Z$, and so $\pi (\hat{W}_{1})=Z$ and so since at a generic
point of $\widehat{Z}$ the sheaf $\ddual{\Gr_{\infty }(\pi ^{\ast }
\Hcal)}/\Gr_{\infty }(\pi ^{\ast }\mathcal{H)}$
is isomorphic with $\ddual{(\Gr_{\infty }\mathcal{H})}/
\Gr_{\infty }(\mathcal{H)}$, we also have that $\hat{m}_{1}^{alg}=m_{Z}$. 

Since $X$ is K\"{a}hler, the class $[Z]$ is homologically
nontrivial, there is a class say $\alpha \in H_{4}(X,\mathbb{Z})$ such that the intersection product $\alpha \cdot \left[ Z\right] \neq 0$. By a classical result of Thom, an integral multiple of any $4$-dimensional
class can be represented by an embedded $4$-dimensional submanifold $\Sigma $,
 and by construction $\left[ \Sigma \right] \cdot \left[ Z\right] \neq 0$.
We may furthermore choose $\Sigma $ so that it intersects $Z$ transversely
in its smooth locus, and $\Sigma$ is analytic in 
a neighbourhood of $\Sigma\cap Z$. Therefore $\Sigma \cap Z=\{z_{1},\cdots ,z_{k}\}$,
where this set is nonempty. By possibly moving $\Sigma $ slightly, we may
suppose that $\Sigma $ intersects $Z$ in each point $z_{i}$ as a transverse
slice. By Lemma \ref{lem:blowup} (3), all
the other $\widehat{W}_{j}$ must map to a set of codimension greater than or
equal to $3$, and so must be contained in $\mathbf{E}$. Also, by the
codimension statement in $(3)$, we may assume that $\Sigma $ misses
the set $\pi (\sing(\Gr_{\infty }(\pi ^{\ast }\mathcal{H))}-Z$,
 and the proper transform $\widehat\Sigma$ of $\Sigma$ only
intersects $\sing(\Gr_{\infty }(\pi ^{\ast }\mathcal{H))}$ in 
$\pi ^{-1}(Z)$. On the other hand, by part $(ii)$ we may assume that $\Sigma 
$ misses $Z\cap \mathbf{C}$, so $\pi ^{-1}(\Sigma )$ and therefore $\widehat{
\Sigma }$ misses $\mathbf{E}\cap \pi ^{-1}(Z)$, and so $\widehat{\Sigma }$
empty intersection with all of the $\hat{W}_{j}$ except for $\widehat{Z}=
\widehat{W}_{1}$.

Since the quotient $\mathcal{Q}_{\infty }$ is torsion free, the
kernel $\mathcal{K}_{\infty }$ is saturated, and therefore reflexive, and so
has at most codimension $3$ singularities. Therefore, we may assume the
 $4$-dimensional submanifold $\Sigma $ misses the set
 $\sing\mathcal{K}_{\infty }$. In other words the restrictions $\mathcal{K}_i\bigr|_{\Sigma }$
and $\mathcal{K}_{\infty}\bigr|_{\Sigma }$ are smooth vector bundles on $\Sigma 
$. We have proven that $\mathcal{Q}_{\infty }$ is isomorphic to a
quotient in $\Quot(\mathcal{H},c(E))$, and in
particular the total Chern class of $\mathcal{Q}_{\infty }$ is the
same as that of the $\mathcal{E}_{i}$, and therefore the total Chern class
of $\mathcal{K}_{\infty }$ is the same as that of the $\mathcal{K}_{i}$. If $
\imath :\Sigma \hookrightarrow X$ is the inclusion map, then by naturality,
we have $c(\mathcal{K}_i\bigr|_{\Sigma })=\imath ^{\ast }\left( c(\mathcal{K}_{i}\right) )=\imath ^{\ast }\left( c(\mathcal{K}_{\infty }\right) )=c(
\mathcal{K}_{\infty}\bigr|_{\Sigma })$. Now the point is that the total Chern
class determines the topological type of a smooth vector bundle on a $4$-manifold,
 so $\mathcal{K}_i\bigr|_{\Sigma }$ and $\mathcal{K}_{\infty}\bigr|_{
\Sigma }$ are smoothly isomorphic. 

As in Proposition \ref{prop:bc-formula}, we fix a smooth
hermitian metric $k$ on $\mathcal{K}_{i}$ (smoothly these bundles are all
the same), and an admissible metric on $k_{\infty }$ on $\mathcal{K}_{\infty
}$. We will denote by $A_{\mathcal{K}_{i}}$ and $A_{\mathcal{K}_{\infty }}$
the corresponding Chern connections. Since $A_{\mathcal{K}_{i}}$ and $A_{
\mathcal{K}_{\infty }}$ restricted to $\Sigma $ live on the same smooth
bundle, there are Chern-Simons forms $\CS(A_{\mathcal{K}_{\infty}}\bigr|_{\Sigma},
A_{\mathcal{K}_{i}}\bigr|_{\Sigma })$ on $\Sigma $ such
that 
\begin{equation*}
\ch_2(\mathcal{K}_{\infty },k_{\infty })-\ch_2(\mathcal{K}_{i},k)=d\CS(A_{
\mathcal{K}_{i}}\bigr|_{\Sigma},A_{\mathcal{K}_{\infty }}\bigr|_{\Sigma})\ .
\end{equation*}
We will write $B_{\infty }=A_{\mathcal{K}_{\infty }}\oplus A_{\infty }$
which is the Chern connection associated to $\ddual{(\Gr_{\infty }
\Hcal)} =\left( \mathcal{K}_{\infty }\oplus \mathcal{
E}_{\infty },k_{\infty }\oplus h_{\infty }\right) $, and $B_{i}$ for the
Chern connection on $\Gr(\mathcal{H)}=(\mathcal{K}_{i}\oplus 
\mathcal{E}_{i},k\oplus h)$. We will furthermore write $B$ for the Chern
connection of $(\mathcal{H},H)$. As in \eqref{eqn:bc2},
\begin{equation*}
\ch_2(\mathcal{K}_{i}\oplus \mathcal{E}_{i},k\oplus h)-\ch_2(\mathcal{H}
,H)=dd^{c}\Psi _{i}=d\CS(B_{i},B)\ .
\end{equation*}

Now for each of the $z_{k}\in Z\cap \Sigma $, take a ball $B_{2\varepsilon
}(z_{k})$, choosing $\varepsilon $ small enough so that $B_{2\varepsilon
}(z_{k})\subset X-\mathbf{C}$ (which can again be achieved by part $(ii)$),
and $B_{2\varepsilon }(z_{k})\cap Z^{alg}\subset Z-\sing\left(
\ddual{(\Gr_{\infty }\Hcal)}\right)$. Then if we let $\psi $
be a cut-off function which is $1$ on the $B_{\varepsilon }(z_{k})$ and
supported in the $B_{2\varepsilon }(z_{k})$. We define admissible metrics $
\hat{k}_{\infty }=\psi k_{\infty }+(1-\psi )\hat{k}_{0}$ on $\widehat{
\mathcal{K}}_{\infty }$ and $\hat{h}_{\infty }=\psi h_{\infty }+(1-\psi )
\hat{h}_{0}$ on $\widehat{\mathcal{E}}_{\infty }:=\ddual{\widehat{\mathcal{Q}}_{\infty }}$ where $\hat{k}_{0}$ and $\hat{h}_{0}$ arbitrary
admissible metrics on $\widehat{\mathcal{K}}_{\infty }$ and $\widehat{
\mathcal{E}}_{\infty }$ respectively. Denote by 
$A_{\widehat{\mathcal{K}}_{\infty }}$ and $\widehat{A}_{_{\infty }}$ the Chern connections of these
metrics. By construction we have connection
 $A_{\widehat{\mathcal{K}}_{\infty }}=A_{\mathcal{K}_{\infty }}$ and $\widehat{A}_{\infty }=A_{\infty
} $ on $\cup _{z_{k}}B_{\varepsilon }(z_{k})$. Proposition \ref{prop:bc-formula} gives an equation of currents
\begin{equation}\label{eqn:Bott-Chern blowup}
\ch_2\left( \widehat{\mathcal{K}}_{\infty }\oplus
 \widehat{\mathcal{E}}_{\infty },A_{\widehat{\mathcal{K}}_{\infty }}\oplus \widehat{A}_{\infty
}\right) -\ch_2\left( \pi ^{\ast }\mathcal{H},\pi ^{\ast }B\right) =\sum_{j}
\widehat{m}_{j}^{alg}\widehat{W}_{j}+dd^{c}\widehat{\Psi}\ .
\end{equation}
On $\widehat{\Sigma }$ and away from $\widehat{Z}$, we also have $dd^{c}\widehat{\Psi
}=d\CS(\widehat{B},\widehat{B}_{\infty })$, where $\widehat{B}_{\infty }=A_{
\widehat{\mathcal{K}}_{\infty }}\oplus \widehat{A}_{\infty }$ and $\widehat{B
}=\pi ^{\ast }B$. Now by Lemma
 \ref{lem:slicing1} and \eqref{eqn:Bott-Chern blowup}, we obtain:
\begin{align*}
(\left[ \Sigma \right] \cdot \left[ Z\right] )m_{Z}&= 
\frac{1}{8\pi ^{2}}\int_{\Sigma \cap (\cup _{z_{k}}B_{\varepsilon
}(z_{k}))}\left\{\tr\left( F_{B}\wedge F_{B}\right) -\tr\left(
F_{B_{\infty }}\wedge F_{B_{\infty }}\right)\right\}  \\
&\qquad\qquad\qquad\qquad -\frac{1}{8\pi^{2}}
\int_{\Sigma \cap \partial (\cup _{z_{k}}B_{\varepsilon }(z_{k}))}d^{c}\widehat{
\Psi} \\
&=
\frac{1}{8\pi ^{2}}\int_{\Sigma \cap (\cup _{z_{k}}B_{\varepsilon
}(z_{k}))}\left\{\tr\left( F_{B}\wedge F_{B}\right) -\tr\left(
F_{B_{\infty }}\wedge F_{B_{\infty }}\right)\right\} \\
&\qquad\qquad\qquad\qquad +\frac{1}{8\pi ^{2}}\int_{
\widehat{\Sigma }\cap (\cup _{z_{k}}B_{\varepsilon }(z_{k}))^{c}}dd^{c}\widehat{
\Psi}\\
&=\frac{1}{8\pi ^{2}}\int_{\Sigma \cap (\cup _{z_{k}}B_{\varepsilon
}(z_{k}))}\left\{\tr\left( F_{B}\wedge F_{B}\right) -\tr\left(
F_{B_{\infty }}\wedge F_{B_{\infty }}\right)\right\} \\
&\qquad\qquad\qquad\qquad+\frac{1}{8\pi ^{2}}\int_{
\widehat{\Sigma }\cap (\cup _{z_{k}}B_{\varepsilon }(z_{k}))^{c}}d\CS(
\widehat{B},\widehat{B}_{\infty })\\
&=\frac{1}{8\pi ^{2}}\int_{\Sigma \cap (\cup _{z_{k}}B_{\varepsilon
}(z_{k}))}\left\{\tr\left( F_{B}\wedge F_{B}\right) -\tr\left(
F_{A_{\mathcal{K}_{\infty }}}\wedge F_{A_{\mathcal{K}_{\infty }}}\right)\right\}\\
 &\hspace{-.75in} -\frac{1}{8\pi ^{2}}\int_{\Sigma \cap (\cup _{z_{k}}B_{\varepsilon
}(z_{k}))}
\tr\left( F_{A_{_{\infty }}}\wedge F_{A_{_{\infty }}}\right)
 -\frac{1}{8\pi ^{2}}\int_{\Sigma \cap \partial (\cup
_{z_{k}}B_{\varepsilon }(z_{k}))}\CS(B,B_{\infty })\ .
\end{align*}
Similarly, we have 
\begin{align*} 
\frac{1}{8\pi ^{2}}&\int_{\Sigma \cap (\cup _{z_{k}}B_{\varepsilon
}(z_{k}))}\left\{\tr\left( F_{B}\wedge F_{B}\right) -\tr\bigl(
F_{A_{\mathcal{K}_{i}}}\wedge F_{A_{\mathcal{K}_{i}}}\bigr) -\tr
\bigl(F_{A_{_{i}}}\wedge F_{A_{_{i}}}\bigr)\right\}\notag \\
&=\frac{1}{8\pi ^{2}}\int_{\Sigma \cap (\cup _{z_{k}}B_{\varepsilon
}(z_{k}))}\left\{\tr\left(F_{B}\wedge F_{B}\right) -\tr\left(
F_{B_{i}}\wedge F_{B_{i}}\right)\right\}\notag \\
&=
\frac{1}{8\pi ^{2}}\int_{\Sigma \cap (\cup
_{z_{k}}B_{\varepsilon }(z_{k}))}d\CS(B,B_{i})
=\frac{1}{8\pi ^{2}}\int_{\Sigma \cap \partial (\cup
_{z_{k}}B_{\varepsilon }(z_{k}))}\CS(B,B_{i}) \ .
\end{align*}
Hence,
\begin{align*}
(\left[ \Sigma \right] &\cdot \left[ Z\right] )m_{Z} 
=\frac{1}{8\pi ^{2}}\int_{\Sigma \cap (\cup _{z_{k}}B_{\varepsilon
}(z_{k}))}\left\{\tr\bigl( F_{A_{\mathcal{K}_{i}}}\wedge 
F_{A_{\mathcal{K}_{i}}}\bigr) -\tr\bigl( F_{A_{\mathcal{K}_{\infty }}}\wedge F_{A_{
\mathcal{K}_{\infty }}}\bigr)\right\}  \\
&+\frac{1}{8\pi ^{2}}\int_{\Sigma \cap (\cup _{z_{k}}B_{\varepsilon
}(z_{k}))}\tr\left( F_{A_{_{i}}}\wedge F_{A_{_{i}}}\right) -
\tr\left( F_{A_{_{\infty }}}\wedge F_{A_{_{\infty }}}\right)  \\
&-\frac{1}{8\pi ^{2}}\int_{\Sigma \cap \partial (\cup
_{z_{k}}B_{\varepsilon }(z_{k}))}\CS(B,B_{\infty })+\frac{1}{8\pi ^{2}}
\int_{\Sigma \cap \partial (\cup _{z_{k}}B_{\varepsilon }(z_{k}))}\CS(B,B_{i})
\\
&=\frac{1}{8\pi ^{2}}\int_{\Sigma \cap (\cup _{z_{k}}B_{\varepsilon
}(z_{k}))}\tr\bigl( F_{A_{_{i}}}\wedge F_{A_{_{i}}}\bigr) -
\tr\bigl( F_{A_{_{\infty }}}\wedge F_{A_{_{\infty }}}\bigr)  \\
&+\frac{1}{8\pi ^{2}}\int_{\Sigma \cap \partial (\cup
_{z_{k}}B_{\varepsilon }(z_{k}))}\CS(B,B_{i})-\CS(B,B_{\infty })+\CS(A_{
\mathcal{K}_{i}}\bigr|_\Sigma ,A_{\mathcal{K}_{\infty }}\bigr|_\Sigma ) \\
&=\frac{1}{8\pi ^{2}}\int_{\Sigma \cap (\cup _{z_{k}}B_{\varepsilon
}(z_{k}))}\tr\bigl( F_{A_{_{i}}}\wedge F_{A_{_{i}}}\bigr) -
\tr\bigl( F_{A_{_{\infty }}}\wedge F_{A_{_{\infty }}}\bigr)  \\
&\qquad +\frac{1}{8\pi ^{2}}\int_{\Sigma \cap \partial (\cup
_{z_{k}}B_{\varepsilon }(z_{k}))}\CS(B_\infty,B_{i})+\CS(A_{
\mathcal{K}_{i}}\bigr|_\Sigma ,A_{\mathcal{K}_{\infty }}\bigr|_\Sigma )\ ,
\end{align*}
where we have used Lemma \ref{lem:Chern-Simons} (1).
Now by Lemma \ref{lem:slicing2}, the first term on the right hand side above converges to $(\left[ \Sigma \right] \cdot \left[ Z\right] )m_{Z}^{an}$ as $i\to\infty$, whereas
the expressions for $\CS(A_{\mathcal{K}_{\infty }}\bigr|_\Sigma ,A_{
\mathcal{K}_{i}}\bigr|_\Sigma )$ and $\CS(B_{i},B_{\infty })$ show that
the second integral vanishes. Hence,
$(\left[
\Sigma \right] \cdot \left[ Z\right] )m_{Z}=(\left[ \Sigma \right] \cdot 
\left[ Z\right] )m_{Z}^{an}$. 
Since $\left[ \Sigma \right] \cdot \left[ Z
\right] \neq 0$, this implies $m_{Z}=m_{Z}^{an}$.
\end{proof}

We now  have the

\begin{proof}[Proof of Proposition \ref{prop:singset-limit}] It follows from Proposition \ref{prop:multiplicity} that $\Ccal^{alg }=\Ccal_{\mathcal{F}_{\infty }}=\Ccal_{
\mathcal{T}_{\infty }}$ is a subcycle of $\Ccal^{an}$, so we may write
\begin{equation}\label{cyclediff}
\Ccal^{an}-\Ccal^{alg} =\sum_j (m_j^{an}- m_j^{alg}) Z_j^{an}\ ,
\end{equation}
with $m_j^{an}\geq m_j^{alg}$.
On the other hand, by Lemma \ref{lem:cohomology},
$
[\Ccal^{an}-\Ccal^{alg}]
$ vanishes
in cohomology. Hence,
\begin{equation*}
0=\lbrack \Ccal^{an}-\Ccal^{alg }]\cdot \lbrack \omega ^{n-2}] \\
=\sum_{j}(m_{j}^{an}-m_{j}^{alg
})\int_{Z_{j}}\omega ^{n-2}\geq \sum_{j}(m_{j}^{an}-m_{j}^{alg
})\ ,
\end{equation*}
and therefore 
 $m_{j}^{alg
}=m_{j}^{an}$ for all $j$.
This proves the first statement. The second statement follows as in the proof of Lemma \ref{lem:singset2}.
\end{proof}

\subsection{Continuity of $\overline\Phi$}\label{subsect:continuity} We now put together the results of the previous several subsections to prove continuity of the map $\overline\Phi$. The main result of this section is the following. 
\begin{thm}\label{thm:continuity}
The map $\overline\Phi: \overline M^\mu\to \widehat{M}_{\HYM}$ defined at the end of Section \ref{sec:map} is continuous and surjects onto $\overline M_{\HYM}$.
\end{thm}

We begin with the following.

\begin{lemma} \label{lem:images}
The images of the maps $\overline{\Psi }\ $and $\overline{\Phi }$ lie in the
gauge theoretic compactification $\overline{M}_{\HYM}$. 
\end{lemma}

\begin{proof}
It follows directly from Propositions \ref{prop:cont1} and \ref{prop:singset-limit} that if $q_{i}\rightarrow
q_{\infty }$ in $\overline \Zcal_{\circ}$, where $q_{i}:\mathcal{H}\rightarrow \mathcal{F}_{i}\rightarrow 0$ is a locally free, stable quotient for each $i$, then $
\lim_{i\rightarrow \infty }\overline{\Psi }(q_{i})=\overline{\Psi }
(q_{\infty })$. Since $\overline{\Psi }(q_{i})\in M_{\HYM}^{\ast }$
by construction, it therefore follows that $\overline{\Psi }(q_{\infty })\in 
\overline{M}_{\HYM}$. Since any element of $\overline \Zcal_{\circ}$ is a limit of
this kind, this implies that $\imag(\overline{\Psi })\subset
 \overline{M}_{\HYM}$ and hence that $\imag(\overline{\Phi })\subset 
\overline{M}_{\HYM}$ by \eqref{eq:relation_phi_psi}.
\end{proof}

\begin{prop}\label{prop:continuity}
The map $\overline{\Psi }:\overline \Zcal_{\circ}\rightarrow \overline{M}_{\HYM}$
is continuous.
\end{prop}

\begin{proof}
Let $q_{i}:\mathcal{H}\rightarrow \mathcal{F}_{i}\rightarrow 0$ be a
sequence of quotients in $\overline \Zcal_{\circ}$ converging to $q_{\infty }:\mathcal{H}\rightarrow \mathcal{F}_{\infty }\rightarrow 0$. 
We must show $\overline\Psi(q_i)\to \overline \Psi(q_\infty)$. By compactness it suffices to show that any subsequential limit of $\overline\Psi(q_i)$ equals $\overline\Psi(q_\infty)$.   Hence, without loss of generality we can assume $\overline\Psi(q_i)$ has an Uhlenbeck limit in $\overline M_\HYM$.
Now we don't assume that
the $\mathcal{F}_{i}$ are locally free, but for each $i$ there is a sequence
of quotients $q_{i,k}:\mathcal{H}\rightarrow \mathcal{F}_{i,k}\rightarrow 0$
with $\mathcal{F}_{i,k}$ locally free and $\mu $-stable of topological type $
E$, so that $q_{i,k}\rightarrow q_{i}$ in $\overline \Zcal_{\circ}$. 
As in the proof of the previous lemma, we have $\lim_{k\to \infty}\overline \Psi(q_{i,k})=\overline\Psi(q_i)$.
Choose a subsequence $\{k_i\}$ such that $q_{i, k_i}\to q_\infty$ in $\overline\Zcal_{\circ}$. As in the proof of Proposition \ref{prop:diagonal}, we may arrange that $\overline\Psi(q_{i,k_i})$ has the same Uhlenbeck limit as $\overline \Psi(q_i)$. 
Again using Propositions \ref{prop:cont1} and \ref{prop:singset-limit}  we conclude that $\overline\Psi(q_{i,k_i})\to \overline\Psi(q_\infty)$. The result follows.
\end{proof}

\begin{proof}[Proof of Theorem \ref{thm:continuity}]
We claim that $\pi|_{\overline{\mathcal{Z}_\circ}} : \overline{\mathcal{Z}_\circ} \twoheadrightarrow
\overline{M}^{\mu}$ is a quotient map in the analytic topology. 
Once this is established, $\overline{\Phi }$ is continuous if and only if $\overline{\Psi }$
is (cf.~\eqref{eq:relation_phi_psi}), and the latter holds by Proposition \ref{prop:continuity}. For surjectivity, we then note that by Proposition \ref{prop:diagonal}, any point in $\overline M_\HYM$ is a limit of points in the image of $\Phi$. Since $\overline M^\mu$ is compact and $\overline\Phi$ is continuous, surjectivity follows as well.

In order to see that $\pi|_{\overline{\mathcal{Z}_\circ}} : \overline{\mathcal{Z}_\circ} \twoheadrightarrow
\overline{M}^{\mu}$ is a quotient map in the analytic topology
we remark that as in the proof of \cite[Thm.\
A.7]{BuchdahlTelemanToma:17} it follows from Langton's
theorem, from the defining properties of the Quot scheme, and
from the observation that any holomorphic map from a smooth
variety into a reduced scheme $Y$ lifts to the weak
normalisation\footnote{Note also that the complex space
$(Y^{wn})^{an}$ associated with $Y^{wn}$ is the weak
normalisation of the complex space $Y^{an}$ associated with
$Y$.} $Y^{wn}$ of $Y$, that the pair
$(\overline{\Zcal_\circ}^{\mathrm{cl}},
\pi|_{\overline{\Zcal_\circ}}: \overline{\Zcal_\circ} \to
\overline{M}^{\mu})$, where $\overline{\Zcal_\circ}^{\mathrm{cl}}$ denotes the closure of $\overline{\Zcal_\circ}$ inside the weak normalisation of the Quot scheme, fulfills property $(\mathfrak{L})$ in the sense of \cite[Def.\ A.8]{BuchdahlTelemanToma:17}. The claim hence follows from \cite[Proposition A.9]{BuchdahlTelemanToma:17}
\end{proof}

\begin{rem}\label{rem:actually_compact}
As noted in the Introduction,
since $\overline{M}_{\HYM}$ is the image of a compact space under a continuous map, it is in fact compact rather than just sequentially compact. 
\end{rem}

\begin{cor}[\protect{cf.\ \cite[Thm.\ 5]{Li:93}}] \label{cor:continuity_in_surface_case}
If $\dim X = 2$, the map $\overline{\Phi} : \overline{M}^\mu \to \overline{M}_{\HYM}$ is a homeomorphism. 
\end{cor}

\begin{proof}
Theorem\ \ref{thm:continuity} implies that $\overline{\Phi}$ is continuous, proper, and onto; hence in order to prove the lemma it suffices to show that it is one-to-one. This however follows from \cite[Thm.~5.5 and Prop.~5.8]{GrebToma:17}.
\end{proof}

\begin{rem}\label{rem:identify_in_surface_case}
 As a consequence of the previous corollary, on surfaces we may identify the two moduli spaces $\overline{M}^\mu$ and $\overline{M}_{\HYM}$ using the map $\overline{\Phi}$. In particular, we may think of elements of $\overline{M}^\mu$ as equivalence classes $[(\mathcal{E}, \mathcal{C})] $ of pairs. Moreover, we note that it follows from the corollary that in the surface case $\overline{M}^\mu$ is homeomorphic (but not necessarily biholomorphic) to the compactification constructed by Jun Li in \cite[\S3]{Li:93}, the potential difference in complex structures stemming from the fact that we consider the determinant line bundle on the weak normalisation of the Quot scheme, not the Quot scheme itself.
\end{rem}

\section{A complex structure on $\overline M_\HYM$ } \label{sect:equivalence_relation}
In this section we prove Theorem \ref{thm:mainII} by analysing the natural finite equivalence relation on the space $\overline M^\mu$ provided by the map $\overline{\Phi}$. We do this by studying sections of the line bundle $\mathscr L_{n-1}$ in more detail. 

In Section \ref{sec:finite-relation} we define a natural
candidate for the structure sheaf on $\overline M_\HYM$ and
state a criterion for this to be the structure sheaf of an
actual complex space. In the following subsections we verify
this criterion in our case by constructing certain saturated
neighbourhoods (see Definitions
\ref{defi:admissible-hypersurface-pair} and
\ref{def:admissible-surface} and Proposition
\ref{prop:transversality-on-close-classes} of Section
\ref{subsect:admissible_flags}) around each point of
$\overline M^\mu$ and morphisms to projective spaces which are
constant on this relation (see Section
\ref{subsect:higher-dim_separation} and in particular Lemma
\ref{lemma:rational maps} and Proposition
\ref{prop:nonseparation}). These maps are constructed by
"lifting" sections from certain well chosen curves which lie
inside of certain complete intersection surfaces
 $S$ in $X$ to obtain sections of $\mathscr L_{n-1}$, (see Sections \ref{subsect:lifting_GT} and \ref{subsect:lifting_II}). We then check the maps are constant on the relation by comparing them to the corresponding (global) morphism on the moduli space of $\mu$-semistable sheaves on $S$ (see Proposition \ref{prop:nonseparation} and Section \ref{subsect:surface_case_separation}). The lifting procedure can only be performed away from a certain proper subvariety of $\overline\Zcal_{\circ}$, due to the loss of flatness of the universal sheaf upon successive restrictions. We therefore require that we be able to construct this non-flatness locus in such a 
way that
 the sections we obtain from lifting extend over it. We give
various results about restriction and flatness in Sections
\ref{subsect:restriction} and
\ref{subsect:extension_restriction}. Finally, to show the
criterion holds, we need to know that (products of) the locally defined maps in question separate the points of $\overline M_\HYM$. This will follow from Proposition \ref{prop:nonseparation} once we know that an element of $\overline M_\HYM$ is determined by its 
restriction to a finite number of appropriately chosen complete intersection surfaces $S_i$ (see Corollary \ref{cor:admissible_plus_lifting} and Lemma \ref{lemma: cycle separation}).

\subsection{The fibres of $\overline{\Phi}$ and the equivalence relation on $\overline M^\mu$} \label{sec:finite-relation}
We begin by introducing some standard terminology regarding equivalence relations.
\begin{defi}
 An equivalence relation $R$ on a locally compact (Hausdorff) space $X$ is called \emph{proper} if for every compact set $K \subset X$, its saturation $R(K)$ under the equivalence relation is compact. Equivalently, when endowed with the quotient topology, $X / R$ is locally compact (in particular, Hausdorff), and the quotient map $\pi: X \to X/R$ is proper. A proper equivalence relation with finite equivalence classes is called \emph{finite}.

\end{defi}

\begin{rem}
 If $X$ and $Y$ are locally compact spaces, and if $f: X \to Y$ is a continuous proper map, then the equivalence relation defined by $f$, i.e., $x \sim y$ if and only if $f(x) = f(y)$, is proper.
\end{rem}

Returning to the specific situation at hand, by Proposition \ref{prop:fibres}, the map $\overline\Phi: \overline M^\mu\to \overline M_{\HYM}$ is finite to one. Explicitly, for each $[(\mathcal{E},\mathcal{C})]$ in  $\overline M_\HYM$, there is an injection $\overline{\Phi} ^{-1}([(\mathcal{E},\mathcal{C})])\hookrightarrow \pi _{0}(\chi ^{-1}(
\mathcal{C))}$
where we recall that
\[
\chi : \Quot(\mathcal{E},\tau _{\mathcal{E}}-\tau _{E})\lra {\Cscr}_{n-2}(X)
\]
is the natural Quot to Chow morphism defined in Section \ref{sec:moduli}. 

We define an equivalence relation $R\subset \overline M^\mu\times \overline M^\mu$ on the space $\overline{M}^\mu
$, defined by the property $([\mathcal{F}_{1}],[\mathcal{F}_{2}])\in R$ if and only if $[\mathcal{F}_{1}]$ 
and $[\mathcal{F}_{2}]$ are in the same fibre of $\overline{\Phi }$. 
By Remark~\ref{rem:actually_compact},  $\overline M_\HYM$ is compact, and hence locally compact. By construction there is a bijection $\overline M^\mu/R\simeq 
\overline{M}_{\HYM}$,
 where $\overline M^\mu/R$ is the topological quotient of $
\overline M^\mu$ given by identifying all points in an equivalence class of $
R$. By Theorem \ref{thm:continuity} and the universal property of the quotient topology, this is a homeomorphism. We conclude that $R$ is a  finite equivalence relation. 

In order to put a complex structure on $\overline M_\HYM$ it therefore suffices to understand the equivalence relation $R$ and under what conditions the quotient of a complex analytic space by a proper equivalence relation can be endowed with a sheaf of functions making it into a complex space. As a first step in  this direction, we make $X/R$ into a ringed space: given a proper equivalence relation $R$ on a reduced complex space $X$ and an open subset $U \subset X/R$, we set $\mathcal{O}_{X/R}(U):= \mathcal{O}_X(\pi^{-1}(U))^R$, where the latter is the algebra of $R$-invariant holomorphic functions on $\pi^{-1}(U)$.

In order to see that the resulting ringed space $(X/R, \mathcal{O}_{X/R})$ is in fact a complex space, we will use the following fundamental result of Henri Cartan, see \cite[Main Theorem]{Cartan:60}.
\begin{thm}[{\sc Cartan's criterion}] \label{thm:Cartan_criterion}
 Consider a proper equivalence relation $R$ on a reduced complex space $X$ with quotient map $\pi: X \to X/R$. In order that the ringed space $X/R$  be a complex space, it suffices that each point of $X/R$ has an open neighbourhood $V$ such that the $R$-invariant holomorphic maps $\pi^{-1}(V) \to Z$ $($$Z$ being a complex space$)$ separate the equivalence classes in $\pi^{-1}(V)$.
\end{thm}

 Then, in order to prove that  $\overline M_\HYM$ is a complex space it suffices to prove that the equivalence classes of the relation $R\subset \overline M^\mu\times \overline M^\mu$ locally over $\overline M_\HYM$ can be separated by $R$-invariant holomorphic maps. We will in fact prove the stronger claim that locally over $\overline M_\HYM$ there exists holomorphic maps to some complex spaces having as fibres exactly the given equivalence classes (see Claim\ \ref{Claim}). The proof of this fact will occupy the remainder of Section~\ref{sect:equivalence_relation}. 

\subsection{Restriction theorems}\label{subsect:restriction}
Throughout the remainder of this section we will need to consider the restriction of sheaves to various subvarieties of $X$. We begin with the following crucial semistable restriction theorem. 

\begin{thm}[Langer, \cite{Langer:04}]\label{thm:langer} Let $X$ be a smooth projective variety, and let $\mathcal{O}_X(1)$ be an ample line bundle on $X$. 
Let $\{\mathcal{F}_{i}\}_{i\in \Scal}$ be a family of
 $\mu$-(semi)stable sheaves on $X$ with fixed rank and Chern classes. Then, there is a positive
integer $k_{0}$ which is independent of $i$, so that for any $k\geq k_{0}$
and any smooth divisor $D\in |\mathcal{O}_X(k)|$, if $(\Gr\mathcal{F}_{i})|_{D}$ is torsion free for some Seshadri graduation of $\Fcal_i$, then $\mathcal{F}_{i}|_{D}$ remains $\mu $-(semi)stable.
\end{thm}

This follows immediately from \cite[Thm 5.2 and Cor. 5.4]{Langer:04} by noticing
that the right hand side of the inequality provided there depends only on
the rank and Chern classes of the sheaf in question.

A higher-dimensional, slightly weaker analogue of the following result  will be proven in Proposition\ \ref{prop: flag existence} and Corollary\ \ref{cor:admissible_plus_lifting} below. 
\begin{prop} \label{prop: surface sheaf separation} 
Let $\mathcal{\{F}_{i}\}_{i\in \mathcal{S}}$  be a set of
 $\mu $-semistable sheaves on a polarised surface $(S, \mathcal{O}_S(1))$ with
fixed rank and Chern classes. Then there is a positive integer $k_0\gg0$ which is independent of $i$, such that for all $k \geq k_0$
 there is a finite subset $\Sigma\subset |\mathcal{O}_S(k)|$ consisting of smooth curves such that for any two sheaves $\mathcal{F}_{i_{1}},\mathcal{F}_{i_{2}}$, there exists a  curve $C\in \Sigma$ such that 
 \begin{enumerate}
  \item  $\sing (\Gr(\mathcal{F}_{i_1})) \cap C = \sing (\Gr(\mathcal{F}_{i_1})) \cap C =\emptyset$,
 \item $\mathcal{F}_{i_{1}}|_{C}$ and $\mathcal{F}_{i_{2}}|_{C}$ remain
semistable,
\item $\Gr(\mathcal{F}_{i_{1}})^{\vee \vee }\cong \Gr(\mathcal{F}
_{i_{2}})^{\vee \vee }$ if and only if $\mathcal{F}_{i_{1}}|_{C}$ and $
\mathcal{F}_{i_{2}}|_{C}$ are $s$-equivalent.
 \end{enumerate}
\end{prop}
\begin{proof}
Since the family $\{\mathcal{F}_{i}\}_{i\in \mathcal{S}}$ is in particular a
bounded family, the arguments of 
 Section \ref{sec:boundedness} imply that the number of singular
points of the sheaves $\{\Gr(\mathcal{F}_{i})\}_{i\in \mathcal{S}}$ is
bounded as $i$ ranges over the set $\mathcal{S}$. Here we have chosen
arbitrary Seshadri filtrations of the $\mathcal{F}_{i}$, but we recall
that the cycles associated to $\Gr(\mathcal{F}_{i})^{\vee \vee }/\Gr(\mathcal{F}_{i})$, and in particular the sets sing$(\Gr(\mathcal{F}_{i}))$, are
independent of this choice. Fix an upper bound $m$ for the maximum number of
points of such a singular set. Fix a number $k$ that satisfies the
hypotheses of Theorem \ref{thm:langer}, and choose $n>4m$ curves in $
|\mathcal{O}_X(k)|$ so that no three have a common intersection point. Then for any
two fixed sheaves $\mathcal{F}_{i_{1}}$ and $\mathcal{F}_{i_{2}}$, there
must be a curve $C$ among the chosen curves that misses the singularities of
both $\Gr(\mathcal{F}_{i_{1}})$ and $\Gr(\mathcal{F}_{i_{2}})$, and in
particular $(\Gr\mathcal{F}_{i_{1}})|_{C}$ and $(\Gr\mathcal{F}_{i_{2}})|_{C}$ are
locally free (and in particular torsion free) and therefore $\mathcal{F}_{i_{1}}|_{C}$ and $\mathcal{F}_{i_{2}}|_{C}$ are semistable.
Moreover, the summands of each of $\Gr(\mathcal{F}_{i_{1}})|_{C}$ and $\Gr(
\mathcal{F}_{i_{2}})|_{C}$ remain slope stable, all of equal slope, and are
the quotients of the restricted filtrants of the respective Seshadri filtrations. This means that the restrictions of the Seshadri
filtrations to $C$  remain Seshadri filtrations for 
$\mathcal{F}_{i_{1}}|_{C}$ and $\mathcal{F}_{i_{2}}|_{C}$. Notice that on $C$, the notion 
of Seshadri filtration is the same for
slope or GM-stability. We therefore obtain
\begin{align*}
\Gr(\mathcal{F}_{i_{1}})^{\vee \vee }|_{C} \cong \Gr(\mathcal{F}
_{i_{2}})^{\vee \vee }|_{C} 
&\Longleftrightarrow\ \Gr(\mathcal{F}
_{i_{1}})|_{C}\cong \Gr(\mathcal{F}_{i_{2}})|_{C} \\
\Longleftrightarrow\ \Gr(\mathcal{F}_{i_{1}}|_{C})\cong \Gr(\mathcal{F}
_{i_{2}}|_{C}) 
&\Longleftrightarrow\ \gr(\mathcal{F}_{i_{1}}|_{C})\cong \gr(
\mathcal{F}_{i_{2}}|_{C})\ .
\end{align*}
We claim that 
$$\Gr(\mathcal{F}_{i_{1}})^{\vee \vee }|_{C}\cong 
\Gr(\mathcal{F}_{i_{2}})^{\vee \vee }|_{C}
\Longleftrightarrow\Gr(\mathcal{F}_{i_{1}})^{\vee
\vee }\cong \Gr(\mathcal{F}_{i_{2}})^{\vee \vee }\ .
$$
 One direction is obvious. For the other, we consider the bounded family of locally free sheaves $\{\mathcal{E}_{i_{1},i_{2}}:=\mathcal{H}om(\Gr(\mathcal{F}_{i_{1}})^{\vee \vee },\Gr(
\mathcal{F}_{i_{2}})^{\vee \vee })\}_{i_{1},i_{2}\in \mathcal{S}}$. By possibly increasing the size of $k$, we can
ensure using Serre vanishing and duality that $H^{1}(S, \mathcal{E}_{i_{1},i_{2}}(-kH))=0$ for all
choices of $i_{1}$ and $i_{2}$. Now considering the exact sequence
\[
0\lra \mathcal{E}_{i_{1},i_{2}}(-kH)\lra \mathcal{E}_{i_{1},i_{2}}\lra \mathcal{E}_{i_{1},i_{2}}|_{C}\lra 0\ ,
\]
and the induced long exact sequence in cohomology, we see that $H^{0}(\mathcal{E}_{i_{1},i_{2}})$ surjects onto $ H^{0}(\mathcal{E}_{i_{1},i_{2}}|_{C})$. Therefore,
any isomorphism $\Gr(\mathcal{F}_{i_{1}})^{\vee \vee }|_{C}\cong \Gr(\mathcal{F}_{i_{2}})^{\vee \vee }|_{C}$ can be lifted to a map
 $\Gr(\mathcal{F}_{i_{1}})^{\vee \vee } \to \Gr(\mathcal{F}_{i_{2}})^{\vee \vee }$, which turns out to be an isomorphism for example by the arguments laid out in step 3 and 4 of the proof of Proposition 5.1 in the preprint version of \cite{GKP13}.
\end{proof}
Proposition\ \ref{prop: surface sheaf separation} is very close to being the same as \cite[Lemma 5.4]{GrebToma:17}, but
the important point that we wish to bring out here is that we may choose $k$
to be independent of the sheaves in question as we range over a bounded
family. The conclusion, however, is somewhat weaker. Namely, for each pair of sheaves in the family we only ask that there be one curve in the chosen finite set having 
the stated property holds, rather than this holding for a
 generic curve in the linear system.  

\subsection{Admissible flags and neighbourhoods in $\overline M_{\HYM}$}\label{subsect:admissible_flags}
In this subsection we give some preliminary definitions required to understand the case of higher dimensions in the next subsection. In the following we will try to exploit properties of the moduli space of semistable sheaves on projective surfaces. We will reduce ourselves to the surface case by successively cutting with (smooth) hypersurfaces in $X$. Recall that a hypersurface $X'$ in $X$ is {\em regular} for a coherent sheaf $\cF$ on $X$, or {\em $\cF$-regular}, if the natural morphism $\cF(-X')\to\cF$ is injective. This is the case if and only if $X'$ contains none of the associated points of $\cF$ (cf.\ \cite[page 8]{HuybrechtsLehn:10}). 

In order to motivate the next definition recall that by \cite[Formula V 3.20]{Kobayashi:87} the singularity sets of a coherent sheaf $\cS$ as defined in \cite[V.(5.5)]{Kobayashi:87} are related to the supports of the local Ext sheaves as follows:
$$S_m(\cS)=\bigcup_{d\ge n-m}\supp(\cE xt^d_{X}(\cS,\omega_{X})), \ 0\le m<n\ ;$$
cf.\  the discussion in Section\ \ref{subsubsect:singular_sets} above.

\begin{defi} \label{def:admissible-surface}
Let $[(\cE,\Ccal)]$ be a point in $\overline M_{\HYM}$, where $\cE$ is a polystable reflexive sheaf and $\Ccal$ is a codimension 2 cycle on $X$. 
 A smooth surface $S$ embedded in $X$ will be called  \emph{admissible for $[(\cE,\Ccal)]$} if $\dim(S\cap |\Ccal|)\leq 0$ and $\dim(S\cap \sing(\cE))\leq 0$.   
An admissible surface $S$ for $[(\cE,\Ccal)]$ will be called  \emph{fully admissible for $[(\cE,\Ccal)]$} if additionally $\cE$ is locally free in a neighbourhood of $S$. 
\end{defi}

\begin{defi}\label{defi:admissible-hypersurface-pair}
 Let $[(\cE,\Ccal)]$ be a point in $\overline M_{\HYM}$, where $\cE$ is a polystable reflexive sheaf and $\Ccal$ is a codimension 2 cycle on $X$.   
 A flag of smooth complete intersections  $(X^{(l)})_{1\le l\le n-2}$, $X^{(l)}:=\cap_{i=1}^lX_i$, of hypersurfaces $X_1,..., X_{n-2}$ in $X$  will be called
\emph{admissible for $[(\cE,\Ccal)]$} if for each $l=1,...,n-2$ both $X^{(l)}\cap |\Ccal|$ and $X^{(l)}\cap \sing(\cE)$ are of codimension at least two in $X^{(l)}$. An admissible flag $(X^{(l)})_{1\le l\le n-2}$ for $[(\cE,\Ccal)]$   will be called  
  \emph{fully admissible for $[(\cE,\Ccal)]$} if in addition for each $l=1,...,n-2$ the complete intersection $X^{(l)}$ is $\Ecal|_{X^{(l-1)}}$-regular and 
 $\cE xt^q_{X^(l-1)}(\cE|_{X^{(l-1)}},\omega_{X^{(l-1)}})$-regular for all $q\ge 0$, where we use the notation $ X^{(0)}=X$, $ X^{(1)}=X'$. 
 The flag $(X^{(l)})_{1\le l\le n-2}$ is said to be of \emph{multi-degree} $(d_l)_{1\le l\le n-2}$ if the hypersurfaces $X_l$ are taken in the linear systems $|\mathcal{O}_X(d_l)|$, where $\mathcal{O}_X(1)$ is a fixed ample line bundle on $X$.
 \end{defi}

\begin{defi} Let  $\cF$ be a torsion free sheaf on $X$. Similar to the above, a flag of complete intersections  $(X^{(l)})_{1\le l\le n-2}$, $X^{(l)}:=\cap_{i=1}^lX_i$ of smooth general hypersurfaces $X_1,..., X_{n-2}$ in $X$  will be called \emph{fully admissible for $\cF$} if for each $l=1,...,n-2$ the complete intersection $X^{(l)}$ is $\cF|_{X^{(l-1)}}$-regular and
 $\cE xt^q_{X^(l-1)}(\cF|_{X^{(l-1)}},\omega_{X^{(l-1)}})$-regular for all $q\ge 0$. Every member $X^{(l)}$ of a fully admissible flag will 
be called \emph{fully admissible} for $\Fcal$.
\end{defi}

\begin{rem}\label{rem:good_restriction} For later usage, we note the following facts:\\

\vspace{-0.3cm}
\noindent
(1) A flag of complete intersections  $(X^{(l)})_{1\le l\le n-2}$, $X^{(l)}:=\cap_{i=1}^lX_i$ of smooth general hypersurfaces $X_1,..., X_{n-2}$ in $X$  is admissible for 
$[(\cE,\Ccal)]$ if and only if its last member 
 $X^{(n-2)}$ is admissible for $[(\cE,\Ccal)]$.\\
\noindent (2) If a flag  $(X^{(l)})_{1\le l\le n-2}$ is fully admissible for 
$[(\cE,\Ccal)]$, then its last member 
 $X^{(n-2)}$ is fully admissible for $[(\cE,\Ccal)]$.\\
\noindent
(3)  If a flag  $(X^{(l)})_{1\le l\le n-2}$ is admissible for 
$[(\cE,\Ccal)]$, then  the condition that $X^{(l)}$ contains no irreducible
components of $|\mathcal{C}|\cap X^{(l-1)}$ implies that $|\mathcal{C}|\cap
X^{(l)}$ is the support of a codimension $2$ cycle on $X^{(l)}$, cf.\ \cite[Sect.\ 2.3]{Fulton:98}: Indeed, the first
hypersurface $X^{\prime }\subset X$, whose associated Cartier divisor we will denote by $\mathcal{D}$, contains no irreducible component $Z_{j}$ of $|\mathcal{C}|$. Hence, $
\mathcal{C\cdot D}:=\sum_{j}n_{j}(Z_{j}\cap X^{\prime })$ is a well-defined Cartier divisor on $|
\mathcal{C}|$. On the other hand,  $\mathcal{C\cdot D}$ also defines a well-defined intersection
cycle on $X^{\prime }$, which we will also denote by $\mathcal{C}|_{X'}$. Since it is a divisor on  $|\mathcal{C}|$ it must
be a codimension $2$ cycle $\mathcal{C}^{\prime}$ on $X^{\prime }$.
The hypothesis now similarly allows us to inductively construct codimension $2$ cycles 
$\mathcal{C}^{(l)}:=\mathcal{C}^{(l-1)}\cdot \mathcal{D}^{(l)}$ on $X^{(l)}$,
where $\mathcal{D}^{(l)}$ is the divisor associated to the hypersurface $
X^{(l)}\subset X^{(l-1)}$.\\
\noindent (4) If a flag  $(X^{(l)})_{1\le l\le n-2}$ is fully admissible for 
$[(\cE,\Ccal)]$, then the restrictions $\cE|_{X^{(l)}}$ are reflexive sheaves on their respective supports $X^{(l)}$ by \cite[Cor.\ 1.1.14]{HuybrechtsLehn:10}.  We will denote the induced pair $(\mathcal{E}|_{X^{(l)}},$ $\mathcal{C}^{(l)})$ consisting of a polystable sheaf and a codimension $2$ cycle on $X^{(l)}$ by $(\mathcal{E},\mathcal{C})|_{X^{(l)}}$.\\
\noindent (5)
 If a flag  $(X^{(l)})_{1\le l\le n-2}$ is fully admissible for a torsion free sheaf $\Fcal$, then for each $l=1,...,n-2$ the intersection $X^{(l)}\cap \sing(\cF)$ is of codimension at least two in $X^{(l)}$.
\end{rem}

The following result establishes a link between restriction of cycles and cycles associated with restricted sheaves via the construction given in Section\ \ref{subsubsect:support_cycles}. For simplicity we state it only for two-codimensional coherent sheaves.
\begin{lemma}\label{lem:restricted_cycles_cycles_of_restrictions}
 Let $\cA$ be a pure coherent sheaf of codimension two on $X$ and $X'$ a smooth hyperplane section of $X$ not containing any irreducible component of the support cycle $\cC_\cA$. Then, we have the following equality of cycles on $X'$:
\begin{equation}\label{eq:restriction}
\cC_{(\cA|_{X'})}=(\cC_\cA)|_{X'}\ .
\end{equation} 
\end{lemma}
\begin{proof} Using the definition of multiplicity as in \cite[Rem.\ 5.3]{GrebToma:17}, the statement is easily checked when $X'$ is a general element  in its linear system. 

Suppose now that $X'$  contains no irreducible component of the support cycle $\cC_\cA$ of $\cA$. This and the purity of $\cA$ imply the existence of an exact sequence of the form
 $$
0\to\cA(-X') \to \cA\to\cA|_{X'}\to 0\ .$$
Each element $X'_t$ in a sufficiently small Zariski open neighbourhood $T$ of $X'$ in its complete linear system will be smooth and will not contain any irreducible component of the support cycle $\cC_\cA$. By the above and by \cite[Lemma 2.1.4]{HuybrechtsLehn:10} it follows that the family of sheaves $\cA|_{X'_t}$ is flat over $T$. 

By the observation made at the beginning of the proof, for $t\in T$ general we have the equality
\begin{equation}\label{eq:general_restriction}
\cC_{(\cA|_{X'_t})}=(\cC_\cA)|_{X'_t}\ .
\end{equation}
Moreover, support cycles vary continuously in flat families of coherent sheaves, cf.\ \cite{BarletMagnussonII}, as do the intersection cycles appearing on the right hand side of \eqref{eq:general_restriction}, cf.\ part (3) of Remark\ \ref{rem:good_restriction}. Thus, we get the desired equality \eqref{eq:restriction} by passing to the limit on both sides of \eqref{eq:general_restriction}. 
\end{proof}

Next, we will introduce a terminology that is useful in formulating some of the technical results below. If $|H|$ is a basepoint free linear system on $X$ and $Y \subset X$ is a subvariety, we will denote by $|H|_Y$ the restricted linear system, i.e., the divisors arising from sections in the image of the restriction map $H^0(X, \mathcal{O}_X(H)) \to H^0(Y, (\mathcal{O}_X(H))|_Y)$.
\begin{defi}\label{defi:sufficiently_general}
 Let $\mathcal{O}_X(1)$ be an ample line bundle on $X$ with corresponding ample divisor $H$, and $d_1, \ldots, d_{n-2}$ positive natural numbers. A property $P$ is said to hold for \emph{any sufficiently general tuple} $(X_1, \ldots, X_{n-2}) \in |d_1 H| \times \ldots \times |d_{n-2} H|$ if there exists a nonempty open subset $U_1 \subset |d_1 H|$ such that for every $X_1 \in U_1$, there exists a nonempty open subset $U_2 \subset |d_2 H |_{X_1}$ such that for every $X_2 \in U_2$, there exists a nonempty open subset $U_3 \subset |d_3 H|_{X^{(2)}}$, ..., such that for   every $X_{n-2} \in U_{n-2}$ property $P$ holds for the tuple $(X_1, \ldots, X_{n-2})$.
\end{defi}

 If $H$ is a very ample divisor on $X$, if $d_1, \ldots, d_{n-2}$ are given positive natural numbers, and if $\mathcal{F}$ is a torsion free sheaf on $X$, then, since  the restriction of $|dH|$ to any subvariety of $X$ stays basepoint free for any positive natural number $d$, prime avoidance holds for these restricted linear systems, and hence any sufficiently general tuple $(X_1, \dots, X_{n-2}) \in |d_1 H| \times \dots \times |d_{n-2} H|$ is fully admissible for $\mathcal{F}$.

The following two results are higher-dimensional analogues of Proposition\ \ref{prop: surface sheaf separation} above and are important for producing the holomorphic maps required by Cartan's criterion, Theorem\ \ref{thm:Cartan_criterion}; see the proof of Claim~\ref{Claim} below.  
 
\begin{prop} \label{prop: flag existence} 
 Let $H$ be a very ample divisor on $X$, and let $k \geq 2$ be a positive integer. Let moreover $\mathcal{S}$ be a bounded set of reflexive sheaves on $X$ and let $[(\cE_0,\Ccal_0)]$ be a point in $\overline M_{\HYM}$, where $\cE_0$ is a polystable reflexive sheaf and $\Ccal_0$ is a codimension 2 cycle on $X$. 
 Then, there exists a positive natural number $N$ such that for any sequence $d_1, \dots, d_{n-2}$ of positive natural numbers and for any choice of $N$ sufficiently general tuples $(X_1, \cdots , X_{n-2}) \in  |d_1 H| \times \dots \times |d_{n-2} H|$  with associated flags $((X^{(l)})_{1 \leq l \leq n-2})$ the following holds: 
 \begin{enumerate}
  \item The flags are fully admissible for $[(\cE_0,\Ccal_0)]$. 
  \item If $\Sigma$ is the chosen finite set of flags, then for any choice of $k-2$ sheaves $\cE_{3}$,..., $\cE_{k}$ from $\mathcal{S}$ as well as any choice of two points $[(\cE_1, \Ccal_1)]$ and $[(\cE_2, \Ccal_2)]$ from $\overline{M}_{\HYM}$ some element of $\Sigma$ is in addition fully admissible for $[(\cE_1, \Ccal_1)]$ and $[(\cE_2, \Ccal_2)]$ as well as fully admissible for all the sheaves $\cE_3$,...,$\cE_k$. 
 \end{enumerate}
\end{prop}
\begin{proof} 
Let $n$ be the dimension of $X$, which we may suppose to be bigger than $2$. We start by choosing the first members $X'\in|d_1H|$ of the flags in $\Sigma$. The first constraint on $X'$ is given by the full admissibility for $[(\cE_0,\Ccal_0)]$, which is a Zariski open condition. In order to additionally ensure admissibility for $[(\cE_1, \Ccal_1)]$ and $[(\cE_2, \Ccal_2)]$ as well as for $\cE_3$,...,$\cE_k$, it is sufficient that $X'$ contains neither of the associated points of  the sheaves $\cE xt^q_{X}(\cE_i,\omega_{X})$, $i\in\{1,...,k\}$, $q\ge0$, nor any of the irreducible components of the cycles $\Ccal_1$, $\Ccal_2$; cf.\  part (5) of Remark\ \ref{rem:good_restriction}. For an element  $\cE\in\mathcal{S}$ denote by $A(\cE)$ the set of subvarieties of $X$ arising from such associated components, and define $A(\widetilde \cE, \widetilde \Ccal)$ analogously for $(\widetilde \cE, \widetilde \Ccal) \in \overline M_{\HYM}$.
 Since $\cE$ and $(\widetilde \cE, \widetilde \Ccal)$  run through bounded families, the arguments of 
 Section \ref{sec:boundedness} imply that the number of elements of the corresponding sets $A(\cE)$ and $A(\widetilde \cE, \widetilde \Ccal)$ are simultaneously bounded by some $m_1\in\N$. 
 
 Choose $n_1>knm_1$ smooth hypersurfaces in $|d_1H|$ fully admissible for $(\cE_0,\Ccal_0)$
and such that no $n+1$ of them have a common intersection point, both of which are open conditions. Then, it follows that for any choice of two points $(\cE_1, \Ccal_1)$ and $(\cE_2, \Ccal_2)$ from $\overline{M}_{\HYM}$ and any choice of $k-2$ sheaves $\cE_3$,...,$\cE_k$ from  $\mathcal{S}$, there must be a hypersurface  $X'$ among the chosen ones containing none of the elements of $A(\cE_1, \Ccal_1) \cup A(\cE_2, \Ccal_2) \cup A(\cE_3)\cup\cdots\cup A(\cE_k)$. 

The argument for higher codimension is analogous and hence omitted.
\end{proof} 
The following is a consequence of the proof of Proposition\ \ref{prop: flag existence}. 
\begin{cor}\label{cor:admissible_plus_lifting}
  Let $H$ be a very ample divisor on $X$, and let $[(\cE_0,\Ccal_0)]$ be a point in $\overline M_{\HYM}$, where $\cE_0$ is a polystable reflexive sheaf and $\Ccal_0$ is a codimension 2 cycle on $X$. Then, there exist a positive natural numbers $d_0$ and $N$  such that for any tuple $(d_1, d_2, \dots, d_{n-2})$ of natural numbers $d_j \geq d_0$ and for any choice of $N$ sufficiently general tuples $(X_1, \cdots , X_{n-2}) \in  |d_1 H| \times \dots \times |d_{n-2} H|$ with associated flags $((X^{(l)})_{1 \leq l \leq n-2})$ the following holds: 
  \begin{enumerate}
   \item All the flags are fully admissible for $[(\cE_0,\Ccal_0)]$. 
   \item If $\Sigma$ is the chosen finite set of flags, then for any choice of two points $[(\cE_1, \Ccal_1)]$ and $[(\cE_2, \Ccal_2])$ from $\overline{M}_{\HYM}$ some element $(X^{(l)})$ of $\Sigma$ is in addition fully admissible for $[(\cE_1, \Ccal_1)]$ and $[(\cE_2, \Ccal_2)]$. 
   \item If $X^{(n-2)}$ is a 2-dimensional member of any of the chosen flags, we have
   \vspace{-0.3cm}
   \begin{equation}\label{eq:iso_nach_einschraenkung}
  \cE_1 \cong \cE_2 \quad \quad \text{ if and only if } \quad \quad \cE_1|_{X^{(n-2)}} \cong \cE_2|_{X^{(n-2)}}.  \end{equation}
  \end{enumerate}
\end{cor}
\begin{proof}
 As the reflexive sheaves $\cE$ appearing as first entries for a tuple $[(\cE, \Ccal)] \in \overline M_{\HYM}$ range over a bounded set $\mathcal{S}$, the corresponding set  of reflexive homomorphism sheaves $\mathcal{H}om (\cE_1, \cE_2)$ is likewise bounded. We may therefore choose a positive natural number $d_0$ such that the conclusion of the "general Enriques-Zariski Lemma", \cite[Prop.\ 3.2]{MehtaRamanathan:81}, holds for any of the sheaves $\mathcal{H}om (\cE_1, \cE_2)$, where $\mathcal{E}_1$ and $\mathcal{E}_2$ are in $\mathcal{S}$.
 
  Let $d_1 \geq d_0$. Then, the first step of the proof of Proposition\ \ref{prop: flag existence} shows that for a general choice of $n_1$ smooth hypersurfaces $X'_j \in |d_1 H|$ the following property holds: for any choice $[(\cE_1, \Ccal_1)], [(\cE_2, \Ccal_2)] \in \overline{M}_{\HYM}$, there exists an index $j$ such that $X'_j$ is fully admissible for $[(\cE_1, \Ccal_1)], [(\cE_2, \Ccal_2)]$ and also for $\mathcal{H}om (\cE_1, \cE_2)$. It follows that $\cE_1|_{X'_j}$  , $\cE_2|_{X'_j}$, hence $\mathcal{H}om(\cE_1|_{X'_j}, \cE_2|_{X'_j})$,  and also $\mathcal{H}om (\cE_1, \cE_2)|_{X'_j}$ are reflexive. Moreover, admissibility implies that $X'_j$ does not contain any component of the singularity set of $\cE_1$ or $\cE_2$, and therefore the reflexive sheaves $\mathcal{H}om(\cE_1|_{X'_j}, \cE_2|_{X'_j})$ and 
  $\mathcal{H}om (\cE_1, \cE_2)|_{X'_j}$ agree on an open subset of $X'_j$ whose complement has codimension at least 2 in $X'_j$, and hence they coincide. Therefore, if $\cE_1|_{X'_j} \cong \cE_2|_{X'_j}$ via an isomorphism $\varphi$, by the choice of $d_1$ and by \cite[Prop.\ 3.2]{MehtaRamanathan:81} we can lift $\varphi$ to a homomorphism $\Phi \in \mathcal{H}om (\cE_1, \cE_2)$, which turns out to be an isomorphism for example by the arguments laid out in step 3 and 4 of the proof of Proposition 5.1 in the preprint version of \cite{GKP13}. We conclude that $\cE_1 \cong \cE_2$ if and only if $\cE_1|_{X'_j} \cong \cE_2|_{X'_j}$. 
  
  Noticing that the set of reflexive sheaves $\{\cE|_{X'_j} \mid \text{ $X'_j$ is fully admissible for $\cE$}\}$ is again bounded, we may argue in a similar fashion to see that the claims holds for any sufficiently general tuple $(X_1, \cdots , X_{n-2}) \in  |d_1 H| \times \dots \times |d_{n-2} H|$, as long as $d_j \geq d_0$ for all $j=1, \dots, n-2$.
\end{proof}
 
 The following crucial result allows us to define open neighbourhoods of 
a given  equivalence class in $\overline M^\mu$ on which the equivalence relation induced by the comparison map $\overline{\Phi}$ is controlled by invariant holomorphic functions.  
  \begin{Prop}\label{prop:transversality-on-close-classes}
 Let $S$ be a surface embedded in $X$. Then the set \[U'_{S} := \bigl\{[(\cE,\Ccal)] \in \overline M_{\HYM} \mid S \text{ is admissible for }[(\cE,\Ccal)] \bigr\} \subset \overline M_{\HYM} \]is open in $\overline M_{\HYM}$ with respect to the topology defined in Section\ \ref{sec:compactification}. In particular, the preimage $U_{S}:=\overline{\Phi}^{-1}(U'_{S})$ is open and $R$-saturated in $\overline M^\mu$. 
\end{Prop}

\begin{proof} 
By the definition of the topology of $\overline M_{\HYM}$, the claim is equivalent to the statement that if $S$ is  admissible for $[(\Ecal,\Ccal)]$ and $[(\Ecal_{i},\Ccal_{i})]$ is a sequence of (isomorphism classes of) ideal connections converging to $[(\Ecal,\Ccal)]$, then $S$ remains  admissible for $[(\Ecal_{i},\Ccal_{i})]$ for $i$ sufficiently large. This in turn breaks into two separate conditions on the pairs $[(\Ecal_{i},\Ccal_{i})]$: first, that $\dim(S\cap |\Ccal_{i}|)\leq 0$,  and second, that $\dim(S\cap \sing(\cE_i))\leq 0$, whenever these two properties are satisfied for $\Ccal$ and $\Ecal$. Since the $\Ccal_{i}$ converge to a subcycle of $\Ccal$ in the cycle space, their supports converge to the support of this subcycle in the Hausdorff sense, from which the first item follows. On the other hand, the second item is a consequence of Lemma \ref{lem:sing-limit}.
\end{proof}

\subsection{Extension of sections and flat restriction}\label{subsect:extension_restriction}

Below we  construct $\SL(V)$-invariant sections of the line
bundle $\mathscr{L}_{n-1}\rightarrow \mathcal{Z}$ away from a subvariety $
\mathcal{T}\subset \mathcal{Z}$. Since we want sections on all of $\mathcal{
Z}$, we will require the following lemma from \cite{GrebToma:17}, which gives a criterion for
when these sections extend over $\mathcal{T}$.  

\begin{lemma}[\protect{\cite[Lemma 2.12]{GrebToma:17}}] \label{lemma:extension}
Let $G$ be is connected algebraic group, $\mathcal{R}$ a weakly normal
 $G$-variety, and $\mathcal{L\rightarrow R}$ a $G$-linearised line bundle. Then
there exists a finite set of closed, irreducible $G$-invariant
subvarieties $\{\mathcal{R}_{i}\}_{i=1}^{m}$ of $\mathcal{R}$ with the
following property: if $\mathcal{T\subset R}$ is a closed $G$-invariant subvariety  with $\codim(\mathcal{T\cap R}_{i})\geq 2$ in $\mathcal{R}_{i}$, then any section $\sigma \in H^{0}(\mathcal{R}\backslash 
\mathcal{T},\mathcal{L)}^{G}$ extends to a section over all of $\mathcal{R}$.  
\end{lemma}

The following is a slight modification of \cite[Cor 3.2]{GrebToma:17}.

\begin{lemma} \label{flat restriction}
Let $G$ be a connected algebraic group, and $\Escr\rightarrow X\times 
\mathcal{R}$ be a $G$-linearised flat family over a parameter space $
\mathcal{R}$, and $H$ a very ample polarisation. Let $[(\mathcal{E},\mathcal{C})]\in \overline M_\HYM$. Suppose $\{\mathcal{R}_{j}\}_{j=1}^{m}$ is a
finite set of closed, irreducible $G$-invariant subvarieties of $\mathcal{R}$
so that for each $i$ there is a point $r_{i}\in $  $\mathcal{R}_{i}$ such
that $\Escr_{r_{i}}$ is torsion free. Then, for any choice of positive integers $d_1, \dots, d_{n-2}$ and any sufficiently general tuple $(X_1, \cdots , X_{n-2}) \in  |d_1 H| \times \dots \times |d_{n-2} H|$ the following holds: 
\begin{enumerate}
 \item The associated flag $(X^{(l)})_{1\leq
l\leq n-2}$ is fully admissible for $[(\mathcal{E},\mathcal{C})]$, (and so in particular the end term $S=X^{(n-2)}$ is fully admissible for $[(\mathcal{E},\mathcal{C})]$.
\item There exists
a closed subvariety $\mathcal{T\subset R}$ such that for each $j$ we have $\codim_{\Rcal_j}(\mathcal{T\cap R}_{j})\geq 2$, and for each 
$l$ the family $\Escr|_{X^{(l)}\times \mathcal{R}}$ remains flat over $
\mathcal{R}\backslash \mathcal{T}$.\footnote{This subvariety may depend on the flag.}
\item Each of the sheaves $\Escr
_{r_{i}}|_{X^{(l)}}$ remains torsion free.
\end{enumerate}
\end{lemma}
\begin{proof}
We will strictly follow the strategy of \cite[proof of Lemma 3.1]{GrebToma:17} with the additional constraint that the flags will be required to be fully admissible for $[(\cE,\Ccal)]$. 
In line with Definition\ \ref{defi:sufficiently_general}, the conditions we have to impose on the flag in order to fulfill the claims made in the formulation of the Lemma are analysed inductively, in particular, the problem of maintaining flatness of the successive restrictions 
of  the family $\Escr$ on $X\times \Rcal$  to the members of a flag $X^{(l)}\times \Rcal$. For this purpose, we define  \[\mathcal{R}_{tf}:=\{r\in \mathcal{R} \mid \mathcal{F}_{r} \text{ is torsion free on }X \} \ .\]
If $X'\subset X$ is a smooth hypersurface, then the restriction of $\Escr$ to $X'\times \Rcal_{tf}$ remains flat. In fact, flatness of the restricted family at a point $r\in \Rcal$ is equivalent to the $\Escr_r$-regularity of $X'$ by \cite[Lemma 2.1.4]{HuybrechtsLehn:10} and this in turn is implied by the torsion freeness of $\Escr_r$.

 We wish to construct a closed subvariety $\Tcal' \subset \Rcal$ whose intersection with any of the subvarieties $\Rcal_i$ has codimension at least 2 in $\Zcal_i$. This is secured by firstly choosing some "main reference" fibers $\Escr_{r_i}$, where the $r_i\in \Rcal_i$ are as in the statement of the lemma, and then choosing our hypersurface $X'$ such that $\Escr_{r_i}|_{X'}$ remains torsion free on $X'$. This is true for every element $X'$  in a dense open set of the linear system $|H|$ by \cite[Lemma 1.1.12 and Corollary 1.1.14]{HuybrechtsLehn:10}.  Now notice that $\Rcal_i\backslash \Rcal_{tf}$ is not dense in $\Rcal_i$ for any $i$, and that, by the previous discussion, the restriction $\Escr|_{X'\times(\Rcal_{tf}\cap \Rcal_i)}$  remains flat. However, it still may happen that $\codim_{\Rcal_i}(\Rcal_i\backslash \Rcal_{tf})=1$. 
To construct the set $\Tcal'$, whenever a set $(\Rcal_i)\backslash \Rcal_{tf}$ has codimension one in $(\Rcal_i)$ we choose some "secondary reference" point $s'_i\in \Rcal_i\backslash \Rcal_{tf}$  and ask that the hypersurface $X'$ also be $\mathscr U_{s'_i}$-regular, in addition to having the property that $\Escr_{r_i}|_{X'}$ remains torsion free on $X'$. This can still be achieved for $X'$ in a dense open subset of $|H|$, again by \cite[Lemma 1.1.12 and Corollary 1.1.14]{HuybrechtsLehn:10}. By the same reference, a dense open subset of such $X'$'s is in addition fully admissible for $[(\Ecal,\Ccal)]$ in the sense of Definition \ref{defi:admissible-hypersurface-pair}. The construction of the set $\Tcal' \subset \Rcal$ then proceeds in the same way as in the proof of Lemma 3.1 of \cite{GrebToma:17}, by setting
 \[
\mathcal{T}^{\prime }=\mathcal{R}\backslash \bigl(\mathcal{R}_{tf}\cup
\bigcup\nolimits_{i\in I}N(X^{\prime },r_{i}^{\prime }) \bigr)\ ,
\]
where $N(X^{\prime },r_{i}^{\prime })$ is an open neighbourhood of $
r_{i}^{\prime }$ in $\mathcal{R}$, where the restriction to $X^{\prime
}$ remains flat. This set has the desired property with respect to the $\Rcal_i$. Set $X'=X^{(1)}$. 

We will repeat this procedure to find successive elements of the flag. Let $\Rcal^{(1)}\subset \Rcal$ be the flatness locus of this restricted family, (i.e. $\Rcal^{(1)}=\Rcal\backslash\Tcal'$), so that $\Escr|_{X^{(1)}\times \mathcal{R}^{(1)}}$ is flat. We again choose secondary reference points $r_{i}^{(2)}\in \mathcal{R}^{(1)}\cap \mathcal{R}_{i}$. Then, for a general hypersurface $X_{2}\in |H|$ the intersection $X^{(2)}=X^{(1)}\cap X_{2}$ is
$\Escr_{r_{i}^{(2)}}$-regular for each $i$ and also fully admissible for $[(\mathcal{E},\mathcal{C})]$. As above, we can then
construct a nonflatness locus $\mathcal{T}^{(2)}\subset \mathcal{R}^{(1)}$
fulfilling codim$_{\mathcal{R}_{i}}(\mathcal{T}^{(2)}\cap \mathcal{R}_{i})\geq 2$
and such that, setting $\mathcal{R}^{(2)}=\mathcal{R}^{(1)}\backslash 
\mathcal{T}^{(2)}$, the restricted family $\Escr|_{X^{(2)}\times \mathcal{R}^{(2)}}$ remains flat. Now inductively, for each successive $\mathcal{R}^{(l-1)}$ we find secondary
reference points $r_{i}^{(l)}\in \mathcal{R}^{(l-1)}\cap \mathcal{R}_{i}$. 
Then, for a general hypersurface $X_{l}\in |H|$ the intersection $X^{(l)}=X^{(l-1)}\cap X_{l}$ is $\mathscr E_{r_{i}^{(l)}}$-regular and also fully admissible for $[(\mathcal{E},\mathcal{C)]}$. To each of these is associated a nonflatness
locus $\mathcal{T}^{(l)}\subset \mathcal{R}^{(l-1)}$ with codim$_{\mathcal{R}_{i}}(\mathcal{T}^{(l)}\cap \mathcal{R}_{i})\geq 2$. We then define 
$\mathcal{R}^{(l)}=\mathcal{R}^{(l-1)}\backslash \mathcal{T}^{(l)}$ so that $
\mathscr E|_{X^{(l)}\times \mathcal{R}^{(l)}}$ remains flat. We repeat this
procedure until we get down to $X^{(n-2)}=S$.

We conclude that a sufficiently general flag $
\{X^{(l)}\}$ is fully admissible for $[(\mathcal{E},\mathcal{C})]$ (and in
particular for $S$), and setting 
$
\mathcal{T=}\bigcup\nolimits_{l=l}^{n-2}\mathcal{T}^{(l)}
$, for the subsets $\mathcal{T}^{(l)}$ introduced in the procedure described above,
 that $\Escr|_{X^{(l)}\times \mathcal{R}\backslash \mathcal{T}}$
remains flat for all $l$, as desired.
\end{proof}
In the previous proof, note that we may keep the same "main reference" sheaves $\Escr_{r_i}$ on which we ask torsion freeness at any section, whereas we may need to choose new sets of "secondary reference" points $r_i^{(l)}\in (\Rcal^{(l)}\cap \Rcal_i)$ on which we ask regularity for the next section $X^{(l+1)}$, $1\le l\le n-3$. In particular, if $\Escr$ is the universal sheaf $\mathscr U\rightarrow X\times \mathcal{Z}$ the restriction to $X'$ remains flat since $\mathscr U_{s}$ is by definition torsion free for all $s\in \Zcal$. In this case, in the notation used above we have $\Zcal'=\Zcal$, however not all the fibers of the restricted family will remain torsion free on $X'$, so we may not apply the same argument to get flatness for the next restriction to $X^{(2)}\times \Zcal'$. It is therefore necessary to follow the argument above for all further restrictions. 

In the same vein, if $\mathscr F_{r}$ is torsion free for all $r$, and $X=S$ is a surface then there is no nonflatness locus $\Tcal$, since flatness is preserved by restricting to any curve $C \subset S$.

\begin{rem}\label{rem:can_arrange_Mehta_Ramanathan}
 If the $\mathscr{E}_{r_i}$ are not just torsion free but semistable, for any $d_1,   \dots,  d_{n-1}$ and any sufficiently general tuple $(X_1, \dots, X_{n-1}) \in |d_1 H | \times \dots \times |d_{n-1}H |$ the restriction of any of the associated graded sheaves $\Gr(\mathscr{E}_{r_i})$ to each term of the flag $X^{(l)}$ will remain torsion free. Hence, if the sequence $d_1, \dots,  d_{n-1}$ is sufficiently increasing, then by Langer's Theorem, Theorem\ \ref{thm:langer}, the restriction of each $\cE_{r_i}$ to $X^{(n-2)}$ will remain semistable. We will say in this case that the flag \emph{satisfies Langer's conditions}. 
\end{rem}

The following lemma follows directly from the proof of Lemma \ref{flat restriction}, but as we shall see, it is useful to separate out the case of curves. 
\begin{lemma} \label{restriction curves}
Let $\mathscr U \rightarrow X\times \mathcal{Z}$ be the universal sheaf, and let $[(
\mathcal{E},\mathcal{C})]\in \overline M_\HYM$. Let $\mathcal{Z}_{i}\subset \mathcal{Z}$ be the finite set of holomorphic subvarieties provided by
Lemma \ref{lemma:extension}, let $\{X^{(l)}\}_{1\leq l\leq n-2}$ be a sufficiently general flag of complete
intersections with $X^{(n-2)}=S$ together with the subvariety $\mathcal{T}$
produced by Lemma \ref{flat restriction}; and let $\mathcal{Z}^{(l)}\subset \mathcal{Z}$ be the subspaces from the proof of
Lemma \ref{flat restriction}. We also write $z_i\in \Zcal_i$ for the main reference points, so that in particular $\mathscr U_{z_{i}}|_{S}$ is torsion free. Suppose $C\subset S$ is a smooth curve such that there are points $z_{i}^{(n-1)}\in \mathcal{Z}^{(n-2)}\cap 
\mathcal{Z}_{i}$ such that $C$ is $U_{z_{i}^{(n-1)}}|_{S}$-regular. Then,
there is a closed holomorphic subvariety $\mathcal{T}_{C}\subset \mathcal{Z}$
 such that 
\begin{enumerate}
 \item $\codim_{\Zcal_i}(\mathcal{T}_{C}\cap \mathcal{Z}_{i})
\mathcal{\geq }2$ for each $i$, and
\item $\mathscr U|{_{C\times 
\mathcal{Z}\backslash \mathcal{T}_{C}\cup \mathcal{T}}}$ remains flat.
\end{enumerate}

\end{lemma}

\subsection{\protect{Lifting sections from curves I: the method of \cite{GrebToma:17}}}\label{subsect:lifting_GT}
The construction of the space $\overline M^\mu$ in \cite{GrebToma:17} requires the
existence of sections of $\mathscr L_{n-1}^{\otimes\kappa}$, for some power of $k$.
In particular if one considers all such sections for all powers of $k$ then
these provide a map to projective space defining $\overline M^\mu$. In
comparing this space to $\overline M_\HYM$, instead of considering all
sections we will need to restrict our attention to a certain linear system
of sections that comes from "lifting" sections from complete intersection
curves. We therefore need to review the procedure by which \ sections of $
\mathscr{L}_{n-1}^{\otimes\kappa}$ are constructed in the first place.

The main idea is that for a smooth curve $C\subset X$,
GM- and slope semistability of sheaves on $C$ are equivalent, which
means that the point in the Quot scheme parametrising quotients on $C$
corresponding to a $\mu $-semistable sheaf on $C$ is GIT semistable for the
appropriate determinant line bundle on $C$, provided that this line bundle
can be constructed (see Section 2.6), and in particular this is the case for any complete intersection curve $C=X^{(n-1)}$.

What is required for the construction is that the restriction of the universal family to $C$ remain flat. 
By Lemma \ref{restriction curves}, this can be achieved away from a certain proper
subvariety $\Tcal\cup\Tcal_C\subset \Zcal$, provided $C$ is well chosen. GIT then produces nonvanishing sections
of this determinant line bundle (see Lemma \ref{lemma: separation curves} below), which can then be lifted to give sections of a power of
of $\mathscr L_{n-1}$ away from $\Tcal\cup\Tcal_C$ (see Lemma \ref{lemma:lifting} below). One then has to know that $\Tcal\cup\Tcal_C$ can be
chosen in such a way that these sections extend over $\Tcal\cup\Tcal_C$, to give sections
on all of $\Zcal$. This will be obtained from the stated properties of $\Tcal\cup\Tcal_C$ in Lemma \ref{restriction curves} by applying Lemma \ref{lemma:extension}. 

We now make this more precise. Consider the universal family $\mathscr U\rightarrow X\times \mathcal{Z}$, and choose a flag $\{X^{(l)}\}_{1\leq l\leq n-2}$ as in Lemma \ref{flat restriction}. In particular, write $X^{(n-2)}=S$ for the end-term of this flag. Now we extend this flag to a flag $\{X^{(l)}\}_{1\leq l\leq n-1}$ by choosing a curve $X^{(n-1)}=C\subset S$. 

 Let $c\in K(X)_{\rm num}$ be a class with $
c_{i}=c_{i}(c)=c_{i}(E)$. Write $c^{(n-1)}=c|_C$, and $\tau _{E,C}$ for the
the Hilbert polynomial determined by $c^{(n-1)}$ with respect to $\Ocal_C(1)$. Take $m^{(n-1)}$ to be a (large) positive integer, and write 
$V_{C}=\mathbb{C}^{\tau _{E,C}(m^{(n-1)})}$, and correspondingly $\mathcal{H}_{C}=V_{C}\otimes \Ocal_{C}(-m^{(n-1)})$, and let $Q_{C}\subset\Quot(\mathcal{H}_{C},\tau _{E,C})$ be the subscheme consisting 
 of $m^{(n-1)}$-regular quotients $q_F: \mathcal{H}_{C} \twoheadrightarrow F$ with determinant $\Jcal |_{C}$ such that $q_F$ induces an isomorphism $V_C \overset{\cong}{\longrightarrow} H^0\bigl(F (m^{(n-1)})\bigr)$. 
Applying the discussion of Section~\ref{subsect:quotschemes_and_bundles} to the universal quotient $\Hcal_{C}\otimes \Ocal_{C}\rightarrow \widehat{\mathscr U}_{C}\rightarrow 0$  we obtain a line bundle $\mathscr {L}_{\widehat{U},0}:=\mathscr {L}_{0,C}\rightarrow Q_{C}\in \Pic(Q_{C})$.

The point is that because the two notions of stability are equivalent on $C$, GIT semistability of points in $Q_{C}$ suffices to produce sections of a sufficiently large tensor power of $\mathscr {L}_{0,C}$. Explicitly, we have the following result, see \cite[top of p.\ 223]{HuybrechtsLehn:10}, \cite[Lemma 3.8]{GrebToma:17}, and especially \cite{PavelThesis} 
for detailed explanations. 

\begin{lemma} \label{lemma: separation curves}
For a sufficiently large choice of $m^{(n-1)},$ for any quotient $q_{
\mathcal{F}}:\mathcal{H}_{C}\rightarrow \mathcal{F}$ giving a point $q_{
\mathcal{F}}\in Q_{C}$, the sheaf $\mathcal{F}$ is semistable with the
property that the natural map $V_{C}\rightarrow H^{0}(\mathcal{F}(m^{(n-1)}))
$ is an isomorphism iff $q_{\mathcal{F}}$ is $GIT$ semistable for the $
\SL(V_{C})$ linearisation of $\mathscr l{L}_{0,C}$, i.e., iff there exists a positive
integer $\kappa $, and an $\SL(V_{C})$ invariant section $\sigma \in H^{0}(
\mathscr {L}_{0,C}^{\otimes \kappa })$ that does not vanish at $q_{\mathcal{F}}$.
 Furthermore, suppose $q_{\mathcal{F}_{1}},q_{\mathcal{F}_{2}}\in Q_{C}$
are quotients satisfying any of the equivalent conditions above, so that in particular
the sheaves $\mathcal{F}_{1}$ and $\mathcal{F}_{2}$ are semistable. Then
there is an $\SL(V_{C})$ invariant section
 $\sigma \in H^{0}(\mathscr{L}_{0,C}^{\otimes \kappa })$ for some positive integer $\kappa $ such that $
0=\sigma (\mathcal{F}_{1})\neq \sigma (\mathcal{F}_{1})$, $($that is, sections
of some tensor power of $\mathscr{L}_{0,C}$ separate $q_{\mathcal{F}_{1}}$
and $q_{\mathcal{F}_{2}}$$)$ iff $\mathcal{F}_{1}$ and $\mathcal{F}_{2}$ are
not $s$-equivalent. 
\end{lemma}

We now state the following key lemma, which produces sections of a tensor power of $
\mathscr{L}_{n-1}$ from the sections given by the previous lemma. This is implicit in \cite{GrebToma:17}, but we find it useful to formulate an explicit statement concerning this point here.

\begin{lemma} \label{lemma:lifting}
Let $C\subset S$ be a curve satisfying the hypotheses of Lemma \ref{restriction curves}, so that in particular $C$ is the end-term of a flag $\{X^{(l)}\}_{1\leq
l\leq n-1}$ such that $\mathscr U|_{C\times \mathcal{Z}\backslash (\mathcal{T\cup T}_{C})}$ is flat. 
There exist positive integers $\kappa_0, \kappa_{n-1}$ such that for all $k \in \mathbb{N}_{>0}$  there exist natural linear "lifting" maps
\[\bar{L}_{\mathscr U,C}:H^{0}(Q_{C},\mathscr L_{0,C}^{\otimes k \kappa _{0}})^{\SL(V_{C})}\lra
H^{0}(\mathcal{Z},\mathscr L_{n-1}^{\otimes k \kappa _{n-1}})^{\SL(V)}
\]
such that the following holds:
If $(q_{\Fcal}: \mathcal{H} \to \mathcal{F} )  \in\Zcal \backslash (\mathcal{T} \cup \mathcal{T}_C)$ is a quotient, $q_{\mathcal{F}_{C}}:\mathcal{H}_{C}\rightarrow \mathcal{F}_{C}$ is a quotient realisation of $  \mathcal{F}_{C} = \mathcal{F}|_C$, and $\sigma \in H^{0}(Q_C, \mathscr L_{0,C}^{\otimes k\kappa_0
})^{Sl(V_{C})}$ is a section, then $\bar{L}_{\mathscr U,C}(\sigma )(q_{\mathcal{F}})\neq 0$ if and only if $\sigma (q_{\mathcal{F}_C}) \neq 0$.
In particular, if $q_{\Fcal}\in\Zcal$ is a quotient such that $\Fcal|_{X^{(l)}}$ remains semistable for all members of the flag, then for $k$ large enough there is a section of $H^{0}(\mathcal{Z},\mathscr L_{n-1}^{\otimes k \kappa _{n-1}})^{\SL(V)}$  that does not vanish at $q_{\Fcal}$.
\end{lemma}
\begin{proof}
The construction of the lifting map is in \cite [Sect.\  3.3]{GrebToma:17}, the numbers $\kappa_{n-1}$ and $\kappa_0$ are determined by \cite[Eq.~(3.15)]{GrebToma:17}. This gives a section in the space $H^{0}(\mathcal{Z}\backslash \mathcal{T\cup T}_{C},\mathscr{L}_{n-1}^{\otimes\kappa_{n-1}})^{\SL(V)}$. By the construction of the sets $\Tcal$ and $\Tcal_{C}$, and Lemma \ref{lemma:extension}, these sections extend to all of $\Zcal$. The second statement follows directly from the existence of the lifting map and Lemma \ref{lemma: separation curves}. 
\end{proof}

We will refer to elements of $H^{0}(\mathcal{Z},
\mathscr {L}_{n-1}^{\otimes k\kappa _{n-1}})^{\SL(V)}$ that are lifted from sections in $
H^{0}(Q_{C},\mathscr {L}_{0,C}^{\otimes k\kappa _{0}})^{\SL(V_{C})}$ in this
way as \emph{$\theta$-functions of level $k \kappa_{n-1}$ lifted from $C$}. Note that since the moduli space $M^{\mu ss}$ is
constructed using sections of the $\mathscr {L}_{n-1}^{\otimes k \kappa _{n-1}}$,
any $\theta$-function $\sigma $ with $\sigma (q_{\mathcal{F}})\neq 0$
descends to a section of some tensor power\footnote{See Diagrams \eqref{eq:comparison_diagram} and \eqref{eq:moduli_comparison} below for a computation of the exact tensor power.} of $\mathcal{O}_{M^{\mu ss}}(1)$, which does not vanish at $[\mathcal{F}]$. We will continue to refer to these induced sections as \emph{$\theta$-functions $($lifted from $C$$)$}.

\subsection{\protect{Lifting sections from curves II: the method of \cite{Li:93}}} \label{subsect:lifting_II} We collect here some information on an alternative way of "lifting" sections (cf.~\cite[p.~433-4]{Li:93}), and its compatibility with the construction described in the previous subsection. This will later allow us to use very fine results concerning separation properties of sections in determinant line bundles for families of sheaves on surfaces obtained by Jun Li in \cite{Li:93} (see the proof of Proposition~\ref{prop:embedding-via-admissible-curves-surface-case} below). 

As mentioned in Section
 \ref{sec:moduli}, 
the space $M^{\muss}$ comes equipped with an ample line bundle $\mathcal{O}_{M^{\muss}}(1)$ and a number $N=N_{n-1}$ with the property that any flat family $\Escr \rightarrow X\times T^\circ$ of semistable sheaves
(with $T^\circ$ weakly normal) gives rise to a unique classifying morphism $\psi _{\Escr}:T^\circ\rightarrow M^{\mu
ss}$, so that \begin{equation}\label{eq:functorial_property_of_class_map}
    \psi _{\Escr}^{\ast }(\mathcal{O}_{M^{\mu
ss}}(1))=\lambda _{\Escr}(u_{n-1}(c,[\mathcal{O}_{H}]))^{\otimes N}.            
                 \end{equation}
 Without loss of generality, as explained in \cite[Sect.\ 4.2]{GrebToma:17}, we may assume that the section ring of $\mathcal{O}_{M^{\mu ss}}(1)$ is generated in degree one. This induces a "lifting" map  $\psi _{\Escr}^{\ast }:H^0(M^{\mu ss},\mathcal{O}_{M^{\mu ss}}(1))\to H^0(T^\circ,\lambda _{\Escr}(u_{n-1}(c,[\mathcal{O}_{H}]))^{\otimes N})$. 
It is clear that if the family $\Escr \rightarrow X\times T^\circ$ is linearized with respect to some $G$-action on $T^\circ$, then $\psi _{\Escr}$ is $G$-invariant and so are the sections in $\psi _{\Escr}^{\ast }(H^0(M^{\mu ss},\mathcal{O}_{M^{\mu ss}}(1)))$.
The above applies in particular to the case when  $T^\circ=\Zcal$, $\Escr= \widetilde{\mathscr U}$ and $G=\SL(V)$ and the induced map $\psi^*_{  \widetilde{\mathscr U}}$ is an isomorphism. It also applies to the case when $X=C$ is  a curve, $T^\circ=Q_C^{ss}$, $\Escr$ is the restriction of the universal quotient bundle $\widehat{\mathscr U}$ to $C\times  Q_C^{ss}$ and $G=\SL(V_C)$ giving an isomorphism $\psi _{\widehat{\mathscr U}}^{\ast }$.

 Let now $P\to T^\circ$ denote the frame bundle associated to the locally free sheaf $pr_{2,*}(\Escr(l))$. We have a commutative diagram of natural morphisms
 $$
 \xymatrix{
 P\ar[d]^p\ar[r]^q& \Zcal\ar[d]\\
 T^\circ\ar[r]& M^{\mu ss}
 }
 $$  
As in \cite[page 224]{HuybrechtsLehn:10} and \cite[proof of Lemma 3.8]{GrebToma:17} by pulling back through $q$ and pushing down through $p$ for all $\kappa \in \mathbb{N}_{>0}$ we get a linear map
$$L_\Escr:  H^{0}(\mathcal{Z},\mathscr L_{n-1}^{\otimes \kappa N})^{\SL(V)}\lra H^0(T^\circ,\lambda _{\Escr}(u_{n-1}(c,[\mathcal{O}_{H}]))^{\otimes \kappa N})$$
which by construction sits in the following commutative diagram
 $$
 \xymatrix{
 &&    H^{0}(\mathcal{Z},\mathscr L_{n-1}^{\otimes \kappa N})^{\SL(V)}   \ar[dll]_{L_\Escr}\\
  H^0(T^\circ,\lambda _{\Escr}(u_{n-1}(c,[\mathcal{O}_{H}]))^{\otimes \kappa N})& &
 H^0(M^{\mu ss},\mathcal{O}_{M^{\mu ss}}(\kappa))\ar[u]_{\psi _{\widetilde{\mathscr U}}^{\ast }}\ar[ll]^{\psi _{\Escr}^{\ast } }
 }
 $$  
showing that in this case the two ways of lifting sections are compatible.

 Let now $X$ be $n$-dimensional and let $\{X^{(l)}\}_{1\leq l\leq n-1}$ be a flag of complete
intersections with $X^{(n-2)}=S$ and $X^{(n-1)}=C$.
Denote by $E_{X^{(i)}}$ and $\Jcal_{X^{(i)}}$ the restrictions
of the smooth bundle $E$ and holomorphic bundle $\Jcal$ to $X^{(i)}$.
 Let $T^\circ$ be a weakly normal base of a flat family  $\Escr \rightarrow X\times T^\circ$ of semistable sheaves, such that its restrictions $\Escr^{(l)} \rightarrow X^{(l)}\times T^\circ$ to all terms of the flag remain flat families of semistable sheaves.  By the same arguments as above and  in a canonical way we obtain a map:
\[
L_{\mathscr E,C}:H^{0}(Q_{C},\mathscr L_{0,C}^{\otimes\kappa _{0}})^{\SL(V_{C})}\rightarrow
H^{0}(T^\circ,\lambda _{\Escr}(u_{n-1}(c,[\mathcal{O}_{H}]))^{\otimes\kappa _{1}}).
\]
As before, if the family $\Escr \rightarrow X\times T^\circ$ is linearized with respect to some $G$-action on $T$, then  the image of $L_{\mathscr E,C}$ lies inside $H^{0}(T^\circ,\lambda _{\Escr}(u_{n-1}(c,[\mathcal{O}_{H}]))^{\otimes\kappa _{1}})^{G}$.

From the above,  we also obtain a commutative diagram 
\begin{equation}\label{eq:comparison_diagram}
\begin{gathered}
\resizebox{.9\hsize}{!}{
\xymatrix{
H^0(M_{C}^{\mu ss},
\Ocal_{M_{C}^{\mu ss}}(\kappa_0 N))\ar@/^2pc/[dr]^{\psi _{\Escr_C}^{\ast }}
\ar[d]^{\psi _{\widehat{\mathscr U}}^{\ast }} &\\ 
H^{0}(Q_{C},\mathscr L_{0,C}^{\otimes\kappa_{0} N N_0})^{\SL(V_{C})} \ar[dr]^{L_{\mathscr E,C}}\ar[d]^ {\bar{L}_{\mathscr U,C}}\ar[r]^>>>>{L_{\mathscr E|_C,C}\quad }& 
H^0(T^\circ,\lambda _{\Escr|_C}(u_{0}(c|_{C},[\mathcal{O}_{H}]_{C}))^{\otimes \kappa_0NN_0}) \ar[d]_{\cong}^{\text{ by \cite[Eq.~(3.5)]{GrebToma:17}}} \\ 
H^{0}(\mathcal{Z},\mathscr L_{n-1}^{\otimes\kappa_{n-1} N_0N})^{\SL(V)}\ar[r]^{L_\Escr\qquad} & H^0(T^\circ,\lambda _{\Escr}(u_{n-1}(c,[\mathcal{O}_{H}]))^{\kappa_{n-1} N_0N})\\
H^0(M^{\mu ss}, \mathcal{O}_{M^{\mu ss}}(\kappa_{n-1} N_0))\ar@/_2pc/[ur]^{\qquad\psi _{\Escr}^{\ast } }
\ar[u]_{\psi _{\widetilde{\mathscr U}}^{\ast }}  &&
 }
 }
 \end{gathered}
\end{equation}
where $N_0$ plays the role of $N$ for the moduli space $M_C^{\mu ss}$ (see eq.\ ~\eqref{eq:functorial_property_of_class_map}). Here we have used that $H^{0}(Q_{C}^{ss},\mathscr L_{0,C}^{\otimes\kappa _{0}})^{\SL(V_{C})}\cong H^{0}(Q_{C},\mathscr L_{0,C}^{\otimes\kappa _{0}})^{\SL(V_{C})}$ in order to construct the map 
$\psi^*_{\widehat{\mathscr{U}}}$ (see \cite[Sect.~8.2]{MR2004511}).

Moreover these maps are compatible with those arising from the intermediate spaces $X^{(l)}$, and in particular from $S$, in the following sense. Indeed, using the fact that $T^\circ$ is also base for the restricted flat family $\Escr_{S\times T^o}$ of semistable sheaves on $S$ and the isomorphism of determinant line bundles proved in \cite[Sect.~3.2.2]{GrebToma:17} we get for the appropriate powers of determinant sheaves a commutative diagram of the form 
\begin{equation}\label{eq:moduli_comparison}
\begin{gathered}
\resizebox{.9\hsize}{!}{
\xymatrix{
H^0(M^{\mu ss},\mathcal{O}_{M^{\mu ss}}(\kappa_{n-1}N_0N_1) \ar[r] &
 H^0(T^\circ,\lambda _{\Escr}(u_{n-1}(c,[\mathcal{O}_{H}]))^{\otimes \kappa_{n-1} N_0N_1 N} ) \\
H^0(M_{C}^{\mu ss},\Ocal_{M_{C}^{\mu ss}}(\kappa_0NN_1))
\ar[r]
\ar[u]\ar[d]&  H^0(T^\circ,\lambda _{\Escr_C}(u_{0}(c|_{C},[\mathcal{O}_{H}]_{C}))^{\otimes \kappa_{0} N N_1 N_0} )  \ar[d]^{\cong} \ar[u]_\cong \\
H^{0}(M_S^{\mu ss},\mathcal{O}_{M_{S}^{\mu ss}}(\kappa_1N_0N))\ar[r]& 
H^0(T^\circ,\lambda _{\Escr_S}(u_{1}(c|_{S},[\mathcal{O}_{H}]_{S}))^{\otimes \kappa_1N_0NN_1}),
}
}
\end{gathered}
\end{equation}
where $N_1$ is defined similarly to $N$ and $N_0$ (see again \eqref{eq:functorial_property_of_class_map}).

In the situation of the proof of Lemma  \ref{lemma:lifting} we may set $T^\circ$ to be equal to $\mathcal{Z} \backslash \mathcal{T} \cup \mathcal{T}_C$ and $\mathscr{E} = \mathscr{U}$. In this case, the map $L_\mathscr{E}$ in \eqref{eq:comparison_diagram} is the restriction map.

\subsection{The 2-dimensional case}\label{subsect:surface_case_separation}
The next proposition deals with moduli spaces of slope-semistable sheaves on surfaces and is implicitly contained in Li's paper \cite{Li:93} although not explicitly stated there.

As remarked at various junctures, the construction of the moduli space $M^{\mu ss}$ proceeds by finding nonvanishing equivariant sections of tensor powers of $\mathscr L_{n-1}$. Proving the separation properties of this moduli space therefore amounts to finding such sections that separate the appropriate quotients (see Section \ref{sec:moduli}). The point of the following proposition is that in the case of surfaces, this may in fact be achieved by only considering the $\theta$-functions that are lifted from a finite number of fixed curves.

\begin{Prop}\label{prop:embedding-via-admissible-curves-surface-case}
Let $(S, \mathcal{O}_S(1))$ be a polarized smooth complex projective surface, $(r,c_1,c_2)\in\Z_{>0}\times \NS(S)\times H^4(S,\Z)$ fixed topological invariants, let $M^{\mu ss}_{S}$ be the moduli space of slope semistable sheaves on $S$ with the given topological invariants and with fixed determinant $\Jcal_S$. Let $A\subset S$ be  finite. Then there exist positive integers $k$ and $l$ $($depending on $k$$)$, and smooth curves $C^{(1)},\ldots,C^{(n)}\in|\mathcal{O}_X(k)|$ not intersecting $A$  such that the linear system \[W_S \subset H^0(M^{\mu ss}_{S}, \mathcal{O}_{M^\muss}(l \kappa_1N_0 ))\] defined as the span of all $\theta$-functions on $M^{\mu ss}_{S}$ of level $l \kappa_1 N_0 N_1$ lifted from the curves $C^{(1)},\ldots,C^{(n_0)}$ gives an injective morphism $$\nu_S:M^{\mu ss}_{S}\hookrightarrow \P(W_S^*)\ .$$
\end{Prop}

\begin{proof} 
Let $\mathcal{Z}$ be the weak normalisation of 
$$R_{S}^{\mu ss}\subset \Quot(\mathcal{H},\tau (r,c_{1},c_{2}))\ ,$$ and consider the family $\{\mathcal{F}_i \}_{i \in \mathcal{Z}}$ of sheaves parametrised by $\mathcal{Z}$. Choose $k\gg0$, $n$, and smooth curves $C^{(1)},\ldots,C^{(n_0)}\in|kH|$ satisfying the conclusion of Proposition \ref{prop: surface sheaf separation}. Note that the curves may be chosen such that they in addition avoid the set $A$. Let $\mathcal{F} = \mathcal{F}_i$ for some $i \in \mathcal{Z}$.  By Proposition~\ref{prop: surface sheaf separation} there exists an index $\alpha \in \{1, \dots, n_0\}$ such that the sheaf $\Fcal|_{C^{(\alpha)}}$  remains semistable and hence by Lemma  \ref{lemma: separation curves} determines a section in $ H^{0}(Q_{C^{(\alpha)}},\mathscr L_{0,C^{(\alpha)}}^{l N_1 N_0\kappa _{0}})^{\SL(V_{C^{(\alpha)}})}$ that does not vanish at $\Fcal|_{C^{(\alpha)}}$. Consequently, by Lemma \ref{lemma:lifting} we obtain a section in $H^{0}(\mathcal{Z},\mathscr L_{n-1}^{l\kappa _{1} N_1 N_0})^{\SL(V)}$ not vanishing at any quotient associated to $\Fcal$. This implies that there exists a positive integer $l$ such that the linear system $W_S$ defined as the span of all $\theta$-functions of level $l\kappa _{1} N_1 N_0$  lifted from the curves $C^{(1)},...,C^{(n_0)}$ has no base points on $M^{\mu ss}_{S}$ and thus gives a morphism $\nu_S:M^{\mu ss}_{S}\to \P(W_S^*)$.

Moreover, it follows from Proposition \ref{prop: surface sheaf separation} and Lemmas \ref{lemma: separation curves} and \ref{lemma:lifting} that, after increasing $l$ if necessary, this morphism separates points of the form $[(\cE_1,\Ccal_1)]$, $[(\cE_2,\Ccal_2)]$ as soon as $\cE_1$ and $\cE_2$ are not isomorphic. Namely, if $\mathcal{E}_{1}=(\Gr\mathcal{F}_{2})^{\vee \vee }$ and $\mathcal{E}_{2}=(\Gr\mathcal{F}_{2})^{\vee \vee }$, then by Proposition \ref{prop: surface sheaf separation} there exists an $\alpha \in \{1, \dots n\}$ such that the restrictions of $\Fcal_1$ and $\Fcal_2$ to $C^{(\alpha)}$ are not s-equivalent, and therefore by Lemma \ref{lemma: separation curves} there is a section of \[H^{0}(Q_{C^{(\alpha)}},\mathscr L_{0,C^{(\alpha)}}^{\otimes l \kappa _{0} N_1N_0})^{\SL(V_{C^{(\alpha)}})}\] that separates them on $C^{(\alpha)}$, and hence there is an induced lifted section in $H^{0}(\mathcal{Z},\mathscr L_{n-1}^{l \kappa _{1} N_1N_0})^{\SL(V)}$ which separates them on $S$ by Lemma \ref{lemma:lifting}.

We thus only need to show that by increasing $l$ again if necessary and by adding $\theta$-functions lifted from further curves $C^{(\alpha)}\in|kH|$ avoiding $A$ to this linear system also points of the form $[(\cE,\Ccal_1)]$, $[(\cE,\Ccal_2)]$ with $\Ccal_1\neq\Ccal_2$ get separated. By a Noetherian induction argument for this it suffices to show that two {\it fixed} distinct points $[(\cE,\Ccal_1)],[(\cE,\Ccal_2)]\in W_{S}$ may be separated by adding $\theta$-functions lifted from further curves as above. This will be done using  \cite[Lemma 3.6]{Li:93} by noticing that its proof actually shows that the added curves may be chosen to avoid $A$.\footnote{See \cite[Lemma 8.3.4]{HuybrechtsLehn:10} for a comparison of the determinant line bundles used. Strictly speaking, Jun Li works in the case of rank two sheaves with zero first Chern class, but his computations remain valid in the general case as well, cf.\ \cite[pp.\ 229/230]{HuybrechtsLehn:10}. The possibility of avoiding $A$ in the general case can also be concluded from the less direct argument for separation presented in \cite[proof of Prop.\ 8.2.13]{HuybrechtsLehn:10} (in the situation discussed there, let $C$ run through all smooth curves avoiding $A$).} Indeed, in order to prove the desired separation, Li uses  $\theta$-functions lifted from two {\it general} members $D_{t_1}$, $D_{t_2}$  of a pencil of curves generated by two elements $D_0, D_1\in |kH|$ chosen as follows (see \cite[p.\ 441]{Li:93}): if $\Ccal_1=\sum_{i=1}^sm'_iP'_i$, $\Ccal_2=\sum_{i=1}^sm''_iP''_i$ and $P=P'_1=P''_1$ is such that $m'_1\neq m''_1$ then it is asked that $P\in D_0$ and $P\notin D_1$, but otherwise $D_0,D_1$ are general in $|kH|$, in particular disjoint from $A\backslash \{ P\}$. Thus general elements $D_{t_1}$, $D_{t_2}$ in this pencil will be disjoint from  $A$.
\end{proof}

\begin{rem}\label{rem:nu_embedded_surface}
In particular, in our global setup, for any smooth surface $S\subset X$ we obtain by restriction
an injective map, still denoted $\nu _{S}$, giving an embedding of $\overline{M}_S^{\mu }\subset M_S^{\muss}$
 into 
$\mathbb{P}(W^{*}_{S})$. Here, $S$ is polarised by $\Ocal_S(1)=\mathcal{O}_X(1)|_S$, the invariants $(r, c_1, c_2)$ are determined by $E|_S$, and the fixed determinant $\mathcal{J}_S$ is given by $\mathcal{J}|_S$.
 \end{rem}

\subsection{The higher dimensional case}\label{subsect:higher-dim_separation}
The following claim deals with the higher dimensional case and  is the main technical statement of this section. It will be proved in the sequel.

\begin{claim}\label{Claim}
For every point $p\in \overline{M}^\mu$ there exist an open $R$-saturated neighbourhood $U\subset \overline M^\mu$,  complex vector spaces $W_j$, $1\le j\le m$ and holomorphic maps $\nu_j:U\to \P(W^{*}_j)$ such that the fibers of 
$$\nu = \nu_1\times...\times\nu_m:U\lra\P(W^{*}_1)\times...\times\P(W^{*}_m)$$
are precisely the equivalence classes of $R$ inside $U$.
\end{claim}

The idea will be to consider different surfaces $S_1, ...,S_m$  arising as last terms of flags of complete intersections in $X$,  which are moreover fully admissible for $p=[(\cE,\Ccal)]$, and their corresponding respective domains of admissibility $U'_j\subset \overline M_{\HYM}$, $U_j\subset \overline M^\mu$, $j=1,\ldots,m$, and to construct morphisms $\nu_j:U_{j}\to\P(V_j)$ induced by restriction to the surfaces $S_j$ in a way to be made precise. Then, we will simply define $\nu_1\times...\times\nu_m$ on  $U=\cap_{j=1}^mU_j$ and show that the collection of $\nu_j$'s has the desired collapsing and separation properties. 

We now describe the construction of the surfaces $S_j$ and of the morphisms $\nu_j$. While we start by explaining how to construct one surface, later we will see why several such surfaces may be needed. We begin by fixing our setup.

\begin{setup}
 Let $p\in \overline M^\mu$ and $\overline{\Phi}(p)=[(\mathcal{E},\mathcal{C})]$ be its image
in $\overline M_\HYM$. Let $d_1, \dots, d_{n-2}$ be a sequence of positive natural number, and  $(X^{(l)})_{1\leq l\leq
n-2}$ be a flag of complete intersections of hypersurfaces $X_{i} \in |\mathcal{O}_X(d_i)|$ such that
\begin{enumerate}
 \item $(X^{(l)})_{1\leq l\leq n-2}$ satisfies  the conclusions of Lemma \ref{flat restriction}
 with respect to $(\mathcal{E},\mathcal{C})$, the universal family $\mathscr U\rightarrow\ X\times\Zcal$, the subvarieties $\Zcal^{(i)}\subset\Zcal$ provided by Lemma \ref{lemma:extension}, and chosen main reference points $z_i$ in $\Zcal^{(i)}$;
 \item the flag satisfies Langer's conditions; i.e., the sequence of integers $d_1, \dots, d_{n-2}$ is sufficiently increasing for a repeated application of Theorem\ \ref{thm:langer}, and the restrictions of the main reference sheaves to $S$ are semistable (cf.~Remark\ \ref{rem:can_arrange_Mehta_Ramanathan}).
\end{enumerate}
\end{setup}

The following will provide us with the building blocks for the $R$-invariant maps we need to construct in order to prove Claim\ \ref{Claim}.
\begin{lemma}\label{lemma:rational maps} In the Setup, if we denote the last term of the flag by $X^{(n-2)}=S\subset X$ and we let $C^{(1)}, \dots, C^{(n_0)}$ be the finite set of curves and $l$ be the natural number provided by Proposition \ref{prop:embedding-via-admissible-curves-surface-case}, then the linear system $W$ of sections of $\mathscr{L}_{n-1}\rightarrow \mathcal{Z}$ provided by lifting theta functions of level $l \kappa_0 N_0N_1N_{n-1}$ from the $C^{(\alpha )}$ to $\mathcal{Z}$ gives a rational map 
$
\eta :\mathcal{Z}\dashrightarrow \mathbb{P}(W_S^*)
$,
into the projective space associated with the linear system $W_S$, which descends to a rational map 
\[
\nu :\overline M^\mu\dashrightarrow \mathbb{P}(W_S^*)
\]
that is well-defined at the images in $\overline M^\mu$ of the points $z_i\in\Zcal$. 
\end{lemma}
\begin{proof}
Since $\mathscr U\rightarrow\ X\times\Zcal$ has the property that $\mathscr U_z$ is torsion free for each $z\in\Zcal$, choosing arbitrary main reference points $z_i\in\Zcal^{(i)}$, we see that the assumptions of Lemma~\ref{flat restriction} are satisfied.  We therefore obtain flatness of $\mathscr U$ on restriction to $S\times\Zcal\backslash\Tcal$, with $\Tcal\subset\Zcal$ the closed subvariety obtained in Lemma \ref{flat restriction}. Moreover, we will assume that the conclusion of Remark~\ref{rem:can_arrange_Mehta_Ramanathan} holds; i.e., the restrictions of the main reference sheaves to $S$ are semistable.

According to Lemma \ref{restriction curves}, for any well chosen curve $C$, we obtain flatness of the restriction of the universal sheaf to  $C\times\Zcal\backslash (\Tcal\cup\Tcal_C)$, where $\Tcal\cup\Tcal_C\subset\Zcal$ is a nonflatness locus which has the property required by Lemma \ref{lemma:extension} for extension of sections. More precisely, this is the case if the curve  $C$ is chosen regular with respect to some "secondary reference" sheaves $\mathscr U_{z^{(n-1)}_i}|_S$. This regularity is guaranteed once the curve  $C$ contains none of the associated points of the sheaves $\mathscr U_{z^{(n-1)}_i}|_S$, which in turn will be the case if $C$ avoids some fixed closed points on $S$, one on each associated component of the sheaves $\mathscr U_{s^{(n-1)}_i}|_S$. We fix $A_1$ to be such a (finite) subset of points of $S$. Moreover, let $A_2$ be the finite set containing the singular points of the Seshadri graduations of the restrictions $\mathscr{U}_{z_i}|_S$ and set $A := A_1 \cup A_2$. 

Let $k$ be a positive integer as guaranteed by Proposition~\ref{prop:embedding-via-admissible-curves-surface-case} and such that the conclusion of Theorem~\ref{thm:langer} holds, and let $C^{(1)}, \dots, C^{(n_0)}$ be the curves provided by Proposition~\ref{prop:embedding-via-admissible-curves-surface-case}.
Let  \[W \subset H^{0}(\mathcal{Z},\mathscr L_{n-1}^{\otimes l \kappa _{n-1}N N_0 N_1 })^{\SL(V)}\]be the span of all theta functions of the appropriate level obtained by lifting sections in $ H^{0}(Q_{C^{(\alpha)}},\mathscr L_{0,C^{(\alpha)}}^{\kappa_0l N N_1 N_0})^{\SL(V_{C^{(\alpha)}})}$ and \[W_X \subset H^0(M^{\mu ss}, \mathcal{O}_{M^{\mu ss}}(\kappa_{n-1}lN_1 N_0  ))\] the corresponding linear system on $M^\muss$; see the discussion in Section~\ref{subsect:lifting_II}, especially Diagrams~\eqref{eq:comparison_diagram} and \eqref{eq:moduli_comparison}, which also shows that naturally 
\begin{equation}\label{eq:same}
   \mathbb{P}(W^*) \cong \mathbb{P}(W_X^*) \cong \mathbb{P}(W_S^*)\ .
   \end{equation}
Note that these linear systems are nontrivial, since the flag  $(X^{(l)})_{1\leq l\leq n-2}$ and the curves $C^{(\alpha)}$ have been chosen so that for each $i$, the sheaves $\mathscr U_{z_i}$ remain semistable on $C^{(\alpha)}$.
By the natural identifications listed in \eqref{eq:same}, the linear system $W$ gives rise to a rational map $\eta :\Zcal\dashrightarrow \P(W_S^*)$ defined at the $z_i$, which descends to the rational map $\nu:M^\muss \dashrightarrow\P(W^*_S)$ associated with $W_X$, and hence by restriction yields a rational map $\nu : \overline{M}^\mu \dashrightarrow \P(W^*_S)$. By construction, $\nu$ is defined at the images of the $z_i$, as desired.
\end{proof}

 The following proposition gives the required collapsing property for the map $\nu$. When combined with Corollary \ref{cor:admissible_plus_lifting}, it will also imply the required separation property regarding "the double-dual component" of elements of the gauge theoretic moduli space.  

\begin{prop}\label{prop:nonseparation} 
The rational map  $\nu:\overline M^\mu\dashrightarrow\P(W^*_S)$ constructed in Lemma \ref{lemma:rational maps} is well-defined and hence holomorphic on the open set $U_{S}$ defined in Proposition \ref{prop:transversality-on-close-classes} above. 

Moreover, given any polystable sheaf $\cF$ with $[\cF] \in U_S$ there is a polystable sheaf $\Fcal_{\Vert S}$ on $S$ such that the map
 $$
 r_{U_{S}}:U_{S}\rightarrow\overline{M}_S^{\mu
} :   
[\Fcal]\mapsto[\Fcal_{\Vert S}]\ ,
$$
is well-defined and extends the natural rational map $U_{S}\dashrightarrow\overline{M}_S^{\mu}$ given by restricting locally free stable sheaves to $S$. This map makes the following into a commutative diagram, where $\nu_S$ is as defined in Proposition~\ref{prop:embedding-via-admissible-curves-surface-case} and Remark\ \ref{rem:nu_embedded_surface}:
\begin{equation}\label{eq:nu_nu_S_diagram}
 \begin{gathered}\xymatrix{
 &&    \overline{M}_S^{\mu
}  \ar[dll]_{\nu_{S}}\\
  \mathbb{P}(W_{S}^{\ast })& &
 U_S \ .\ar[u]_{r_{U_{S}}}   \ar[ll]^{\nu}
 }
 \end{gathered}
\end{equation}
The map $r_{U_S}$ is $R$-invariant. In particular, $\nu: U_S \to \P(W^*_S)$ is $R$-invariant.
\end{prop}
\begin{proof}
Let $p\in U_{S}$ be an arbitrary point, let $\cF$ be a polystable representative of $p$ and set $\Ecal := \cF^{\vee \vee}$. 

Recall from Definition \ref{defi:cycle,couple} that $\gamma (\cF) = (\ddual \cF, \mathcal{C}_{\cF})$, where $\mathcal{C}_{\cF}$ is the support cycle of the sheaf $ \Qcal_{\cF} = \ddual\cF / \cF$. As also introduced there, we set $\widehat \Qcal_{\cF} = \Qcal_{\cF}/ \cT (\Qcal_{\cF})$. 
Note that the associated points of  $\widehat \Qcal_\cF$ correspond to the irreducible components of $C_\cF$. The admissibility of the surface $S$ for $\overline{\Phi}([\cF])$ hence implies that the hypersurface $X'$ is regular for $\widehat \Qcal_\cF$. By  \cite[Cor.\ 1.1.14]{HuybrechtsLehn:10} the restriction $\cE|_{X'}$ is torsion free on $X'$. 
Using this as well as the long exact sequence induced by restricting the defining sequence of $\Qcal_\cF$ to $X'$, we obtain an inclusion $\supp(\Tors_{\cO_{X'}}(\cF|_{X'}))\subset \supp(\cT or_1^{\cO_X}( \Qcal_\cF, \cO_{X'}))$. 
As $X'$ does not contain any $2$-codimensional associated point of $\Qcal_\cF$, the latter subvariety is at least three codimensional in $X$, from which we conclude that the torsion of $\cF|_{X'}$ on $X'$ will be supported in codimension at least two on $X'$. Thus $\gamma(\cF|_{X'})$ has a meaning according to Definition~\ref{defi:cycle,couple}. Setting\footnote{The definition \eqref{eq:gamma_restriction} extends the one given in Remark \ref{rem:good_restriction} in the fully admissible case.}

\begin{equation}\label{eq:gamma_restriction}
\gamma(\cF)|_{X'}:=(\ddual{(\cE|_{X'})},(\mathcal{C}_{\cF})|_{X'}+\mathcal{C}_{(\cE|_{X'})})\ ,
\end{equation}
we claim that
 \begin{equation}\label{eq:gamma_after_first_restriction}
[\gamma(\cF)|_{X'}] = [\gamma(\cF|_{X'})]\ .
\end{equation}

As $\cF$ and $\Ecal$ coincide outside of a subvariey of codimension two in $X$ whose codimension two part is the support of $Q_{\cF}$, and as $X'$ intersects this support in codimension two, equality in the sheaf component follows from reflexivity.

In order to prove equality in the cycle component, note first that $(\cC_\cF)|_{X'}=(\cC_{\Qcal_\cF})|_{X'}=(\cC_{\widehat \Qcal_\cF})|_{X'}$, and second that by Lemma \ref{lem:restricted_cycles_cycles_of_restrictions} the last cycle is equal to $\cC_{(\widehat \Qcal_\cF)|_{X'} }$, which is computed on $X'$. We thus have to show that \[\cC_{(\widehat \Qcal_\cF)|_{X'} }=\Ccal((\widehat \Qcal_\cF)|_{X'}) =-\mathcal{C}(\cF|_{X'})-\mathcal{C}_{(\cE|_{X'})}\ .\] Applying  Lemma \ref{lemma:cycles2} to the morphism $\cF|_{X'}\to\cE|_{X'}$ one gets
$-\mathcal{C}(\cF|_{X'})-\mathcal{C}_{(\cE|_{X'})}=\cC((\Qcal_\cF)|_{X'})-\cC(\cT or_1^{\cO_X}( \Qcal_\cF, \cO_{X'}))$, so it will be enough to check that 
\[\cC({(\widehat \Qcal_\cF)|_{X'} })=\cC((\Qcal_\cF)|_{X'})-\cC(\cT or_1^{\cO_X}( \Qcal_\cF, \cO_{X'}))\ .\] For this, we consider the following commutative diagram with exact rows and columns
\begin{equation}
\label{diagr: on S}
\begin{gathered}
\resizebox{.85\hsize}{!}{
\xymatrix{ 
 &  0 \ar[d] &   0 \ar[d]& & 
 \\ 
0 \ar[r]& \cT or_1^{\cO_X}(\cT (\Qcal_{\cF}),\cO_{X'}) \ar[r]\ar[d]& \cT or_1^{\cO_X}(\Qcal_{\cF},\cO_{X'}) \ar[r]\ar[d] & 0\ar[d] &  
\\
0 \ar[r]&
\cT(\Qcal_{\cF})({{-X'}})  \ar[r]\ar[d] & \Qcal_{\cF}({{-X'}})  \ar[d]\ar[r] & \widehat \Qcal_{\cF}({{-X'}}) \ar[r] \ar[d] & 0 
\\
0 \ar[r] & \cT(\Qcal_{\cF})  \ar[r]\ar[d] & \Qcal_{\cF}  \ar[d]\ar[r] & \widehat Q_{\cF}
 \ar[r]\ar[d] & 0 
 \\
0 \ar[r] & \cT(\Qcal_{\cF})|_{{X'}}  \ar[r]\ar[d] &\Qcal_{\cF}|_{{X'}} \ar[r] \ar[d] & \widehat \Qcal_{\cF}|_{{X'}} \ar[r] \ar[d] & 0
\\
 & 0 &0 &\ 0, & 
}
}
\end{gathered}
\end{equation}
which incorporates the fact that $\cT or_k^{\cO_X}( \widehat Q_\cF, \cO_{X'})=0$ for $k=1,2$ due to the regularity noticed at the beginning of the argument. Using again that $ \cT or_1^{\cO_X}( \Qcal_\cF, \cO_{X'})$ has codimension at least three in $X$, a computation with the help of Lemma \ref{lemma:cycles1} yields the desired equality on the cycle level and finishes the proof of \eqref{eq:gamma_after_first_restriction}. 
 
Note that the above considerations, the admissibility of $S$ for $\gamma(\cF)$, and the fact that $\supp( \mathcal{C}_{(\cE|_{X'})})\subset \sing(\cE)\cap X'$ together imply that $S$ is admissible for $\gamma(\cF)|_{X'}$ as well.

We will next prove the following claim, from which all the statements of the proposition will follow relatively quickly. 

\begin{claim}\label{claim:gamma}
In the situation of Lemma \ref{lemma:rational maps}, if $\cF$ is a polystable sheaf with $[\cF] \in U_S$, then on each $X^{(l)}$ there exist 
a polystable sheaf $\cF^{(l)}$ 
such that $S$ is admissible for $\gamma( \cF^{(l)})$ and such that for $0\le l\le n-3$ we have 
\begin{equation}\label{eq:gamma_after_restriction}
[\gamma(\cF^{(l)})|_{X^{(l+1)}}] = 
[\gamma(\cF^{(l)}|_{X^{(l+1)}})] = [\gamma(\cF^{(l+1)})]\ ,
\end{equation} 
where $\cF^{(0)}:=\cF$, and where $\gamma(\cF^{(l)})$ is defined as in \eqref{eq:gamma_restriction}.
\end{claim} 
\begin{proof}[Proof  of Claim\ \ref{claim:gamma}]
We start by constructing the sheaf $\cF'$ on $X'$. If the restriction $\cF|_{X'}$ is torsion free, we just put $\cF':=\cF|_{X'}$, since it is then polystable by the assumptions of Lemma \ref{lemma:rational maps} and Langer's Theorem, Theorem \ref{thm:langer}, as we noticed in the first paragraph of this proof. 
The other requirements on $\cF'$ are also satisfied by the discussion preceding the Claim.

If $\cF|_{X'}$ is not torsion free, we will replace $\cF|_{X'}$ by a torsion free sheaf $\cF'$ on $X'$ such that $[\gamma(\cF')] = [\gamma(\cF|_{X'})]$. 
For this, we recall from Definition \ref{def:defining_the_closure} that $\overline M^\mu$ is the closure of $M^s$ inside $M^{\muss}$. It follows that $[\cF]$ is the limit of points represented by locally free, stable sheaves and that $\cF$ may be realized as the central fiber of a one dimensional subfamily of $\Zcal$ all of whose other fibers are locally free and stable (cf.\ the discussion at the end of Section\ \ref{sec:map}). The restriction of this family to $X'$ remains flat since $\cF$ is torsion free.
By Langton's Theorem the special fiber $\cF|_{{X'}}$ of the restricted subfamily may be replaced by a sheaf $\cF'$ which is semistable on ${X'}$. 
Following the proof of that theorem presented in \cite[Appendix 2B]{HuybrechtsLehn:10} we will check that  $\gamma(\cF')=\gamma(\cF|_{X'})$ on $X'$. In our case it will turn out to be in fact sufficient to follow the proof up to the point where one finds a torsion free central fiber $\cF'$. For this we will proceed along the way sketched in \cite[Exercise 2B2]{HuybrechtsLehn:10}. 

We will write $\mathscr F$ for a flat family of sheaves on ${X'}$  parameterized by the unit disk $\Delta\subset\C$. We suppose that all fibers over $\Delta^*=\Delta\backslash\{ 0\}$ of $\mathscr F$ are semistable sheaves, that the fiber $\mathscr F_0$ has torsion 
and that $\mathscr F_0/\Tors(\mathscr F_0)$ is polystable. 
We consider the torsion filtration $ 0\subset \mathscr T_0 (\mathscr F_0) \subset \cdots\subset \mathscr T_{n-1}(\mathscr F_0)=\mathscr F_0$ of $\mathscr F_0$ as in 
\cite[Def.\ 1.1.4]{HuybrechtsLehn:10}. The idea is to replace the central fiber $\mathscr F_0$ successively by sheaves having no zero-dimensional torsion to begin with, then no one-dimensional torsion, and so on. So supposing that $\mathscr T_0(\mathscr F_0)\neq0$, set $\mathscr F^1$ to be the kernel of the composition of the natural projections $\mathscr F\to \mathscr F_0\to \mathscr F_0/\mathscr T_0(\mathscr F_0)$. One gets two exact sequences 
$$0\lra \mathscr T_0(\mathscr F_0)\lra \mathscr F_0\lra \mathscr F_0/\mathscr T_0(\mathscr F_0)\lra 0,$$
$$0\lra \mathscr F_0/\mathscr T_0(\mathscr F_0)\lra \mathscr F^1_0\lra \mathscr T_0(\mathscr F_0)\lra 0$$
from which one infers with the help of Lemma \ref{lemma:cycles2} that
\begin{equation} \label{eq:gamma_agrees_after_Langton}
\ddual{\mathscr F_0}=\ddual{(\mathscr F_0^1)} \quad \text{ and } \quad
\Ccal(\mathscr F_0^1)=\Ccal(\mathscr F_0)\ .
\end{equation}If $\mathscr F^1_0$ continues to have zero-dimensional torsion we construct the sheaf $\mathscr F^2$ starting from $\mathscr F^1$ and so on until we obtain some $\mathscr F^m_0$ which has no torsion. The argument in \cite[Appendix 2B]{HuybrechtsLehn:10} ensures that this happens for some $m\in\N$. By \eqref{eq:gamma_agrees_after_Langton} above we have $[\gamma(\ddual{\mathscr F_0})] = [\gamma(\mathscr F_0^m)]$. We then look at one-dimensional torsion and so on. Eventually we obtain a subsheaf $\mathscr F'\subset \mathscr F$ coinciding with $\mathscr F$ over $\Delta^*$ and such that its central fiber $\mathscr F'_0$ is torsion free with 
$[\gamma(\mathscr F_0)] = [\gamma(\mathscr F'_0)]$. Hence, if we set $\cF':=\mathscr F_0'$, then from this equality together with \eqref{eq:gamma_after_first_restriction} we infer that
\begin{equation}\label{gammas_are_equal}
 [\gamma(\cF')] = [\gamma(\cF|_{X'})] = [\gamma(\Fcal)|_{X'}]\ .
\end{equation}
Recall from the observation made right before Claim\ \ref{claim:gamma} that $S$ is admissible for $\gamma(\cF')$. Moreover, $\cF'$ is polystable, since $\ddual{(\cF')}\cong\ddual{(\ddual{\cF}|_{X'})}$ by \eqref{eq:gamma_agrees_after_Langton}. Thus, the sheaf $\cF'$ fulfills the conditions of Claim\ \ref{claim:gamma} at level $l=0$.

The sheaves $\cF^{(2)}$,\ldots,$\cF^{(n-2)}$ are now constructed inductively by the same procedure, using the above family over the disk. 
\end{proof}

Now we use Claim\ \ref{claim:gamma} to prove Proposition\ \ref{prop:nonseparation}. We again work in the setup of Lemma \ref{lemma:rational maps}. 
 
As a first observation, note that if a point  $p$  admits a locally free polystable representative $\cF$, then for all $\alpha$ the restriction $\cF|_{C^{(\alpha)}}$ remains semistable on $C^{(\alpha)}$. Hence, corresponding sections in $\mathscr{L}_{n-1}$ will exist that do not vanish at $p$, and therefore $\nu$ is defined at such points. For such points
we set $\Fcal_{\Vert S}:=\Fcal\bigr|_S$, and 
  the equality 
\begin{equation}\label{eq:desired_equality}
 \nu([\Fcal]) = \nu_S([\Fcal_{\Vert S}])
\end{equation}
holds by definition of the relevant linear systems and the compatibility of lifting maps discussed at the end of Section \ref{subsect:lifting_II} (see the proof of Lemma \ref{lemma:rational maps}, especially the isomorphism \eqref{eq:same}). 
In fact, the above holds more generally for points $p$ admitting a polystable representative $\cF$ such that  the flag is fully admissible for $\cF$,  since then $\cF|_{X^{(l)}}$ will be semistable on all the $X^{(l)}$ by the assumptions made in the formulation of Lemma \ref{lemma:rational maps} and Langer's Theorem, Theorem \ref{thm:langer}. 

Now let $p=[\Fcal]\in U_S$ be arbitrary.
 We define the sheaf $\Fcal_{\Vert S}$ on $S$ to be the sheaf $\Fcal^{(n-2)}$ produced by Claim \ref{claim:gamma} in the case $l=n-2$.
Then $p=[\Fcal]$ is a limit 
\begin{equation}\label{eq:approximate_p}
 p = [\Fcal] = \lim_{i \to \infty }[\Fcal_{i}]
\end{equation}
of points $[\Fcal_{i}] \in \overline{M}^\mu$ with $\Fcal_{i}$ locally free and $\mu$-stable, and such that
\begin{equation}\label{eq:limit_comparison}
\lim_{i\rightarrow \infty }[\mathcal{F}_{i}|_{S}]=[\mathcal{F}_{\Vert S}]\in 
\overline{M}_{S}^{\mu }\ .
\end{equation}
As $\Fcal_{\Vert S}$ is polystable, there exists an element of $W_S$ that does not vanish at $[\Fcal_{\Vert S}]$. Using eqs.\ \eqref{eq:approximate_p} and \eqref{eq:limit_comparison}, together with the observation made in the first paragraph of the proof, especially  \eqref{eq:desired_equality}, and continuity of theta functions we see that the corresponding element in $W_X$ does not vanish at $p$. The map $\nu$ is therefore well-defined and hence holomorphic at $p$, as claimed. Moreover, again by continuity we conclude that 
 \begin{equation*}\label{eq:first_part} \nu(p) = \nu([\Fcal]) = \nu_S([\Fcal_{\Vert S}])\ ,\end{equation*} which proves commutativity of Diagram\ \eqref{eq:nu_nu_S_diagram}, once we have established that $r_{U_S}$ is well-defined. 

To see that $r_{U_{S}}$ is well-defined, let $\Fcal_1$ and $\Fcal_2$ be polystable such that $[\mathcal{F}_{1}]=[\mathcal{F}_{2}]\in U_{S}$. By Proposition\ \ref{prop:GT} and Remark\ \ref{rem:gamma_campatible_with_earlier_defi} we then get $[\gamma(\Fcal_1)] = [\gamma(\Fcal_2)]$, 
which together with an recursive application of eq.\ \eqref{eq:gamma_after_restriction} implies 
\begin{equation*}
[\gamma (\mathcal{F}_{1}^{(n-2)})] = [\gamma (\mathcal{F}_{2}^{(n-2)})].
\end{equation*}
We therefore obtain from the surface case of  Proposition \ref{prop:GT} that $\lbrack \mathcal{F}_{1}^{(n-2)}]=[\mathcal{F}_{2}^{(n-2)}]$, as needed.

To see that  $r_{U_{S}}$ is $R$-invariant, note that if $\Fcal_1$, $\Fcal_2$ are polystable such that $[\mathcal{F}_{1}],[\mathcal{F}_{2}]\in U_{S}$ are in the same $R$-equivalence class, then by definition this means that $[\gamma(\mathcal{F}_{1})] = [\gamma(\mathcal{F}_{2})]$ (see the remark made after eq.\ \eqref{eq:relation_phi_psi}). Hence, the same reasoning as in the previous paragraph implies that 
\begin{equation*}
r_{U_{S}}([\mathcal{F}_{1}])=[\mathcal{F}_{1}^{(n-2)}]=[\mathcal{F}_{2}^{(n-2)}]=r_{U_{S}}([\mathcal{F}_{2}]).
\end{equation*}
The $R$-invariance of $\nu$ now follows by the commutativity of Diagram\ \eqref{eq:nu_nu_S_diagram}. This concludes the proof of Proposition\ \ref{prop:nonseparation}. 
\end{proof}

To complete the proof of the main claim, we will have need of one further lemma, which says that the cycle components $\Ccal$ are determined completely by their restrictions to appropriately chosen complete intersection surfaces.
Its proof, omitted here, follows easily by Noetherian induction.
\begin{lemma}\label{lemma: cycle separation} Let $X$ be a smooth projective variety and $\mathcal{O}_X(1)$ be an ample line bundle. Assume that $\mathcal{O}_X(k)$ is very ample. Let $\Ccal_0$ be an element in the variety $\mathscr C:=\mathscr C_{n-2,d}(X)$ of  codimension 2 cycles of degree $d$ on $X$ and for any surface in $X$ denote by $U''_S\subset \mathscr C$ the open subset of cycles meeting $S$ in at most finitely many points. Then for any fixed multi-degree divisible by $k$ there exists a positive natural number $m$ such that the following holds for every sufficiently general choice of $m$ smooth complete intersection surfaces $S_1,\ldots,S_m$ of the given multi-degree in $X$:
\begin{enumerate}
 \item  $\Ccal_0\in \cap_{j=1}^m U''_{S_j}$, and 
 \item  for all $\Ccal_1, \Ccal_2\in \cap_{j=1}^m U''_{S_j}$ we have $\Ccal_1 = \Ccal_2$ if and only if $\Ccal_1\cap S_j =  \Ccal_2\cap S_j$ for all $j=1,...,m$.
\end{enumerate}\end{lemma}

After these preparations, we are now in the position to give the proof of the main technical result of this section, Claim \ref{Claim}.

 \begin{proof}[Proof of Claim \ref{Claim}]
Fix a point $p \in \overline{M}^\mu$ and let $U_{S}$ be one of the neighbourhoods of $p$ 
provided by Proposition \ref{prop:transversality-on-close-classes}. We will write $p=[\mathcal{F]}$, for $\Fcal$ a polystable sheaf, and $\overline{\Phi 
}([\mathcal{F}])=[(\mathcal{E},\mathcal{C})]$. 

Choose a sequence $0 \ll d_1 \ll d_2 \ll \dots \ll d_{n-2}$ of positive integers that is sufficiently increasing such that a repeated application of Langer's Theorem, Theorem~\ref{thm:langer}, and an application of Corollary\ \ref{cor:admissible_plus_lifting} is possible. Let $N$ be the natural number provided by Corollary\ \ref{cor:admissible_plus_lifting} and assume without loss of generality that it coincides with the maximal number $m$ provided by applying Lemma\ \ref{lemma: cycle separation} to the cycle spaces $\mathscr C_{n-2,d}(X)$, where $d$ runs through the finite set of degrees of cycles $\mathcal{C}$ forming the second entry in a pair $[(\mathcal{E}, \mathcal{C})] \in \overline{M}_{\HYM}$ (cf.\ Section\ \ref{sec:boundedness}). Choose $m$ (general) flags satisfying the assumptions made in the Setup and such that furthermore the conclusion of Corollary\ \ref{cor:admissible_plus_lifting} holds. Denote the chosen set of flags by $\Sigma$. Denote by $S_1, \dots, S_m$ the end terms of the flags contained in $\Sigma$, and by $\nu_j: U_{S_j} \to \mathbb{P}(W^{*}_j) := \mathbb{P}(W_{S_j}^*)$ and $\nu_{S_j}:  \overline M^{\mu}_{S_j}\hookrightarrow \P(W_j^*)$ the maps provided by Lemma \ref{lemma:rational maps} and Proposition~\ref{prop:embedding-via-admissible-curves-surface-case}, respectively; see also Proposition \ref{prop:nonseparation}. Set $U:= U_{S_1} \cap \dots \cap U_{S_m}$ and let
$$\nu:= \nu _{1}\times \cdots
\times \nu _{m}:U \longrightarrow \mathbb{P}(W^{*}_{1})\times \cdots \times \mathbb{P}(W^{*}_{m})$$
be the product map, which is $R$-invariant by Proposition \ref{prop:nonseparation}. It therefore remains to show that $\nu$ separates different $R$-equivalence classes.

First, let $\mathcal{F}_{1}$, $\mathcal{F}_{2}$ be two polystable sheaves representing points $[\mathcal{F}_1], [\mathcal{F}_2] \in \overline{M}^\mu$ such that $\mathcal{E}_1 := \ddual{\Gr(\mathcal{F}_{1})}$ and $\mathcal{E}_2 :=\ddual{\Gr(\mathcal{F}_{2})}$ are not isomorphic. Then, there exists a flag $(X^{(l)})$ in $\Sigma$ that is fully admissible for both $[\gamma(\mathcal{F}_1)]= [(\mathcal{E}_1, \mathcal{C}_1)]$ and $[\gamma(\mathcal{F}_2)] = [(\mathcal{E}_2, \mathcal{C}_2)]$. Let $S_{j} \subset X$ be the corresponding surface. If we let $r_{U_{S_j}}:U_{S_j}\rightarrow\overline M^{\mu}_{S_j}$ be the map produced by Proposition\ \ref{prop:nonseparation}, the $j$-th component of $\nu$ is equal to 
\begin{equation} \label{eq:jthcomponent}
 \nu_j([\mathcal{F}_i])=\nu_{S_j}(r_{U_{S_j}}([\cF_{i}])) \quad \quad \text{ for }i = 1,2.
\end{equation}
Since $\lbrack\gamma(\Fcal_{1}^{(n-2)})]=[\gamma (\mathcal{F}_{1})|_{S}]\neq \lbrack \gamma (\mathcal{F}_{2})|_{S}]=[\gamma
(\Fcal_{2}^{(n-2)})]$, 
it follows from Corollary\ \ref{cor:admissible_plus_lifting} and Claim\ \ref{claim:gamma} that $r_{U_{S_j}}([\cF{_1}])=[\Fcal_{1}^{(n-2)}]$ and $r_{U_{S_j}}([\cF_2])=[\Fcal_{2}^{(n-2)}]$ define different points in the moduli space $\overline M^{\mu}_{S_j}$ (cf.\ Remark\ \ref{rem:identify_in_surface_case}).
 From the defining property of $\nu_{S_j}$ (see Proposition\ \ref{prop:embedding-via-admissible-curves-surface-case}) and from eq.\ \eqref{eq:jthcomponent} above we hence conclude that $\nu([\mathcal{F}_1]) \neq \nu([\mathcal{F}_2])$. 

Second, write $[\mathcal{F}_{1}]$ and $[\mathcal{F}_{2}]$ as above for two points of $\overline M^\mu$, where $\Fcal_1$ and $\Fcal_2$ are polystable, and let $[\gamma(\mathcal{F}_1)]=[(\mathcal{E}_{1},
\mathcal{C}_{1})]$, $[\gamma(\mathcal{F}_2)]=$ $[(\mathcal{E}_{2},
 \mathcal{C}_{2})]$ be such that $\mathcal{C}_1 \neq \mathcal{C}_2$. Note that by the definition of admissibility for any $[\mathcal{E}'] \in U$ with associated point $\overline{\Phi}([\mathcal{E}']) = [(\mathcal{E}, \mathcal{C})]\in \overline{M}_{\HYM}$, we have $\mathcal{C} \in U''_{S_j}$ for all $j \in \{1, \dots, m\}$. By the choice of $m$ made above there exists an index $j \in \{1, \dots, m\}$ such that $\mathcal{C}_1|_{S_j} \neq \mathcal{C}_2|_{S_j}$, and therefore $[(\Ecal_1,\Ccal_1)|_{S_{j}}]\neq [(\Ecal_2,\Ccal_2)|_{S_{j}}]$. Applying Proposition \ref{prop:nonseparation} and using injectivity of $\nu _{S_{j}}$  again, similar to the first step we obtain 
$$\nu _{j}([\mathcal{F}_{1}])=\nu_{S_j}(r_{U_{S_j}}([\cF_{1}]))\neq 
\nu_{S_j}(r_{U_{S_j}}([\cF_{2}]))=\nu _{j}([\mathcal{F}
_{2}])\ .$$
This concludes the proof of Claim\ \ref{Claim}.
\end{proof}

\begin{proof}[\protect{Proof of Theorem\ \ref{thm:mainII}}]
 The existence of a complex structure on $\overline M_\HYM$ making the map $\overline{\Phi}:\overline{M}^\mu \to \overline M_\HYM$ holomorphic follows from Cartan's criterion, Theorem\ \ref{thm:Cartan_criterion}, together with Claim\ \ref{Claim}. 
 
To prove the other statements, recall that we endowed $M^\ast_\HYM$ with the complex structure such that $\Phi: M^{s} \to M^\ast_\HYM$ becomes biholomorphic. Consequently, the induced map $\Phi^{wn}: (M^{s})^{wn} \to (M_\HYM^\ast)^{wn}$ of weak normalisations is likewise biholomorphic. From the discussion in the paragraph before Definition\ \ref{def:defining_the_closure} we hence conclude that we have a natural map $\iota: (M_\HYM^\ast)^{wn} \to \overline M_\HYM$ fitting into the following commutative diagram:
\[\begin{xymatrix}{
  (M^{s})^{wn} \ar@{^(->}[r] \ar[d]_{\Phi^{wn}}^{\cong} & \overline{M}^\mu \ar[d]^{\overline{\Phi}}\\
  (M_\HYM^\ast)^{wn} \ar[r]^\iota & \overline M_\HYM\ .
}
  \end{xymatrix}
\]
As the image of $(M^{s})^{wn}$ is a $\overline{\Phi}$-saturated subset on which the equivalence relation $R$ built from $\Phi$ is trivial, $\iota$ is an open embedding, whose complement is the image of ${\overline M}^\mu \backslash (M^{s})^{wn}$ under $\overline{\Phi}$, and hence Zariski closed.

The natural map $(\overline M^{\GMC})^{wn}\to\overline M^\mu$ (see \eqref{eq:Xi}) and its birationality have already been discussed at the end of Section\ \ref{sec:moduli}.
 \end{proof}

\begin{proof}[Proof of Corollary\ \ref{cor:algebraic}]
 Since $\overline{\Phi}$ is finite and surjective, and since $\overline{M}^\mu$ is projective and therefore in particular Moishezon, the complex space $\overline M_\HYM$ is likewise Moishezon (see \cite[Chap.\ V., Cor.\ 11]{AnconaTomassini:82},  but note that the proof is much easier for finite maps than for the general case). It hence follows from a result of Artin (see \cite[Thm.\ 7.3]{Artin:70}) that $\overline M_\HYM$ is the analytification of a proper algebraic space. The generalisation of Serre's GAGA to holomorphic maps between analytifications of proper algebraic spaces (see again \emph{loc.\ cit.}), then implies that $\overline{\Phi}$ is (induced by) a morphism of algebraic spaces. The final statement follows from the corresponding statement in the analytic category, which is contained in Theorem\ \ref{thm:mainII}, and the fact that the analytification of the algebraic weak normalisation of $M^s$ is naturally biholomorphic to the weak normalisation of $M^s$ as a complex space (cf.\ \cite[Sect.\  2.3]{GrebToma:17}).
\end{proof}


\vspace{0.1cm}

\providecommand{\MR}[1]{}
\providecommand{\bysame}{\leavevmode\hbox to3em{\hrulefill}\thinspace}
\providecommand{\MR}{\relax\ifhmode\unskip\space\fi MR }
\providecommand{\MRhref}[2]{%
  \href{http://www.ams.org/mathscinet-getitem?mr=#1}{#2}
}
\providecommand{\href}[2]{#2}

\vspace{0.6cm}

\begin{center}
---------------------------------------
\end{center}
\vspace{0.6cm}

\end{document}